\documentclass[11pt, reqno]{amsart}
\usepackage{xr}
\usepackage{amsthm,amssymb,amstext,amscd,amsfonts,soul}
\usepackage{amsbsy,amsxtra,latexsym,url,multirow}
\usepackage{amsmath,color,comment}
\usepackage{fancybox}
\usepackage{fullpage}
\usepackage[english]{babel}
\usepackage[latin1]{inputenc}
\usepackage{float}\restylefloat{table}
\usepackage{bm}
\usepackage{enumitem}
\usepackage[draft]{hyperref}
\usepackage{cancel,ulem}

\newcommand{\field}[1]{\mathbb{#1}}
\newcommand{\N}{\field{N}}
\newcommand{\Z}{\field{Z}}
\newcommand{\R}{\field{R}}

\newcommand{\Q}{\field{Q}}

\newcommand{\SL}{\operatorname{SL}}

\def \a{\alpha}

\newcommand{\bea}{\begin{eqnarray}}
\newcommand{\eea}{\end{eqnarray}}
\newcommand{\be}{\begin {equation}}
\newcommand{\ee}{\end{equation}}

\newcommand{\n}{\frak n}

\newcommand{\la}{\langle}
\newcommand{\ra}{\rangle}

\newcommand{\ord}{\operatorname{ord}}

\newcommand{\gen}{\text{gen}}
\newcommand{\legendre}[2]{\left( \frac{#1}{#2} \right)}
\newcommand{\pfrac}[2]{\left( \frac{#1}{#2} \right)}

\newif \ifdetails 

\detailsfalse

\newif\ifdiscs 

\discsfalse

\newenvironment{extradetails}{\ifdetails \noindent  \newline \noindent \bf
************************** Begin Extra Details **************************
\rm \\ \bf\fi } {\ifdetails \noindent \newline
\noindent \bf  ************************** End Extra Details ***************************\rm\\ \fi }

\ifdetails
\else
\excludecomment{extradetails}
\excludecomment{extradetailsproof}
\fi

\ifdiscs
\else
\excludecomment{discussion}
\fi

\numberwithin{equation}{section}
\numberwithin{table}{section}
\newtheorem{theorem}{\textbf{Theorem}}
\numberwithin{theorem}{section}
\newtheorem{lemma}[theorem]{\textbf{Lemma}}
\newtheorem{proposition}[theorem]{\textbf{Proposition}}
\newtheorem{remark}[theorem]{Remark}

\newtheorem{corollary}[theorem]{\textbf{Corollary}}

\theoremstyle{remark}
\newtheorem{example}[theorem]{\textbf{Example}}

\renewcommand{\pmod}[1]{\, \left(  \mathrm{mod} \,  #1 \right)}

\begin{document}
\title{Explicit Class number formulas for Siegel--Weil averages of ternary quadratic forms}
\author{Ben Kane}
\address{Department of Mathematics, University of Hong Kong, Pokfulam, Hong Kong}
\email{bkane@hku.hk}
\author{Daejun Kim}
\address{School of Mathematics, Korea Institute for Advanced Study, Seoul 02455, Korea}
\email{dkim01@kias.re.kr}
\author{Srimathi Varadharajan}
\address{Department of Mathematics, University of Hong Kong, Pokfulam, Hong Kong}
\email{srimathi1388@gmail.com}
\thanks{The research of the first author was supported by grants from the Research Grants Council of the Hong Kong SAR, China (project numbers HKU 17301317 and 17303618). This work of the second author was supported by Basic Science Research Program through NRF funded by the Minister of Education (NRF-2020R1A6A3A03037816) while his staying at HKU as an honorary research associate, and by a KIAS Indivisual Grant (MG085501) at Korea Institute for Advanced Study.}
\date{\today}
\begin{abstract}
In this paper, we investigate the interplay between positive-definite integral ternary quadratic forms and class numbers. We generalize a result of Jones relating the theta function for the genus of a quadratic form to the Hurwitz class numbers, obtaining an asymptotic formula (with a main term and error term away from finitely many bad square classes $t_j\Z^2$) relating the number of lattices points in a quadratic space of a given norm with a sum of class numbers related to that norm and the squarefree part of the discriminant of the quadratic form on this lattice. 
\end{abstract}
\maketitle

\section{Introduction and statement of results}
Sums of squares and representations of integers by quadratic forms have a long and storied history. Consider, for example, the Euclidean norm in three-dimensional space, which for a vector $\bm{x}=(x_1,x_2,x_3)\in\R^3$ is given by 
\[
\|\bm{x}\|^2:=\sum_{j=1}^3 x_j^2.
\]
If one restricts to the lattice $\Z^3$, then counting the number of points on this lattice of norm $n$ is equivalent to solving the Diophantine equation 
\begin{equation}\label{eqn:sum3squares}
\sum_{j=1}^{3}x_j^2=n.
\end{equation}
Of course, one can also count points on other lattices in an analogous way. Counting the number of points on the lattice $\sqrt{2}\Z\times\Z^2$ of norm $n$ yields the Diophantine equation 
\[
2x_1^2+x_2^2+x_3^2=n.
\]
Fermat conjectured in 1638 that every positive integer may be written as the sum of four squares and three triangular numbers $T_x=\frac{x(x+1)}{2}$ (which is equivalent to \eqref{eqn:sum3squares} with $n\mapsto 8n+3$ by completing the square), which can be interpreted in Euclidean space via the above geometric relation. These statements were proven in celebrated works of Lagrange in 1770 and Gauss in 1796. Gauss went further, giving the following precise count of the number of solutions $r_3(n)$ to \eqref{eqn:sum3squares} in terms of the the Hurwitz class number $H(N)$ ($N>0$) associated to quadratic orders in the ring of integers in $\Q(\sqrt{-N})$ (see \eqref{def-Hurwitz-class-num} for a definition):
\begin{equation}\label{eqn:Gauss}
r_3(n)=\begin{cases}
12H(4n)&\text{if }n\equiv 1,2\pmod{4},\\
24H(n)&\text{if }n\equiv 3\pmod{8},\\
r_3\left(\frac{n}{4}\right)&\text{if }4\mid n,\\
0&\text{otherwise}.
\end{cases}
\end{equation}
Formulas like \eqref{eqn:Gauss} for the number of solutions to more general Diophantine equations resembling \eqref{eqn:sum3squares} are somewhat special in a sense, but also not entirely isolated in another sense. To describe this, we recall that a \begin{it}positive-definite integer-valued quadratic form\end{it} is a homogeneous polynomial $Q:\Z^{\ell}\to\N_0:=\N\cup\{0\}$ of degree two that satisfies $Q(\bm{x})=0$ if and only if $\bm{x}=\bm{0}$. Hence for some $a_{i,j}\in\Z$, one can write 
\[
Q(\bm{x})=\sum_{1\leq i\leq j\leq\ell} a_{i,j} x_ix_j.
\]
We restrict to the case that $\ell=3$ (the ternary case) and we also further restrict our consideration to quadratic forms that satisfy $a_{i,j}\in 2\Z$ for $i\neq j$ (these are called \begin{it}integral\end{it} quadratic forms). For $n\in\N$, we let
\[
r_Q(n):=\#\{\bm{x}\in\Z^{\ell}: Q(\bm{x})=n\}
\]
 denote the number of solutions $\bm{x}\in\Z^{\ell}$ to $Q(\bm{x})=n$. Although equations like \eqref{eqn:Gauss} are ``rare'' for $r_Q(n)$ in general, one expects that they satisfy formulas that ``approximate'' \eqref{eqn:Gauss}. Namely, $r_Q(n)$ should essentially satisfy a formula like \eqref{eqn:Gauss} up to an error that in general grows slower than the class numbers as $n\to\infty$. Taking a natural weighted average $r_{\gen(Q)}(n)$ (see \eqref{eqn:SiegelWeil}), known as the Siegel--Weil average, over the classes in what is called the \begin{it}genus\end{it} $\gen(Q)$ (roughly speaking, those quadratic forms that cannot be distinguished from $Q$ by looking $p$-adically one prime at a time; one cannot $p$-adically distinguish $x_1^2+4x_2^2+9x_3^2$ from $x_1^2+x_2^2+36x_3^2$, for example), one expects that $r_{\gen(Q)}(n)$ is related to class numbers in a manner similar to \eqref{eqn:Gauss}, while for $n$ away from finitely many ``bad'' square classes (i.e., $n\notin t_j\Z^2$ for some finite set of squarefree positive integers $\{t_1,\dots,t_r\}$), Duke and Schulze-Pillot \cite{DukeSchulzePillot} have proven that 
\begin{equation}\label{eqn:DukeSchulzePillot}
r_Q(n)=r_{\gen(Q)}(n) + O\left(n^{\frac{3}{7}+\varepsilon}\right). 
\end{equation}
Moreover, Jones \cite[Theorem 86]{Jones} has proven that for $n$ relatively prime to twice the \begin{it}discriminant\end{it} (the determinant of the matrix whose $(i,j)$-th component is $\frac{1+\delta_{i=j}}{2} a_{i,j}$) $d_Q$, the numbers $r_{\gen(Q)}(n)$ are proportional to certain class numbers associated to binary quadratic forms of discriminant $-Dn$ and Siegel \cite{SiegelQuadratische} has proven that the class numbers grow faster than $(Dn)^{\frac{1}{2}-\varepsilon}$. So for $n$ away from finitely many bad square classes (these are known as spinor exceptional square classes due to a connection to the spinor norm map when considering the algebraic theory of quadratic lattices), $r_Q(n)$ satisfies a formula like \eqref{eqn:Gauss} up to a ``small'' error for those $n$ relatively prime to $2d_Q$. The error is zero in the case of \eqref{eqn:Gauss} because the quadratic form given by the sum of three squares lies in a genus with only one class, and hence the average $r_{\gen(Q)}(n)$ equals $r_Q(n)$. 
 
As noted in \cite[(1.3)]{BKTernaryQF}, one can uniformly write \eqref{eqn:Gauss} as 
\[
r_3(n)=12H(4n)-24H(n).
\]
A number of similar formulas relating $r_Q(n)$ with sums of class numbers are found in \cite[Lemmas 4.1 to 4.56]{BKTernaryQF} when $Q$ is diagonal (i.e., $a_{i,j}=0$ for $i\neq j$) quadratic forms that lie in a genus with only one class. In \cite{BKTernaryQF}, analytic techniques were used to obtain the formulas one-by-one. In this paper, we employ algebraic techniques in order to obtain these types of approximate formulas.
\begin{theorem}\label{thm:sumclassnumapproximate}
Let $Q$ be an integral positive-definite ternary quadratic form and let $s_Q$ be the squarefree part of $d_Q$. Then there exist $M\in\N$ and constants $a_{Q,f}(m)$ and $c_Q(m)$ only depending on $m$ modulo $M$ such that for $n$ not in any of the finitely-many spinor exceptional square classes $t_j\Z^2$ we have 
\[
r_Q(n)=c_Q(n)\sum_{f\mid 2d_Q} a_{Q,f}(n)\cdot H\pfrac{4s_Qn}{f^2}+O\left(n^{\frac{3}{7}+\varepsilon}\right).
\]
In particular,
\[
r_{\gen(Q)}(n)= c_Q(n)\sum_{f\mid 2d_Q} a_{Q,f}(n)\cdot H\pfrac{4s_Qn}{f^2}.
\]
\end{theorem}
\begin{remark}
The assumption that $Q$ is integral (i.e., $a_{i,j}\in 2\Z$ for $i\neq j$) can be removed by noting that $r_Q(n)=r_{2Q}(2n)$, yielding a result for what are known as integer-valued quadratic forms.
\end{remark}
The constants $a_{Q,f}(m)$ and $c_Q(m)$ in Theorem \ref{thm:sumclassnumapproximate} are recursively constructed by applying a certain reduction theory using transformations of Watson from quadratic forms to ``simpler'' quadratic forms coming from what are known as \begin{it}stable lattices\end{it}. In practice, one can compute these constants given an explicit form $Q$ (see Example \ref{ex:arbitrary}, where the case $Q(\bm{x})=x_1^2+x_2^2+75x_3^2$ is worked out in detail), but it is difficult to obtain a general formula for these constants due to their recursive definition. However, for the ``base case'' of a quadratic form coming from a stable lattice (see \eqref{def-quadform-assoc-lattice} for a definition of a quadratic form coming form a lattice), we obtain the constants explicitly. To state the theorem, for $Q$ coming from a stable lattice $L$, we set (see \eqref{defn-mathfrakP} for a definition of $\mathfrak{P}$ for arbitrary lattices/quadratic forms)
\[
\mathfrak{P}=\mathfrak{P}_Q:=2^{1-2\ord_2(d_Q)} d_Q,
\]
we set $e_Q:=4$ if $L$ is an odd lattice and $e_Q:=1$ if $L$ is an even lattice (see Section \ref{sec:prelim} for a definition of even and odd lattices), and for a prime $p$ we let $S_p^*(L)$ denote a modified \begin{it}Hasse symbol\end{it} defined in \eqref{HasseSymbol-IsoAniso} (which takes values in $\{\pm 1\}$).
\begin{theorem}\label{mainthm}
	Let $Q$ be a positive-definite integral ternary quadratic form associated to a stable lattice $L$. Then for any $n\in \N$ not contained in the finitely-many spinor exceptional square classes $t_j\Z^2$, we have
	\[
	r_Q(n)= 12\prod_{p\mid \mathfrak{P}} \left(p+S_p^\ast(L)\right)^{-1} \times \sum_{f\mid \mathfrak{P}} \left[\prod_{p\mid f} S_p^\ast(L) \right] \cdot f \cdot H\pfrac{e_Q d_Q n}{f^2}+O\left(n^{\frac{3}{7}+\varepsilon}\right).
	\]
In particular, 
\[
r_{\gen(Q)}(n)= 12\prod_{p\mid \mathfrak{P}} \left(p+S_p^\ast(L)\right)^{-1} \times \sum_{f\mid \mathfrak{P}} \left[\prod_{p\mid f} S_p^\ast(L) \right] \cdot f \cdot H\pfrac{e_Q d_Q n}{f^2}.
\]
\end{theorem}
The paper is organized as follows. In Section \ref{sec:prelim}, we introduce the theory of quadratic lattices needed for our purposes, including local densities and the Siegel--Weil average $r_{\gen(Q)}(n)$, stable lattices, and reduction theory via Watson transformations. In Section \ref{sec:stable}, we investigate the Siegel--Weil average on quadratic forms related to stable lattices, proving Theorem \ref{mainthm}. Section \ref{section-formula-reduction} is devoted to the recursive formulas coming from reduction from arbitrary lattices to stable lattices and the proof of Theorem \ref{thm:sumclassnumapproximate}. We conclude the paper by demonstrating how to apply the theory to some explicit examples in Section \ref{sec:examples}.

\section{Preliminaries}\label{sec:prelim}

\subsection{Quadratic forms and quadratic lattices}\label{sec:quadlattice}
We adopt the language of quadratic spaces and lattices. Any unexplained notation and terminology can be found in \cite{KiBook} or \cite{OMBook}. We let $V$ be a \begin{it}quadratic space\end{it} over $\Q$, which is a $\Q$-vector space with an associated non-degenerate symmetric bilinear form $B$ such that $B(\bm{u},\bm{v})\in \Q$ for all $\bm{u},\bm{v}\in V$. The corresponding quadratic map is defined by $Q(\bm{x})=B(\bm{x},\bm{x})$ for any $\bm{x}\in V$. A \begin{it}quadratic $\Z$-lattice\end{it} is a free $\Z$-module  
\[
L=\Z \bm{v_1} + \Z \bm{v_2} + \cdots + \Z \bm{v_k}
\]
of rank $k$, where $\bm{v_1},\dots,\bm{v_k}\in V$ is a basis of $V$. The $k\times k$ symmetric matrix $A_L:=(B(\bm{v_i},\bm{v_j}))$ is called the \begin{it}Gram matrix\end{it} of $L$ with respect to a basis $\bm{v_1},\ldots, \bm{v_k}$, and we write $L \cong (B(\bm{v_i},\bm{v_j}))$. The discriminant $d_L$ is defined to be the determinant of $(B(\bm{v_i},\bm{v_j}))$. Any $\bm{x}\in L$ may be written as $\bm{x}=\sum_{j=1}^k x_j\bm{v_j}$ for some $(x_1,\ldots,x_k)^T\in\Z^k$, and by a {\em quadratic form associated to} (or coming from) $L$ (with respect to a basis $\{\bm{v_1},\ldots,\bm{v_k}\}$), we mean 
\begin{equation}\label{def-quadform-assoc-lattice}
	Q_L(x_1,\ldots,x_k):=Q(\bm{x})=\sum_{1\le i, j \le k} B(\bm{v_i},\bm{v_j})x_ix_j.
\end{equation}
Note that one may also write the quadratic form in terms of the Gram matrix via 
\[
Q_{L}(x_1,\ldots,x_k)=(x_1,\ldots,x_k) A_L (x_1,\ldots,x_k)^T. 
\]
The quadratic form $Q_L$ is integral precisely when the Gram matrix $A_L$ has integer entries and integer-valued when the diagonal entries of $A_L$ are integers and the other entries lie in $\frac{1}{2}\Z$. If $B(\bm{v_i},\bm{v_j})=0$ for any $i\neq j$ (equivalently, if the Gram matrix is diagonal), then we simply write $L=\la Q(\bm{v_1}), \ldots , Q(\bm{v_k}) \ra$. More generally, if $L=L_1+\dots +L_r$ for sublattices $L_1,\dots, L_r$ with $B(\bm{x_j},\bm{x_m})=0$ for every $\bm{x_j}\in L_j$ and $\bm{x_m}\in L_m$ whenever $j\neq m$, then we write 
\[
L=L_1\perp L_2\perp\cdots\perp  L_r.
\]
Note that $L=\la Q(\bm{v_1}), \ldots , Q(\bm{v_k}) \ra$ is precisely the case where each $L_j$ has rank one.

The scale ideal $\mathfrak{s}(L)$ (resp. the norm ideal $\n(L)$) of $L$ is the $\Z$-module generated by the subset $B(L,L)$ (resp. $Q(L)$) of $\Q$. We say a $\Z$-lattice $L$ is {\em positive definite} if $Q(\bm{x})>0$ for any $\bm{x}\in L\setminus\{0\}$. Since $Q_L$ is integral precisely when $\mathfrak{s}(L) \subseteq \Z$, we call such a $\Z$-lattice $L$ {\em integral} and we call the lattice {\em primitive} if $\mathfrak{s}(L) = \Z$. When a lattice is not integral or primitive, it is frequently beneficial to \begin{it}scale\end{it} the lattice appropriately; hence for $a\in\Q$ and a lattice $L$ in the quadratic space $V$, let $V^a$ denote the vector space $V$ provided with a new bilinear form $B^a(\bm{x},\bm{y}):=aB(\bm{x},\bm{y})$, and we shall use $L^a$ to denote the  \begin{it}scaled lattice\end{it}, which is lattice $L$ when it is regarded as a lattice in $V^a$. The quadratic form $Q_L$ is integer-valued precisely when $\n(L)\subseteq \Z$. Throughout this article, we restrict ourselves to  positive-definite integral $\Z$-lattice, and hereafter simply use the term ``lattice'' to refer to such lattices. We say that a lattice $L$ is {\em even} if $\n(L)\subseteq 2\Z$, and it is {\em odd} otherwise.

For two lattices $K$ and $L$, we say $K$ is represented by $L$, denoted $K \rightarrow L$,  if there is a linear map $\sigma : K\rightarrow L$ such that 
$$
B(\sigma(\bm{x}),\sigma(\bm{y})) = B(\bm{x},\bm{y}) \quad \text{for any }\bm{x}, \bm{y}\in K.
$$
Such a linear map is called a {\em representation} from $K$ to $L$ (or a representation of $K$ by $L$).

We say $K$ and $L$ are {\em isometric} to each other, denoted $K \cong L$, if $K \rightarrow L$ and $L \rightarrow K$, and in this case we call $\sigma$ an \begin{it}isometry\end{it}. 
We let $R(K,L)$ be the set of all representations from $K$ to $L$ and denote its size by $r(K,L):=|R(K,L)|$. 
In particular, the set $O(L):=R(L,L)$ is the group of {\em automorphs} of $L$, and we use the notation $o(L)$ to denote $r(L,L)$.
Moreover, for any integer $n$, we abbreviate 
\[
r(n,L):=r(\la n \ra, L) = \# \{\bm{x} \in L : Q(\bm{x})=n\}=r_{Q_L}(n).
\]
Note that although $Q_L$ depends on the choice of basis, $r_{Q_L}(n)=r(n,L)$ does not. This is because $r_{Q_L}(n)$ is a class invariant (under isometry) and different choices of basis yield isometric quadratic forms via the isometry defined by the change of basis map.

\subsection{Local densities and the Siegel--Weil averaging theorem}
For a prime $p$, a quadratic lattice over the ring of $p$-adic integers $\Z_p$ and related notations and terminologies may be defined analogously to the definitions given in Section \ref{sec:quadlattice}. Let $M$ and $N$ be two integral $\Z_p$-lattices of rank $s$ and $t$ with Gram matrices $S$ and $T$, respectively. Suppose that $t\ge s \ge 1$. Then Siegel \cite{Siegel1,Siegel2} defined the {\em local density} of representations of $M$ by $N$ as
$$
\alpha_p(M,N):= 2^{-\delta_{s,t}}  \lim_{r\rightarrow\infty} (p^r)^{s(s+1)/2-st} \# A_{p^r}(S,T),
$$
where $\delta_{m,n}=1$ if $m=n$ and $0$ otherwise, and (embedding $\Z/p^r\Z$ into $\Z_p$ by choosing an integer representative as usual)
$$
A_{p^r}(S,T):=\{X \in M_{t,s}(\Z/p^r\Z)  : X^tTX\equiv S \pmod{p^r}\}.
$$
Sometimes the following modified local density is used:
$$
\beta_p(M,N):=  \lim_{r\rightarrow\infty} (p^r)^{s(s+1)/2-st} \# \{X \in M_{t,s}(\Z/p^r\Z) : X^tTX\equiv S \pmod{p^rE_s(\Z_p)}\},
$$
where $E_s(\Z_p)$ is the set of $s\times s$ symmetric matrices over $\Z_p$ all of whose diagonal entries belong to $2\Z_p$.
In fact, they are related by $\a_p(M,N)=2^{s\delta_{p,2}-\delta_{s,t}}\beta_p(M,N)$, where $\delta_{p,2}=0$ if $p$ is odd, and $\delta_{p,2}=1$ if $p=2$.

Let $L$ be a lattice. For any prime number $p$, we define the localization of $L$ at $p$ by the $\Z_p$-lattice $L_p:=L\otimes_\Z \Z_p$. We define the {\em genus} $\gen(L)$ of $L$ as 
$$
\gen(L)=\{K\subseteq \Q L : K_p \cong L_p \text{ for any prime } p\},
$$ 
where $\Q L= \{ \a v : \a \in \Q , v \in L\}$ is the quadratic space $V$ on which $L$ lies and $K$ are $\Z$-lattices. The isometric relation induces an equivalence relation on $\gen(L)$, and $\gen(L)/_\sim$ is the set of classes in $\gen(L)$. The latter is a finite set and its cardinality is called the {\em  class number} of $L$. The (isometry) class of $L$ in $\gen(L)$ is denoted by $[L]$.
For two lattices $M$ and $N$, and a prime $p$, we simply write $\a_p(M,N)=\a_p(M_p,N_p)$ and $\beta_p(M,N)=\beta_p(M_p,N_p)$.

Let $L$ be a lattice and let $n$ be a positive integer. 
We say $n$ is {\em locally represented} by $L$ if $n$ is represented by $L_p$ over $\Z_p$ for any prime number $p$.
Define
\begin{equation}\label{eqn:SiegelWeil}
w(L):=\sum_{[T]\in \gen(L)/_\sim} \frac{1}{o(T)} \quad \text{and} \quad r(n,\gen(L)):= \frac{1}{w(L)}\sum_{[T]\in \gen(L)/_\sim} \frac{r(n,T)}{o(T)}.
\end{equation}
The value $w(L)$ is called the {\em mass (or weight)} of $L$. It is well known that if $n$ is locally represented by $L$, then $n$ is represented by a lattice $K$ in $\gen(L)$ (see \cite[102:5]{OMBook}), and hence $r(n,\gen(L))>0$.
\begin{theorem}[Minkowski--Siegel Formula]\label{Minkowski-Siegel}
Let $L$ be a lattice of rank $\ell\ge2$ and $n$ be an integer. 
Let $\varepsilon_\ell = \frac{1}{2}$ if $\ell=2$ and $\varepsilon_\ell=1$ otherwise.
Then we have
\begin{gather*}
w(L)^{-1}=\frac{1}{2} \cdot \pi^{\ell(\ell+1)/4} \cdot \prod_{i=0}^{\ell-1} \Gamma((\ell-i)/2)^{-1} \cdot (d_L)^{-(\ell+1)/2}\times  \prod_{p} \a_p(L,L), \text{ and}\\
r(n,\gen(L))=\varepsilon_{\ell} \cdot \pi^{\ell/2} \cdot \Gamma(\ell/2)^{-1} \cdot (d_L)^{-1/2} n^{\ell/2-1} \times \prod_{p} \a_p(n,L).	
\end{gather*}
Here, $\Gamma$ is the ordinary gamma function and $\a_p$ is the local density at a prime $p$.
\end{theorem}

\subsection{Stable lattices} 
We say a primitive ternary lattice $L$ is \begin{it}stable at a prime $p$\end{it} if $\ord_p(d_L)\le 1$, and in addition  when $p=2$, $\ord_2(d_L)=1$ if and only if $L$ is even.
A primitive ternary lattice $L$ is called {\em stable} if it is stable at every prime.
We say $L_p$ is {\em isotropic} if $Q(\bm{x})= 0$ for some non-zero element $\bm{x}\in L_p$, and it is {\em anisotropic} otherwise. 
Note that for any ternary lattice $L$, the Hasse symbol $S_p(L)$ of $L$ at a prime $p$ indicates whether $L_p$ is isotropic or not; more precisely (see \cite[58:6]{OMBook})
\begin{equation}\label{HasseSymbol-IsoAniso}
S_p^\ast(L):=(-1)^{\delta_{p,2}}S_p(L)=
\begin{cases}
	1 & \text{if } L_p \text{ is isotropic},\\
	-1 & \text{if } L_p \text{ is anisotropic}.
\end{cases}
\end{equation}

For each prime $p$, a Jordan decomposition of $L_p$ for a stable lattice $L$ may be described as follows.
\begin{lemma}\label{JordanDecomp-stable}
Let $L$ be a stable lattice. For an odd prime $p$, we have
$$
L_p \cong \begin{cases}
	\la 1, 1, d_L \ra & \text{if } \ord_p(d_L)=0,\\
	\la 1, -1, -d_L \ra & \text{if } \ord_p(d_L)=1 \text{ and } L_p \text{ is isotropic},\\
	\la 1, -\Delta_p, -\Delta_p d_L \ra & \text{if } \ord_p(d_L)=1 \text{ and } L_p \text{ is anisotropic},
\end{cases}
$$
where $\Delta_p$ denotes a non-square unit in $\Z_p$. For $p=2$, we have
$$
L_2 \cong \begin{cases}
	\mathbb{A}\perp \la 3d_L \ra & \text{if } L_2 \text{ is anisotropic},\\
	\mathbb{H}\perp \la -d_L \ra & \text{otherwise},
\end{cases}
$$
where $\mathbb{A}= \left( \begin{smallmatrix}	2&1\\1&2  \end{smallmatrix}\right)$ and $\mathbb{H}= \left( \begin{smallmatrix}	0&1\\1&0  \end{smallmatrix}\right)$.
\end{lemma}

The above lemma implies that if $K$ and $L$ are two stable lattices with $d_K=d_L$, then for any prime $p$ we have $K_p\cong L_p$ if and only if $S_p(K)=S_p(L)$. 
We next give a formula for the local density of a lattice $L$ at an odd prime $p$ at which the lattice is stable. To describe the result, we write $L_p=M_p\perp U_p$, where $M_p$ is the leading Jordan component and $\mathfrak{s}(U_p) \subseteq p\mathfrak{s}(M_p)$. 
\begin{lemma}\label{localdensities-stable1}
Let $p$ be an odd prime and let $L$ be a ternary lattice that is stable at $p$. For $n\in \N$, let $\nu_p:=\ord_p(n)$. 
If $\ord_p(d_L)=0$, then we have
$$
\a_p(n,L)=
\begin{cases}
	1+\frac{1}{p}-\frac{1}{p^{\frac{\nu_p+2}{2}}}+\left( \frac{-p^{-\nu_p}nd_L}{p}\right) \frac{1}{p^{\frac{\nu_p+2}{2}}} & \text{if } \nu_p \text{ is even},\\
	1+\frac{1}{p}-\frac{1}{p^{\frac{\nu_p+1}{2}}}-\frac{1}{p^{\frac{\nu_p+3}{2}}} & \text{if } \nu_p \text{ is odd}.
\end{cases}
$$
If $\ord_p(d_L)=1$, then we have
$$
\a_p(n,L)=
\begin{cases}
	1+\left(1-\frac{1}{p^{\frac{\nu_p}{2}}}-\frac{1}{p^{\frac{\nu_p+2}{2}}}\right) \left(\frac{-d_{M_p}}{p}\right) & \text{if } \nu_p \text{ is even},\\
	1+\left(1-\frac{1}{p^{\frac{\nu_p+1}{2}}}+\left( \frac{-p^{-(\nu_p+1)}nd_L}{p}\right) \frac{1}{p^{\frac{\nu_p+1}{2}}}\right) \left(\frac{-d_{M_p}}{p}\right) & \text{if } \nu_p \text{ is odd}.
\end{cases}
$$
Here, $\left(\frac{\cdot}{p}\right)$ is the Legendre symbol.
\end{lemma}
\begin{proof}
	This is a direct consequence of Lemma \ref{JordanDecomp-stable} and \cite[Theorem 3.1]{Yang}.
\begin{extradetails}
\rm 
By Lemma \ref{JordanDecomp-stable}, if $\ord_{p}(d_L)=0$, then
\[
 L_p \cong \la 1, 1, d_L \ra.
\]
We are interested in the $\ell=r=0$ case of \cite[Theorem 3.1]{Yang}. Plugging in the definition \cite[(3.2)]{Yang}, we get 
\begin{equation}\label{eqn:Yang}
\alpha_{p}(n,L)=1+ R_1(1,n,L)= 1+  \left(1-p^{-1}\right) \sum_{\substack{0< k\leq \nu_p\\ \ell(k,1)\text{ is even}}} v_kp^{d(k)} +v_{\nu_p+1} p^{d(\nu_p+1)}f_1(n).
\end{equation}
Here, from \cite[(1.6)]{Yang}, for $L_p=\la \varepsilon_1 p^{\ell_1},\dots,\varepsilon_d p^{\ell_d}\ra$ we have
\[
L(k,1):=\{1\leq j\leq d: \ell_j-k<0\text{ is odd}\}\qquad \ell(k,1):=\#L(k,1),
\]
from \cite[(1.7)]{Yang} we have
\[
d(k):=k+\frac{1}{2}\sum_{\ell_j<k} \left(\ell_j-k\right),
\]
from \cite[(1.8)]{Yang} we have 
\[
v(k):=\left(\frac{-1}{p}\right)^{\left\lfloor\frac{\ell(k,1)}{2}\right\rfloor} \prod_{j\in L(k,1)} \left(\frac{\ell_j}{p}\right),
\]
and for $n=n'p^{\nu_p}$ we have (see \cite[(3.1)]{Yang})
\[
f_1(n)=\begin{cases} 
-\frac{1}{p}&\text{if }\ell(\nu_p+1,1)\text{ is even},\\
\left(\frac{n'}{p}\right)\frac{1}{\sqrt{p}}&\text{if }\ell(\nu_p+1,1)\text{ is odd},
\end{cases}
\]
In our case, $d=3$ and $\ell_j=0$ for all $j$, so for $k\in \N$ we have $d(k)=k-\frac{3}{2}(k)=-\frac{k}{2}$, while $\ell(k,1)=3$ if $k$ is odd and $\ell(k,1)=0$ if $k$ is even. Hence, plugging this into \eqref{eqn:Yang} and using the geometric series formula 
\begin{equation}\label{eqn:geomseries}
\sum_{k=1}^{A} x^{-k} = x^{-1}\sum_{k=0}^{A-1} x^{-k} = x^{-1}\frac{1-x^{-A}}{1-x^{-1}}
\end{equation}
 in the second step below,
\begin{multline}\label{eqn:alphapnLord0}
\alpha_{p}(n,L)=1 + \left(1-p^{-1}\right)\sum_{k=1}^{\left\lfloor\frac{\nu_{p}}{2}\right\rfloor} p^{-k}  +\begin{cases}
  \left(\frac{-d_L}{p}\right) p^{-\frac{\nu_p+1}{2}}\left(\frac{n'}{p}\right) p^{-\frac{1}{2}}&\text{if }\nu_{p}\text{ is even},\\
-p^{-\frac{\nu_p+3}{2}}&\text{if }\nu_p\text{ is odd},
\end{cases}
\\
=1+\left(1-p^{-1}\right)p^{-1}\frac{1-p^{-\left\lfloor\frac{\nu_p}{2}\right\rfloor}}{1-p^{-1}} + \begin{cases}
  \left(\frac{-n' d_L}{p}\right) p^{-\frac{\nu_p}{2}-1}&\text{if }\nu_{p}\text{ is even},\\
-p^{-\frac{\nu_p+3}{2}}&\text{if }\nu_p\text{ is odd},
\end{cases}
\\
=1+p^{-1}-p^{-\left\lfloor\frac{\nu_p}{2}\right\rfloor-1}+ \begin{cases}
  \left(\frac{-n'd_L}{p}\right) p^{-\frac{\nu_p}{2}-1}&\text{if }\nu_{p}\text{ is even},\\
-p^{-\frac{\nu_p+3}{2}}&\text{if }\nu_p\text{ is odd},
\end{cases}
\\
=1+p^{-1} +\begin{cases}
 -p^{\frac{\nu_p}{2}-1}+ \left(\frac{-n'd_L}{p}\right) p^{-\frac{\nu_p}{2}-1}&\text{if }\nu_{p}\text{ is even},\\
-p^{-\frac{\nu_p+1}{2}}-p^{-\frac{\nu_p+3}{2}}&\text{if }\nu_p\text{ is odd}.
\end{cases}
\end{multline}
This is precisely the claim in the case $\ord_p(d_L)=0$.

Now suppose that $\ord_p(d_L)=1$. In this case, we have $\ell_1=\ell_2=0$ and $\ell_3=1$. Therefore 
\[
\ell(k,1)=\begin{cases}
2&\text{if $k$ is odd},\\
1&\text{if $k$ is even}.
\end{cases}
\]
For $k\in\N$ We also have 
\[
d(k)=k+\frac{1}{2}\sum_{\ell_j< k} \left(\ell_j-k\right)  =k+\frac{1}{2}\sum_{\ell_j\leq k} \left(\ell_j-k\right) = \frac{-k}{2}+\frac{1}{2}=\frac{1-k}{2}.
\]
By Lemma \ref{JordanDecomp-stable} we either have 
\[
L_p\cong \la 1,-1,-d_L\ra \qquad\text{ or }\qquad L_p\cong \la 1,-\Delta_p,-\Delta_pd_L\ra,
\]
We have written $M_p:=\la 1,-1\ra$ (resp. $M_p:=\la 1,-\Delta_p\ra$) and $U_p:=\la -d_L\ra$ (resp. $U_p:=\la -\Delta_p d_L\ra$) in the first (resp. second) case. Note that 
\[
d_{U_p} = d_{M_p} d_{L},
\]
 so we have 
\[
v(k)=\begin{cases} 
\left(\frac{-d_{M_p}}{p}\right)&\text{if $k$ is odd},\\
\left(\frac{d_{M_p} d_{L}}{p}\right)&\text{if $k$ is even}.
\end{cases}
\]
We see that \eqref{eqn:Yang} becomes (noting that only the terms with $k$ odd survive in the sum)
\begin{multline*}
\alpha_{p}(n,L)=1 + \left(1-p^{-1}\right)\left(\frac{-d_{M_p}}{p}\right)\sum_{k=0}^{\left\lfloor\frac{\nu_{p}-1}{2}\right\rfloor} p^{-k}  + \begin{cases}
  -\left(\frac{-d_{M_p}}{p}\right) p^{-\frac{\nu_p}{2}-1}&\text{if }\nu_{p}\text{ is even},\\
\left(\frac{n'd_{M_p}d_L}{p}\right) p^{-\frac{\nu_p+1}{2}}&\text{if }\nu_p\text{ is odd},
\end{cases}
\\
=1+\left(1-p^{-1}\right)\left(\frac{-d_{M_p}}{p}\right)\frac{1-p^{-\left\lfloor\frac{\nu_p+1}{2}\right\rfloor}}{1-p^{-1}} + 
\left(\frac{-d_{M_p}}{p}\right) \begin{cases}
 - p^{-\frac{\nu_p}{2}-1}&\text{if }\nu_{p}\text{ is even},\\
  \left(\frac{-n'd_L}{p}\right) p^{-\frac{\nu_p+1}{2}}&\text{if }\nu_p\text{ is odd},
\end{cases}
\\
=1+\left(\frac{-d_{M_p}}{p}\right) \times 
\begin{cases}
1-p^{-\frac{\nu_p}{2}}  -p^{-\frac{\nu_p}{2}-1}&\text{if }\nu_{p}\text{ is even},\\
1-p^{-\frac{\nu_p+1}{2}}+\left(\frac{-n'd_L}{p}\right) p^{-\frac{\nu_p+1}{2}}&\text{if }\nu_p\text{ is odd},
\end{cases}
\end{multline*}
as claimed.
\end{extradetails}
\end{proof}
We also require the local density at the prime $2$ for lattices that are $2$-adically stable.
\begin{lemma}\label{localdensities-stable2}
	Let $L$ be a ternary lattice that is stable at $2$. For $n\in \N$, write $n=\beta \cdot 2^a$ with $\beta\equiv 1\pmod{2}$ and $a\ge 0$.
	If $L$ is odd and anisotropic, then we have 
	$$
	\a_2(n,L)=\begin{cases}
		3\cdot 2^{-\frac{a+1}{2}} & \text{if } a \text{ is odd},\\
		3\cdot 2^{-\frac{a+2}{2}} & \text{if } a \text{ is even and } \beta\equiv d_L \pmod{4},\\
		2^{-\frac{a}{2}} & \text{if } a \text{ is even and } \beta\equiv 3d_L \pmod{8},\\
		0 & \text{if } a \text{ is even and } \beta\equiv 7d_L \pmod{8}.\\		
	\end{cases}
	$$	
	If $L$ is odd and isotropic, then we have
	$$
	\a_2(n,L)=\begin{cases}
		2-3\cdot 2^{-\frac{a+1}{2}} & \text{if } a \text{ is odd},\\
		2-3\cdot 2^{-\frac{a+2}{2}} & \text{if } a \text{ is even and } \beta\equiv d_L \pmod{4},\\
		2-2^{-\frac{a}{2}} & \text{if } a \text{ is even and } \beta\equiv 3d_L \pmod{8},\\
		2 & \text{if } a \text{ is even and } \beta\equiv 7d_L \pmod{8}.\\		
	\end{cases}
	$$	
	If $L$ is even, then we have 
	$$
	\a_2(n,L)=\begin{cases}
		3-3\cdot 2^{-\frac{a}{2}} & \text{if } a \text{ is even},\\
		3-3\cdot 2^{-\frac{a+1}{2}} & \text{if } a \text{ is odd and } \beta\equiv d_L/2 \pmod{4},\\
		3-2^{-\frac{a-1}{2}} & \text{if } a \text{ is odd and } \beta\equiv 3d_L/2 \pmod{8},\\
		3 & \text{if } a \text{ is odd and } \beta\equiv 7d_L/2 \pmod{8}.\\		
	\end{cases}
	$$	
	
\end{lemma}
\begin{proof}
    Note that when $L$ is even, $L_2=\left( \begin{smallmatrix} 0&1 \\ 1&0 \end{smallmatrix}\right)\perp\la -d_L\ra$ and 
	$$
	\a_2(n,L)=\begin{cases}
		0 & \text{if } n \text{ is odd},\\
		2\a_2\left(\frac{n}{2},\left( \begin{smallmatrix} 0&\frac{1}{2} \\ \frac{1}{2}&0 \end{smallmatrix}\right)\perp \la -\frac{d_L}{2} \ra \right)& \text{if } n \text{ is even}.
	\end{cases} 
	$$
	Now, the lemma is a direct consequence of Lemma \ref{JordanDecomp-stable} and \cite[Theorem 4.1]{Yang}.
\begin{extradetails}
\rm
We want the $\ell=0$ case of \cite[Theorem 4.1]{Yang}. Plugging in the definition \cite[(4.4)]{Yang}, we get 
\begin{align}
\nonumber \alpha_{2}(n,L)&=1+ R_1(1,n,L)= 1+  \sum_{\substack{0< k\leq a+3 \\ \ell(k-1,1)\text{ is odd}}}\delta(k)  p(k)\left(2,\mu \epsilon(k)\right)_2' 2^{d(k)-\frac{3}{2}}\\
\label{eqn:Yangp=2}&\quad + \sum_{\substack{0< k\leq a+3 \\ \ell(k-1,1)\text{ is even}}} \delta(k) p(k) \left(2,\epsilon(k)\right)_2' 2^{d(k)-1}\psi\left(\frac{\mu}{8}\right)\operatorname{char}(4\Z_2)(\mu).
\end{align}
Here 
\[
(2,b)_2':=\begin{cases} (2,b)_2&\text{if $b$ is odd},\\ 0&\text{otherwise},\end{cases}
\]
where $(a,b)_p$ denotes the local Hilbert symbol, and, from \cite[(4.3)]{Yang}, for 
\[
L_2=\la \varepsilon_1 2^{\ell_1},\dots,\varepsilon_d 2^{\ell_d}\ra\oplus \bigoplus_{j=1}^M \varepsilon_j' 2^{m_j}\left(\begin{smallmatrix} 0&\frac{1}{2}\\ \frac{1}{2}&0\end{smallmatrix}\right)\oplus \bigoplus_{j=1}^N \varepsilon_j'' 2^{n_j}\left(\begin{smallmatrix} 1&\frac{1}{2}\\ \frac{1}{2}&1\end{smallmatrix}\right)
\]
 we have
\begin{align*}
L(k,1)&:=\{1\leq j\leq d: \ell_j-k<0\text{ is odd}\}\qquad \ell(k,1):=\#L(k,1),\\
d(k)&:=k+\frac{1}{2}\sum_{\ell_j<k-1} \left(\ell_j+1-k\right)+ \sum_{m_j<k}(m_j-k)+\sum_{n_j<k} (n_j-k),\\
p(k)&:=(-1)^{\sum_{n_j<k}(n_j-k)},\\
\epsilon(k)&:=\prod_{j\in L(k-1,1)} \epsilon_j,\\
\delta(k)&:=\begin{cases}0&\text{if $\ell_j=k-1$ for some $j$},\\
1&\text{otherwise},
\end{cases}
\end{align*}
by \cite[(4.6)]{Yang}, for $n=2^{a}\beta$ (with $\beta$ odd) we have 
\[
\mu=\mu_k(n):=\beta 2^{a+3-k} -\sum_{\ell_j<k-1}\epsilon_j,
\]
and (see \cite[after (1.2)]{Yang}) $\psi(x):=e^{-2\pi i \lambda(x)}$  with $\lambda(x)$ being the embedding of $x\in\Z_p$ to $\Q/\Z$ (with denominator a power of $p$).  

If $L$ is odd and anisotropic, then by Lemma \ref{JordanDecomp-stable}, we have 
\[
L_2\cong \left<3d_L\right>\perp 2\left(\begin{smallmatrix} 1&\frac{1}{2}\\ \frac{1}{2}&1\end{smallmatrix}\right).
\]  
Hence in this case we have $\ell_1=0$ (by definition, $\ord_2(d_L)=1$ if and only if $L$ is even, and we have $L$ odd here) and $n_1=1$. Thus for $k\in \N$ we have 
\[
d(k)=k-\frac{1}{2}(k-1)+(1-k)=\frac{3-k}{2},
\]
 while $\ell(k-1,1)=1$ if $k$ is even and $\ell(k-1,1)=0$ if $k$ is odd. We furthermore have 
\[
\epsilon(k)=\begin{cases} 3d_L &\text{if $k$ is even},\\ 1&\text{if $k$ is odd},
\end{cases}
\]
$\delta(1)=0$, and $\delta(k)p(k)=(-1)^{1-k}$ for $k>1$. Finally, $\mu=2^{a+3-k}\beta-3d_L$ for $k>1$. Hence, plugging this into \eqref{eqn:Yangp=2}, we have 
\begin{multline*}
	\alpha_{2}(n,L)=1 - \sum_{k=1}^{\left\lfloor\frac{a+3}{2}\right\rfloor} \left(2, 2^{a+3-2k}\beta 3d_L  - 9d_L^2\right)_2' 2^{-k} \\
	+\sum_{k=1}^{\left\lfloor\frac{a+2}{2}\right\rfloor}  2^{-k} \psi\left(\frac{ 2^{2-2k+a}\beta-3d_L}{8}\right)\chi(4\Z_2) \left(2^{2-2k+a}\beta-3d_L\right).
\end{multline*}	
Due to the characteristic function $\chi(4\Z_2)$, every term in the last sum other than the possible term $k=\frac{2+a}{2}=1+\frac{a}{2}$ (if $a$ is even) vanishes and we moreover compute
\[
\left(2, 2^{a+3-2k}\beta 3d_L  - 9d_L^2\right)_2' = \begin{cases}
(2,-1)_2=1&\text{if }k\leq \frac{a}{2},\\
(2,3)_2=-1&\text{if }k=\frac{a+1}{2},\\
(2,6 \beta d_L-1)_2=(-1)^{\frac{\beta d_L+1}{2}}&\text{if }k=\frac{a+2}{2},\\
0&\text{if }k=\frac{a+3}{2}.
\end{cases}
\]
We thus obtain (evaluating $\sum_{k=1}^d 2^{-k} = 1-2^{-d}$)
\begin{multline*}
\alpha_{2}(n,L)=1 - \left(1-2^{-\left\lfloor\frac{a}{2}\right\rfloor}\right)  + \delta_{2\nmid a} 2^{-\frac{a+1}{2}} -(-1)^{\frac{\beta d_L+1}{2}}\delta_{2\mid a} 2^{-\frac{a+2}{2}}\\
+\delta_{2\mid a} \delta_{\beta\equiv 3d_L\pmod{4}} (-1)^{\frac{\beta-3d_L}{4}} 2^{-\frac{a+2}{2}}.
\end{multline*}
If $a$ is odd, then we get 
\[
2^{-\frac{a-1}{2}} + 2^{-\frac{a+1}{2}} =  3\cdot 2^{-\frac{a+1}{2}},
\]
as claimed.

If $a$ is even, then we have 
	\begin{multline*}
	2^{-\frac{a}{2}} -(-1)^{\frac{\beta d_L+1}{2}} 2^{-\frac{a+2}{2}} +\delta_{\beta\equiv 3 d_L\mod{4}}(-1)^{\frac{\beta-3d_L}{4}} 2^{-\frac{a+2}{2}}\\
	=\begin{cases}
		2^{-\frac{a}{2}}+ 2^{-\frac{a+2}{2}}+0 & \text{if } \beta\equiv d_L\pmod{4},\\
		2^{-\frac{a}{2}}-2^{-\frac{a+2}{2}}+(-1)^{\frac{\beta-3d_L}{4}} 2^{-\frac{a+2}{2}} & \text{if } \beta\equiv 3d_L\pmod{4},\\
	\end{cases}\\
	=\begin{cases}
	3\cdot 2^{-\frac{a+2}{2}} & \text{if } \beta\equiv d_L\pmod{4},\\
	2^{-\frac{a}{2}}-2^{-\frac{a+2}{2}}+ 2^{-\frac{a+2}{2}}=2^{-\frac{a}{2}} & \text{if } \beta\equiv 3d_L\pmod{8},\\
	2^{-\frac{a}{2}}-2^{-\frac{a+2}{2}}- 2^{-\frac{a+2}{2}}=0 & \text{if } \beta\equiv 7d_L\pmod{8}.
	\end{cases}
	\end{multline*}

For $L$ odd and isotropic, Lemma \ref{JordanDecomp-stable} implies that 
\[
L_2\cong \left<-d_L\right>\perp 2\left(\begin{smallmatrix} 0&\frac{1}{2}\\ \frac{1}{2}&0\end{smallmatrix}\right).
\]  
Hence in this case we have $\ell_1=0$ and $m_1=1$. Thus for $k\in \N$ we have 
\[
d(k)=k-\frac{1}{2}(k-1)+1-k=\frac{3-k}{2},
\]
 while $\ell(k-1,1)=1$ if $k$ is even and $\ell(k-1,1)=0$ if $k$ is odd. We furthermore have 
\[
\epsilon(k)=\begin{cases} -d_L &\text{if $k$ is even},\\ 1&\text{if $k$ is odd},
\end{cases}
\]
$\delta(1)=0$, and $\delta(k)p(k)=1$ for $k>1$. Finally, $\mu=2^{a+3-k}\beta+d_L$ for $k>1$. Hence, plugging this into \eqref{eqn:Yangp=2}, we have 
\begin{multline*}
\alpha_{2}(n,L)=1 + \sum_{k=1}^{\left\lfloor\frac{a+3}{2}\right\rfloor} \left(2, -2^{a+3-2k}\beta d_L  -d_L^2\right)_2' 2^{-k} \\
 +\sum_{k=1}^{\left\lfloor\frac{a+2}{2}\right\rfloor}  2^{-k} \psi\left(\frac{ 2^{2-2k+a}\beta+d_L}{8}\right)\chi(4\Z_2)\left(2^{2-2k+a}\beta+d_L\right).
\end{multline*}
Every term in the last sum other than the possible term $k=\frac{2+a}{2}=1+\frac{a}{2}$ (if $a$ is even) again vanishes and  
\[
\left(2, -2^{a+3-2k}\beta d_L  - d_L^2\right)_2' = \begin{cases}
	(2,-1)_2=1&\text{if }k\leq \frac{a}{2},\\
	(2,-5)_2=-1&\text{if }k=\frac{a+1}{2},\\
	(2,-2\beta d_L-1)_2=(-1)^{\frac{\beta d_L+1}{2}}&\text{if }k=\frac{a+2}{2},\\
	0&\text{if }k=\frac{a+3}{2}.
\end{cases}
\]	
We thus obtain (evaluating $\sum_{k=1}^d 2^{-k} = 1-2^{-d}$)
\begin{multline*}
	\alpha_{2}(n,L)=1 + \left(1-2^{-\left\lfloor\frac{a}{2}\right\rfloor}\right)  - \delta_{2\nmid a} 2^{-\frac{a+1}{2}} +(-1)^{\frac{\beta d_L+1}{2}}\delta_{2\mid a} 2^{-\frac{a+2}{2}}\\
	+\delta_{2\mid a} \delta_{\beta\equiv -d_L\pmod{4}} (-1)^{\frac{\beta+d_L}{4}} 2^{-\frac{a+2}{2}}.
\end{multline*}	
If $a$ is odd, then we get 
\[
2-2^{-\frac{a-1}{2}} - 2^{-\frac{a+1}{2}} =  2-3\cdot 2^{-\frac{a+1}{2}},
\]
as claimed.

If $a$ is even, then we have 
	\begin{multline*}
		2-2^{-\frac{a}{2}} +(-1)^{\frac{\beta d_L+1}{2}} 2^{-\frac{a+2}{2}} +\delta_{\beta\equiv -d_L\mod{4}}(-1)^{\frac{\beta+d_L}{4}} 2^{-\frac{a+2}{2}}\\
		=\begin{cases}
			2-2^{-\frac{a}{2}}-2^{-\frac{a+2}{2}}+0 & \text{if } \beta\equiv d_L\pmod{4},\\
			2-2^{-\frac{a}{2}}+2^{-\frac{a+2}{2}}+(-1)^{\frac{\beta+d_L}{4}} 2^{-\frac{a+2}{2}} & \text{if } \beta\equiv 3d_L\pmod{4},\\
		\end{cases}\\
		=\begin{cases}
			2-3\cdot 2^{-\frac{a+2}{2}} & \text{if } \beta\equiv d_L\pmod{4},\\
			2-2^{-\frac{a}{2}}+2^{-\frac{a+2}{2}}-2^{-\frac{a+2}{2}} =2-2^{-\frac{a}{2}} & \text{if } \beta\equiv 3d_L\pmod{8},\\
			2-2^{-\frac{a}{2}}+2^{-\frac{a+2}{2}}+2^{-\frac{a+2}{2}} =2 & \text{if } \beta\equiv 7d_L\pmod{8}.
		\end{cases}
	\end{multline*}

Finally, assume that $L$ is even. Clearly we have $\alpha_2(n,L)=0$ if $a=0$, so we assume that $a\geq 1$. As noted in the proof, we then have 
\[
\alpha_2(n,L)=2\alpha_2\left(\frac{n}{2},\left<-\frac{d_L}{2}\right>\oplus \left(\begin{smallmatrix}0&\frac{1}{2}\\ \frac{1}{2}&0\end{smallmatrix}\right)\right).
\]
Since $\ord_2(d_L)=1$ by assumption, for $L':=\left<-\frac{d_L}{2}\right>\oplus \left(\begin{smallmatrix}0&\frac{1}{2}\\ \frac{1}{2}&0\end{smallmatrix}\right)$ we now have $\ell_1=0$ and $m_1=0$. Thus for $k\in \N$ we have 
\[
d(k)=k-\frac{1}{2}(k-1)-k=\frac{1-k}{2},
\]
 while $\ell(k-1,1)=1$ if $k$ is even and $\ell(k-1,1)=0$ if $k$ is odd. We furthermore have 
\[
\epsilon(k)=\begin{cases} -\frac{d_L}{2} &\text{if $k$ is even},\\ 1&\text{if $k$ is odd},
\end{cases}
\]
$\delta(1)=0$, and $\delta(k)p(k)=1$ for $k>1$. Finally, noting that we have $a\mapsto a-1$ because we are plugging in $\frac{n}{2}$ instead of $n$, $\mu=2^{a+2-k}\beta+\frac{d_L}{2}$ for $k>1$. Hence, plugging this into \eqref{eqn:Yangp=2}, we have 
\begin{multline*}
	\alpha_{2}\left(\frac{n}{2},L'\right)=1 + \sum_{k=1}^{\left\lfloor\frac{a+2}{2}\right\rfloor} \left(2, -2^{a+2-2k}\beta \frac{d_L}{2} -\left(\frac{d_L}{2}\right)^2\right)_2' 2^{-k-1} \\
	+\sum_{k=1}^{\left\lfloor\frac{a+1}{2}\right\rfloor}  2^{-k-1} \psi\left(\frac{ 2^{1-2k+a}\beta+\frac{d_L}{2}}{8}\right)\chi(4\Z_2)\left( 2^{1-2k+a}\beta+\frac{d_L}{2}\right).
\end{multline*}	
Every term in the last sum other than the possible term $k=\frac{1+a}{2}$ (if $a$ is odd) again vanishes and  
\[
\left(2, -2^{a+2-2k}\beta \frac{d_L}{2}  - \pfrac{d_L}{2}^2\right)_2' = \begin{cases}
	(2,-1)_2=1&\text{if }k\leq \frac{a-1}{2},\\
	(2,-5)_2=-1&\text{if }k=\frac{a}{2},\\
	(2,-2\beta\frac{d_L}{2}-1)_2=(-1)^{\frac{\beta \frac{d_L}{2}+1}{2}}&\text{if }k=\frac{a+1}{2},\\
	0&\text{if }k=\frac{a+2}{2}.
\end{cases}
\]	
We thus obtain (evaluating $\sum_{k=1}^d 2^{-k} = 1-2^{-d}$)
\begin{multline*}
	\alpha_{2}\left(\frac{n}{2},L'\right)=1 + \frac{1}{2}\left(1-2^{-\left\lfloor\frac{a-1}{2}\right\rfloor}\right)  -\delta_{2\mid a} 2^{-\frac{a}{2}-1} +(-1)^{\frac{\beta \frac{d_L}{2}+1}{2}}\delta_{2\nmid a} 2^{-\frac{a+3}{2}}\\
	+\delta_{2\nmid a} \delta_{\beta\equiv -\frac{d_L}{2}\pmod{4}} (-1)^{\frac{\beta+\frac{d_L}{2}}{4}} 2^{-\frac{a+3}{2}}.
\end{multline*}
If $a\geq 2$ is even, then we get 
\[
\frac{3}{2}-2^{-\frac{a}{2}} - 2^{-\frac{a}{2}-1} =  \frac{3}{2}-3\cdot 2^{-\frac{a}{2}-1},
\]
and multiplying by $2$ gives the claim. Note that for $a=0$ the above formula vanishes, so we may combine all cases with $a$ even to obtain one uniform formula.

If $a$ is odd, then we have 
\begin{multline*}
\frac{3}{2}-2^{-\frac{a+1}{2}} +(-1)^{\frac{\beta \frac{d_L}{2}+1}{2}} 2^{-\frac{a+3}{2}} +\delta_{\beta\equiv - \frac{d_L}{2}\mod{4}}(-1)^{\frac{\beta+\frac{d_L}{2}}{4}} 2^{-\frac{a+3}{2}}\\
=\begin{cases}
	\frac{3}{2}-2^{-\frac{a+1}{2}}-2^{-\frac{a+3}{2}}& \text{if } \beta\equiv d_L/2 \pmod{4}\\
	\frac{3}{2}-2^{-\frac{a+1}{2}}+2^{-\frac{a+3}{2}} +(-1)^{\frac{\beta+\frac{d_L}{2}}{4}} 2^{-\frac{a+3}{2}}& \text{if } \beta\equiv 3d_L/2 \pmod{4}\\
\end{cases}\\
=\begin{cases}
	\frac{3}{2}-3\cdot2^{-\frac{a+3}{2}}& \text{if } \beta\equiv d_L/2 \pmod{4}\\
	\frac{3}{2}-2^{-\frac{a+1}{2}}+2^{-\frac{a+3}{2}}-2^{-\frac{a+3}{2}}=\frac{3}{2}-2^{-\frac{a+1}{2}}& \text{if } \beta\equiv 3d_L/2 \pmod{8}\\
	\frac{3}{2}-2^{-\frac{a+1}{2}}+2^{-\frac{a+3}{2}} +2^{-\frac{a+3}{2}}=\frac{3}{2}& \text{if } \beta\equiv 7d_L/2 \pmod{8}\\
\end{cases}
\end{multline*}	
Multiplying by 2 again yields the claim.
\rm
\end{extradetails}
\end{proof}

\begin{corollary}\label{cor-local-rep-num}
	Let $L$ be a ternary lattice that is stable at a prime $p$. Then the set of $p$-adic integers that are represented by $L_p$ over $\Z_p$ is
	$$
	\begin{cases}
		\Z_p & \text{if $p$ is odd and $L_p$ is isotropic},\\
		2\Z_2 & \text{if $p=2$ and $L$ is even},\\
		\Z_p\setminus(-d_L)\Z_p^2 & \text{if $L_p$ is anisotropic}.
	\end{cases}
	$$
\end{corollary}
\begin{proof}
	Note that $n\in\Z_p$ is represented by $L_p$ if and only if $\a_p(n,L)\neq 0$. Thus, this is a consequence of Lemmas \ref{localdensities-stable1} and \ref{localdensities-stable2}.
	The corollary may also directly be verified from the Jordan decomposition of $L_p$ in Lemma \ref{JordanDecomp-stable} since it is maximal (see \cite[91:3]{OMBook}).
\end{proof}
\subsection{Reduction theory: Watson Transformations}
We breifly introduce the so-called Watson Transformations of lattices which will play a crucial role in Section \ref{section-formula-reduction}.
Let $m$ be a positive integer. For any lattice $L$, we define
$$
\Lambda_m(L):=\{\bm{x}\in L : Q(\bm{x}+\bm{y})\equiv Q(\bm{y}) \pmod{m} \text{ for all } \bm{y} \in L\},
$$
and for any prime number $p$, we define 
$$
\Lambda_m(L_p):=\{\bm{x}\in L_p : Q(\bm{x}+\bm{y})\equiv Q(\bm{y}) \pmod{m} \text{ for all } \bm{y} \in L_p\}.
$$
Clearly, $\Lambda_m(L)$ is a sublattice of $L$ and $\Lambda_m(L)\subseteq \{\bm{x}\in L : Q(\bm{x})\equiv 0 \pmod{m}\}$ by taking $\bm{y}=\bm{0}$.
Moreover, $\Lambda_m(L)_p=\Lambda_m(L_p)$ for any prime $p$, and $\Lambda_m(L)_p=L_p$ for any prime $p$ not dividing $m$.
We refer the readers to \cite{ChanEarnest} and \cite{ChanOh03} for more properties of the operators $\Lambda_m$.
We denote by $\lambda_m(L)$ the primitive lattice obtained by scaling the quadratic map on $\Lambda_m(L)$ suitably.
The mappings $L\mapsto \lambda_m(L)$ are called Watson transformations.

Now, we describe how $\Lambda_m$ transform $L$ when $m$ is prime or $4$.
For any prime $p$, we write $L_p=M_p\perp U_p$, where $M_p$ is the leading Jordan component and $\mathfrak{s}(U_p) \subseteq p\mathfrak{s}(M_p)$.
\begin{lemma}\label{even-transform}
	Suppose that $M_p$ is even unimodular and $\n(U_p)\subseteq 2p\Z_p$. Then
	$$
	\Lambda_{ep}(L)_p=pM_p\perp U_p,
	$$
	where $e=2$ if $p=2$, and $1$ otherwise. In particular,
	\begin{enumerate}[]
		\item[{\rm (a)}] if $\ord_p(d_L)\ge 2$, then $\ord_p(d_{\lambda_{ep}(L)})<\ord_p(d_L)$,
		\item[{\rm (b)}] if $m$ is a squarefree positive integer and $\ord_p(d_L)\le 1$ for all $p \mid m$, then $\lambda_m^2(L)=L$.
	\end{enumerate}
\end{lemma}
\begin{proof}
	See \cite[Lemma 2.1]{ChanOh14}. 
\end{proof}
We note that, in particular, for an odd prime $p$ one may recursively use part (a) of Lemma \ref{even-transform} to reduce the power of $p$ dividing the discriminant of the lattice until the order is at most one for the reduced lattice. An overview of the corresponding reduction for $p=2$ is given in the next lemma.
\begin{lemma}\label{odd-transform}
	Suppose that $L$ is a ternary integral lattice with $\n (L_2)=\Z_2$.
	\begin{enumerate}[leftmargin=*]
		\item[{\rm (a)}] If $\mathrm{rank}(M_2)=2$ and $\mathfrak{s}(U_2)\subseteq 2\Z_2$, then 
		$$
		\lambda_2(L)_2\cong \begin{cases}
			M_2\perp U_2^{\frac{1}{2}} & \text{if } d_M\equiv 1 \pmod{8},\\
			M_2^3\perp U_2^{\frac{1}{2}} & \text{if } d_M\equiv 5 \pmod{8},\\			
			\left(\begin{smallmatrix} 2&1\\1&2\end{smallmatrix}\right) \perp U_2^{\frac{1}{2}} & \text{if } d_M\equiv 3 \pmod{8},\\
			\left(\begin{smallmatrix} 0&1\\1&0\end{smallmatrix}\right)\perp U_2^{\frac{1}{2}} & \text{if } d_M\equiv 7 \pmod{8}.
		\end{cases}
		$$
		\item[{\rm (b)}] If $\mathrm{rank}(M_2)=1$ and $\mathfrak{s}(U_2)\subseteq 4\Z_2$, then $\lambda_2(L)_2 \cong M_2 \perp U_2^{\frac{1}{4}}$.
		\item[{\rm (c)}] If $\mathrm{rank}(M_2)=1$ and $\mathfrak{s}(U_2)= 2\Z_2$, then $\lambda_2(L)_2 \cong M_2^2 \perp U_2^{\frac{1}{2}}$.
	\end{enumerate}
	In particular, if $\ord_2(d_L)\ge2$, then $\ord_2(d_{\lambda_2(L)})<\ord_2 (d_L)$; and if $ord_2(d_L)=1$, then $\lambda_2(L)_2$ is unimodular.
\end{lemma}
\begin{proof}
	See \cite[Lemma 2.2]{ChanOh14}. 
\end{proof}
Combining Lemmas \ref{even-transform} and \ref{odd-transform}, Chan and Oh \cite[Corollary 2.4]{ChanOh14} proved the following.
\begin{corollary}\label{cor:finitereduction}
	A primitive ternary lattice $L$ can be transformed, via a finite sequence of Watson transformations at the primes dividing $d_L$ or at $4$, to a stable lattice.
\end{corollary}
To describe the relationship between the reduction theory via Corollary \ref{cor:finitereduction} and genera, one needs to investigate how isometries commute with the Watson transformations at primes dividing $d_L$ or at $4$.

\begin{lemma}\label{sigma-lambda-commute}
	Let $L$ be a lattice on a quadratic space $V$ and let $m$ be a prime or $4$. Then
	\begin{enumerate}
		\item[{\rm (1)}] $\sigma\circ \Lambda_m(L)=\Lambda_m\circ \sigma(L)$ for any $\sigma \in O(V)$. In particular, the group $O(L)$ is a subgroup of $O(\Lambda_k(L))$.
		\item[{\rm (2)}] $\Lambda_m$ induces a surjective function from $\gen(L)/_\sim$ onto $\gen(\Lambda_m(L))/_\sim$.
	\end{enumerate}
\end{lemma}
\begin{proof}
	See \cite{ChanOh14}. The case when $m=4$ may be shown similarly.
\end{proof}

Let $L$ be a lattice and let $m$ be a prime or $4$. For any lattice $N$ in $\gen(\Lambda_m(L))$, we define
$$
\Gamma_m^L(N):=\{T\in \gen (L) : \Lambda_m(T)=N\} \quad \text{and}
$$
$$
\Gamma_m^L(N)/_\sim :=\{[T]\in \gen (L)/_\sim : \Lambda_m(T)=N\}.
$$
Clearly, 
\begin{equation}\label{partition-genus}
	\gen(L)/_\sim = \bigsqcup_{[N]\in\gen(\Lambda_m(L))/_\sim} \Gamma_m^L(N)/_\sim.	
\end{equation}

Let $\sigma\in O(N)$ and $T \in \Gamma_m^L(N)$. 
Since $\Lambda_p(\sigma(T))=\sigma(\Lambda_p(T))=\sigma(N)=N$, $\sigma(T)$ belongs to $\Gamma_m^L(N)$. Hence, $O(N)$ act on $\Gamma_m^L(N)$.
Moreover, by Lemma \ref{sigma-lambda-commute}, if $\tau(T)\in \Gamma_m^L(N)$ for some $\tau \in O(V)$, then $\tau\in O(N)$. Therefore, we have
\begin{equation}\label{sum-orbit}
  	\left|\Gamma_m^L(N)\right|=\sum_{[T]\in \Gamma_m^L(N)/_\sim} \frac{o(N)}{o(T)}.	
\end{equation}

Furthermore, we have the following.
\begin{lemma}\label{weight-ratio}
	Let $m$ be a prime or $4$. For any $N \in \gen(\Lambda_m(L))$
	$$
	\left|\Gamma_m^L(N)\right|=\frac{w(L)}{w(\Lambda_m(L))}.
	$$
\end{lemma}
\begin{proof}
	See \cite[Proposition 3.3]{ChanOh14}. The case when $m=4$ may be shown similarly.
\end{proof}

\section{Class number formulas for stable lattices}\label{sec:stable}
In this section, we prove our main theorem regarding the class number formula for the stable lattices.
%
%
%
%
%
%
We first describe some properties of primitive class numbers and Hurwitz class numbers.
Let $d$ be a negative integer.
Let $h(d)$ denote the number of $\SL_2(\Z)$-equivalence classes of binary quadratic forms $ax^2+bxy+cy^2$ with $b^2-4ac=d$ and $(a,b,c)=1$.
We call $h(d)$ {\em the primitive class number of discriminant $d$}.
Let $w_d$ denote the number of automorphisms of such a binary quadratic form.
Note that $w_d=6$ if $d=-3$, $w_d=4$ if $d=-4$, and $w_d=2$ otherwise.

From a well-known bijection between $\SL_2(\Z)$-equivalence classes of binary quadratic forms and Picard group of orders in the ring of integers of integers of quadratic fields, together with the relation between the orders of Picard groups, we have the following lemma.

\begin{lemma}\label{lem-primclassnum-relation}
	Let $d<0$ be a fundamental discriminant and let $f\in\N$. Then we have
	$$ 
	h(df^2) = \frac{h(d)f}{{w_{d}}/w_{df^2}}\prod_{p|f} \left(1-\legendre{d}{p}\frac{1}{p}\right).
	$$
\end{lemma}
\begin{proof}
This is a direct consequence of \cite[Theorems 7.7 and 7.24]{Cox}. See also \cite[Corollary 7.28]{Cox}.
\end{proof}

Now, we introduce Hurwitz class numbers. For $N\in\N$, the {\em Hurwitz class number} $H(N)$ is a modification of the primitive class number, which counts the number of $SL_2(\Z)$-equivalence classes of binary quadratic forms $ax^2+bxy+cy^2$ with $b^2-4ac=-N$, weighted by $2/g$ for $g$ the order of their automorphism group.
Therefore, $H(N)$ may explicitly be defined by 
\begin{equation}\label{def-Hurwitz-class-num}
	H(N)=\sum\limits_{f^{2}|N} \frac{h(-N/f^2)}{w_{-N/f^{2}}/2}.
\end{equation}
\begin{lemma}\label{lem-Hurwitz-primitive-classnum-relation}
	Let $N\in\N$ and let $d$ be the discriminant of $\Q(\sqrt{-N})$.
	Let $F\in\frac{1}{2}\N$ be such that $-N=dF^2$.
	Then we have 
	$$
	H(N) = \frac{h(d)}{{w_{d}}/2}\prod\limits_{p}\left( \frac{p^{\ord_{p}(F)+1}-1-\legendre{d}{p}(p^{\ord_{p}(F)} - 1)} {p-1}\right),
	$$
	where the product runs over all prime numbers $p$.
\end{lemma}

\begin{proof} 
	We first note that the product is indeed a finite product since $\ord_p(F)=0$ for almost all primes, in which case the value in the product is equal to $1$.
	Note also that $F\in\frac{1}{2}\N\setminus\N$ if and only if $N\equiv 1 \text{ or } 2 \pmod{4}$, and the both sides of the above equation are zero since $d\equiv 0\pmod{4}$ and $\ord_2(F)=-1$.
	Now, we assume that $N\equiv 0 \text{ or }3 \pmod{4}$, and hence $H(N)\neq 0$ and $F\in\N$.
	By \eqref{def-Hurwitz-class-num} and Lemma \ref{lem-primclassnum-relation}, we have
	$$
	H(N)=\sum\limits_{f|F} \frac{h(f^2d)}{w_{f^2d}/2}=\frac{h(d)}{{w_{d}}/2} \sum\limits_{f|F}f \prod_{p|f} \left(1-\legendre{d}{p}\frac{1}{p}\right)=\frac{h(d)}{{w_{d}}/2} \sum\limits_{f|F} \prod_{p|f}p^{\ord_p(f)} \left(1-\legendre{d}{p}\frac{1}{p}\right).
	$$
The right-hand side is multiplicative in $F$ and computing the value for $F$ a prime power yields
	$$
	\frac{h(d)}{{w_{d}}/2}\prod_{p|F}\left[1+p\left(1-\legendre{d}{p}\frac{1}{p}\right)+\cdots+p^{\ord_{p}(F)}\left(1-\legendre{d}{p}\frac{1}{p}\right)\right].
	$$
Evaluating the geometric series $1+p+\cdots+p^{\ord_p(F)}$ and $\legendre{d}{p}\left(1+p+\cdots+p^{\ord_p(F)-1}\right)$ then yields the claim.
\end{proof}

\begin{corollary}\label{cor-Hurwitz-reduce-q}
	With the same notations as the above lemma, let $\mu_q=\ord_q(N)-2\delta_{q=2}\delta_{d\equiv0\pmod{4}}$ for a prime $q$. Then we have
\begin{equation*}
	H(N) = H\left(Nq^{-2\lfloor{\mu_{q}/2}\rfloor}\right)\cdot \frac{q^{\lfloor{\mu_{q}/2}\rfloor+1}-1-\legendre{-Nq^{-2\lfloor{\mu_{q}/2}\rfloor}}{q}\left(q^{\lfloor{\mu_{q}/2}\rfloor}-1\right)}{q-1}.
\end{equation*}
\end{corollary}
\begin{proof}
	Note that $\lfloor \mu_q/2 \rfloor=\ord_q(F)$, and hence 
\[
-Nq^{-2\lfloor{\mu_{q}/2}\rfloor}=d(F\cdot q^{-ord_q(F)})^2.
\]
Applying Lemma \ref{lem-Hurwitz-primitive-classnum-relation} to obtain formulas for $H(N)$ and $H(Nq^{-2\lfloor{\mu_{q}/2}\rfloor})$, we observe that the values in the product are the same for each prime $p$ except $p=q$. Noting that $\legendre{d}{q}=\legendre{-Nq^{-2\lfloor{\mu_{q}/2}\rfloor}}{q}$, the corollary follows.
\end{proof}

\begin{proposition}\label{prop1}
	Let $L$ be a ternary lattice and let $q$ be a prime such that $(q,2d_L)=1$.
	For any $n\in\mathbb{N}$ that is locally represented by $L$, we have
	$$
	\frac{r(nq^2,gen(L))}{r(n,gen(L))}= \frac{H\pfrac{4d_Lnq^2}{f^2}}{H\pfrac{4d_Ln}{f^2}}
	$$
	for any $f\in\N$ such that $(f,q)=1$ and $H\pfrac{4d_Ln}{f^2} \neq 0$.
\end{proposition}
\begin{proof}
	Note that $nq^2$ is locally represented by $L$ since $n$ is, so neither $r(n,\gen(L))$ nor $r(nq^2,\gen(L))$ are equal to zero.
	Now, by the Minkowski--Siegel formula in Theorem \ref{Minkowski-Siegel} and by the local density formula in Lemma \ref{localdensities-stable1} for $\ord_q(d_L)=0$,
	one may verify that
	\begin{equation}\label{prop1:eq1}
		\frac{r(nq^2,gen(L))}{r(n,gen(L))} = q \cdot \frac{\alpha_q(nq^2,L)}{\alpha_q(n,L)}
		=\frac{q^{\lfloor{\nu_{q}/2\rfloor}+2}-1-\legendre{-q^{-2\lfloor{\nu_{q}/2\rfloor}}nd_L}{q}\left(q^{\lfloor{\nu_{q}/2\rfloor+1}}-1\right)}
		{q^{\lfloor{\nu_{q}/2\rfloor}+1}-1- \legendre{-q^{-2\lfloor{\nu_{q}/2\rfloor}}nd_L}{q}\left(q^{\lfloor{\nu_{q}/2\rfloor}}-1\right)},
	\end{equation}
	where $\nu_q=\ord_q(n)$.
\begin{extradetails}
\bf
The first identity in \eqref{prop1:eq1} is simply Theorem \ref{Minkowski-Siegel}. Using Lemma \ref{localdensities-stable1} and setting $n':=q^{-\nu_q}n$, for $\nu_q$ even we compute 
\begin{align*}
q\frac{\alpha_q(nq^2,L)}{\alpha_q(n,L)}&=\frac{q+1 -\frac{1}{q^{\frac{\nu_q+2}{2}}}+\left(\frac{-n'd_L}{q}\right)\frac{1}{q^{\frac{\nu_p+2}{2}}}}{1+\frac{1}{q} -\frac{1}{p^{\frac{\nu_q+2}{2}}}+\left(\frac{-n'd_L}{p}\right)\frac{1}{p^{\frac{\nu_q+2}{2}}}}\\
&=\frac{q^{\frac{\nu_q}{2}+2}+q^{\frac{\nu_q}{2}+1} -1+\left(\frac{-n'd_L}{q}\right)}{q^{\frac{\nu_q}{2}+1}+q^{\frac{\nu_q}{2}} -1+\left(\frac{-n'd_L}{q}\right)}
\end{align*}
The claimed identity is then (the $!$ indicates that it is not yet proven)
\[
\frac{q^{\frac{\nu_q}{2}+2}+q^{\frac{\nu_q}{2}+1} -1+\left(\frac{-n'd_L}{q}\right)}{q^{\frac{\nu_q}{2}+1}+q^{\frac{\nu_q}{2}} -1+\left(\frac{-n'd_L}{q}\right)}\overset{!}{=}\frac{q^{\frac{\nu_{q}}{2}+2}-1-\legendre{-n'd_L}{q}\left(q^{\frac{\nu_{q}}{2}+1}-1\right)}
		{q^{\frac{\nu_{q}}{2}+1}-1- \legendre{-n'd_L}{q}\left(q^{\frac{\nu_{q}}{2}}-1\right)},
\]
which, setting $\varepsilon:= \left(\frac{-n'd_L}{q}\right)\in \{\pm 1\}$, is equivalent to the polynomial equation
\begin{multline*}
\left(q^{\frac{\nu_q}{2}+2}+q^{\frac{\nu_q}{2}+1} -1+\varepsilon\right)\left(q^{\frac{\nu_{q}}{2}+1}-1- \varepsilon\left(q^{\frac{\nu_{q}}{2}}-1\right)\right)\\
\overset{!}{=}\left(q^{\frac{\nu_q}{2}+1}+q^{\frac{\nu_q}{2}} -1+\varepsilon\right)\left(q^{\frac{\nu_{q}}{2}+2}-1-\varepsilon\left(q^{\frac{\nu_{q}}{2}+1}-1\right)\right).
\end{multline*}
Expanding and simplifying, this becomes (noting that $\varepsilon^2=1$)
\begin{multline*}
\cancel{q^{\nu_q+3}} -\bcancel{q^{\frac{\nu_q}{2}+2}} -\xcancel{\varepsilon q^{\nu_q+2}}+\varepsilon q^{\frac{\nu_q}{2}+2}+ q^{\nu_q+2} - q^{\frac{\nu_q}{2}+1} -\varepsilon q^{\nu_q+1}+\varepsilon q^{\frac{\nu_q}{2}+1} - q^{\frac{\nu_{q}}{2}+1}\xout{+1}+ \varepsilon q^{\frac{\nu_{q}}{2}}\xout{-\varepsilon}
+\varepsilon q^{\frac{\nu_{q}}{2}+1}\xout{-\varepsilon}-q^{\frac{\nu_{q}}{2}} \xout{+ 1}\\
\overset{!}{=}\cancel{q^{\nu_q+3}}-q^{\frac{\nu_q}{2}+1}-\xcancel{\varepsilon q^{\nu_q+2}}+\varepsilon q^{\frac{\nu_q}{2}+1}+ q^{\nu_q+2} - q^{\frac{\nu_q}{2}} -\varepsilon q^{\nu_q+1}+\varepsilon q^{\frac{\nu_q}{2}} - \bcancel{q^{\frac{\nu_{q}}{2}+2}}\xout{+1}+ \varepsilon q^{\frac{\nu_{q}}{2}+1}\xout{-\varepsilon} +\varepsilon q^{\frac{\nu_{q}}{2}+2}\xout{-\varepsilon}-q^{\frac{\nu_{q}}{2}+1} \xout{+1}\\
\Leftrightarrow
\cancel{\varepsilon q^{\frac{\nu_q}{2}+2}}+ \bcancel{q^{\nu_q+2}} -\xcancel{ q^{\frac{\nu_q}{2}+1}} -\overbrace{\varepsilon q^{\nu_q+1}}^{(1)}+\overbrace{\varepsilon q^{\frac{\nu_q}{2}+1}}^{(2)} - \overbrace{q^{\frac{\nu_{q}}{2}+1}}^{(3)}+ \overbrace{\varepsilon q^{\frac{\nu_{q}}{2}}}^{(4)}+\overbrace{\varepsilon q^{\frac{\nu_{q}}{2}+1}}^{(5)}-\overbrace{q^{\frac{\nu_{q}}{2}}}^{(6)} \\
\overset{!}{=}-\xcancel{q^{\frac{\nu_q}{2}+1}}+\overbrace{\varepsilon q^{\frac{\nu_q}{2}+1}}^{(2)}+\bcancel{ q^{\nu_q+2}} - \overbrace{q^{\frac{\nu_q}{2}}}^{(6)} -\overbrace{\varepsilon q^{\nu_q+1}}^{(1)}+\overbrace{\varepsilon q^{\frac{\nu_q}{2}}}^{(4)} +\overbrace{ \varepsilon q^{\frac{\nu_{q}}{2}+1}}^{(5)} +\cancel{\varepsilon q^{\frac{\nu_{q}}{2}+2}}-\overbrace{q^{\frac{\nu_{q}}{2}+1}}^{(3)}.
\end{multline*}
This verifies the identity, and hence the claim in this case.

For $\nu_q$ odd we compute 
\begin{align*}
q\frac{\alpha_q(nq^2,L)}{\alpha_q(n,L)}&=\frac{q+1 -\frac{1}{q^{\frac{\nu_q+1}{2}}}-\frac{1}{q^{\frac{\nu_q+3}{2}}}}{1+\frac{1}{q}-\frac{1}{q^{\frac{\nu_q+1}{2}}}-\frac{1}{q^{\frac{\nu_q+3}{2}}}}\\
&=\frac{q^{\frac{\nu_q+5}{2}+2}+q^{\frac{\nu_q+3}{2}} -q-1}{q^{\frac{\nu_q+3}{2}+2}+q^{\frac{\nu_q+1}{2}} -q-1}\\
&=\frac{ (q+1) \left( q^{\frac{\nu_q+3}{2}}-1\right)}{(q+1)\left( q^{\frac{\nu_q+1}{2}}-1\right)}\\
&=\frac{q^{\frac{\nu_q+3}{2}}-1}{q^{\frac{\nu_q+1}{2}}-1} = \frac{q^{\frac{\nu_q-1}{2}+2}-1}{q^{\frac{\nu_q-1}{2}+1}-1}.
\end{align*}
Noting that the Legendre symbol in \eqref{prop1:eq1} vanishes in this case, this is precisely the claimed identity.
\rm
\end{extradetails}
	On the other hand, note that $\frac{4d_Ln}{f^2}\in\N$ and $\mu_q:=\ord_q\left(\frac{4d_Ln}{f^2}\right)=\ord_q(n)=\nu_q$ since $(q,2d_L)=(f,q)=1$.
	By applying Corollary \ref{cor-Hurwitz-reduce-q} for $N=\frac{4d_Lnq^2}{f^2}$ and $\frac{4d_Ln}{f^2}$ respectively, and by dividing the obtained equalities side-by-side,
	\begin{equation}\label{prop1:eq2}
	\frac{H\left(\frac{4d_Lnq^2}{f^2}\right)}{H\left(\frac{4d_Ln}{f^2}\right)}=
	\frac{q^{\lfloor{\nu_{q}/2\rfloor}+2}-1-\legendre{-4d_Lnq^{-2\lfloor{\nu_{q}/2\rfloor}}}{q}\left(q^{\lfloor{\nu_{q}/2\rfloor+1}}-1\right)}
	{q^{\lfloor{\nu_{q}/2\rfloor}+1}-1- \legendre{-4d_Lnq^{-2\lfloor{\nu_{q}/2\rfloor}}}{q}\left(q^{\lfloor{\nu_{q}/2\rfloor}}-1\right)}.
	\end{equation}
	Since the right hand sides of \eqref{prop1:eq1} and \eqref{prop1:eq2} are equal, the proposition follows.
\end{proof}

\begin{proposition}\label{prop2}
	Let $L$ be a ternary lattice and let $p$ be an odd prime with $ord_{p}(d_L) = 1$. 
	For any $n\in\mathbb{N}$ that is locally represented by $L$, we have
	$$
	\frac{r(np^2,gen(L))}{r(n,gen(L))}=	\frac{H \pfrac{4d_Lnp^{2}}{f^2}+S_p^\ast(L)pH\pfrac{4d_Ln}{f^2}}{H\pfrac{4d_Ln}{f^2}+S_p^\ast(L)pH\pfrac{4d_Ln}{f^2p^{2}}}
	$$
	for any $f\in\N$ such that $(f,p)=1$ and $H\pfrac{4d_Ln}{f^2} \neq 0$.
\end{proposition}
\begin{proof}
	By the Minkowski--Siegel formula in Theorem \ref{Minkowski-Siegel}, \eqref{HasseSymbol-IsoAniso}, and Lemma \ref{JordanDecomp-stable}, it suffices to show that 
	\begin{equation}\label{prop2:eq1}
		p\cdot \frac{\alpha_{p}(np^2,L)}{\alpha_{p}(n,L)}=
		\begin{cases}
			\frac{H\pfrac{4d_Lnp^{2}}{f^2}-pH\pfrac{4d_Ln}{f^2}}{H\pfrac{4d_Ln}{f^2}-pH\pfrac{4d_Ln}{f^2p^{2}}} & \text{if } M_{p}\cong \langle 1,-\Delta_{p} \rangle,\\[15pt]
			\frac{H\pfrac{4d_Lnp^{2}}{f^2}+pH\pfrac{4d_Ln}{f^2}}{H\pfrac{4d_Ln}{f^2}+pH\pfrac{4d_Ln}{f^2p^{2}}} & \text{if } M_{p}\cong \langle 1,-1\rangle, 
		\end{cases}
	\end{equation}
	where $M_p$ is a unimodular Jordan component of $L_p$. To prove \eqref{prop2:eq1}, we follow the same argument as in the proof of Proposition \ref{prop1}.
	Namely, we compute the left-hand side of \eqref{prop2:eq1} by using the formula in Lemma \ref{localdensities-stable1}, and the right-hand side of \eqref{prop2:eq1} is computed as in the proof of Proposition \ref{prop1} by applying Corollary \ref{cor-Hurwitz-reduce-q}, verifying that they are equal.

\begin{extradetails}
When $M_p\cong \langle 1,-\Delta_{p} \rangle$, then $\left(\frac{-d_{M_p}}{p}\right)=\left(\frac{\Delta_p}{p}\right)=-1$, so Lemma \ref{localdensities-stable1} implies that (writing $n'=p^{-\nu_p} n$ and $d_L'=p^{-1}d_L$)
\[
\alpha_p(n,L)=\begin{cases} p^{-\frac{\nu_p}{2}} + p^{-\frac{\nu_p}{2}-1}=p^{-\frac{\nu_p}{2}-1}(p-1)&\text{if $\nu_p$ is even,}\\
p^{-\frac{\nu_p+1}{2}} -\left(\frac{-n'd_L'}{p}\right) p^{-\frac{\nu_p+1}{2}}=p^{-\frac{\nu_p+1}{2}}\left(1-\left(\frac{-n'd_L'}{p}\right)\right) &\text{if $\nu_p$ is odd.}
\end{cases}
\]
Thus 
\[
\frac{\alpha_{p} (np^2,L)}{\alpha_{p} (n,L)}=p
\]
So the left-hand side of \eqref{prop2:eq1} equals $1$. 

For the right-hand side, we plug in Lemma \ref{lem-Hurwitz-primitive-classnum-relation} and note that every factor immediately cancels except possibly for the factor for the prime $p$. This gives 
\begin{multline*}
\frac{H\pfrac{4d_Lnp^{2}}{f^2}-pH\pfrac{4d_Ln}{f^2}}{H\pfrac{4d_Ln}{f^2}-pH\pfrac{4d_Ln}{f^2p^{2}}}\\
=\frac{p^{\left\lfloor \frac{\nu_p}{2}\right\rfloor+2}-1-\left(\frac{n'd_L'}{p}\right)\left(p^{\left\lfloor \frac{\nu_p}{2}\right\rfloor+1}-1\right) - p^{\left\lfloor \frac{\nu_p}{2}\right\rfloor+2}+p+p\left(\frac{n'd_L'}{p}\right)\left(p^{\left\lfloor \frac{\nu_p}{2}\right\rfloor}-1\right)}{p^{\left\lfloor \frac{\nu_p}{2}\right\rfloor+1}-1-\left(\frac{n'd_L'}{p}\right)\left(p^{\left\lfloor \frac{\nu_p}{2}\right\rfloor}-1\right) - p^{\left\lfloor \frac{\nu_p}{2}\right\rfloor+1}+p+p\left(\frac{n'd_L'}{p}\right)\left(p^{\left\lfloor \frac{\nu_p}{2}\right\rfloor-1}-1\right)}\\
=\frac{p-1+(1-p)\left(\frac{n'd_L'}{p}\right)}{p-1+(1-p)\left(\frac{n'd_L'}{p}\right)}=1.
\end{multline*}
So the claim follows for $M_p\cong \langle 1,-\Delta_{p} \rangle$.
\end{extradetails}
	When $M_p\cong \langle 1,-\Delta_{p} \rangle$, one may easily show that both sides of \eqref{prop2:eq1} are equal to 1, so we only provide a detailed calculation for the case when $M_p\cong \langle 1,-1\rangle$. In this case,  $\legendre{-d_{M_p}}{p}=1$. Setting $\nu_p:=\ord_p(n)$, Lemma \ref{localdensities-stable1} implies that
	$$
	\alpha_{p}(n,L)=
	\begin{cases}
		2-\frac{1}{p^{\lceil{\nu_{p}/2}\rceil}}-\frac{1}{p^{\lceil{\nu_{p}/2}\rceil+1}}& \text{if}\  \nu_{p} \equiv 0 \pmod 2,\\
		
		2 - \frac{1}{p^{\lceil{\nu_{p}/2}\rceil}} + \varepsilon_p(n)\frac{1}{p^{\lceil{\nu_{p}/2}\rceil}} & \text{if}\ \nu_{p} \equiv 1 \pmod 2,
	\end{cases}
	$$
	where $\varepsilon_p(n):=\legendre{-nd_Lp^{-{(\ord_p(n)+1)}}}{p}$. Thus, we have
	\begin{equation}\label{prop2:eq2}
		p\cdot \frac{\alpha_{p}(np^2,L)}{\alpha_{p}(n,L)}=
		\begin{cases}
		\frac{2p^{\lceil{\nu_{p}/2}\rceil+2}-p-1}{2p^{\lceil{\nu_{p}/2}\rceil+1}-p-1}
		& \text{if }  \nu_{p} \equiv 0 \pmod 2,\\[10pt]
		
		\frac{2p^{\lceil{\nu_{p}/2}\rceil+1}-1+\varepsilon_p(np^2)}{2p^{\lceil{\nu_{p}/2}\rceil}-1+\varepsilon_p(n)}
		& \text{if}\ \nu_{p} \equiv 1 \pmod 2.
		\end{cases}
	\end{equation}
\begin{extradetails}
Plugging in the previous displayed formula, we directly have
\begin{align*}
p\cdot \frac{\alpha_{p}(np^2,L)}{\alpha_{p}(n,L)}&=	\begin{cases}
\frac{2p-\frac{1}{p^{\lceil{\nu_{p}/2}\rceil}}-\frac{1}{p^{\lceil{\nu_{p}/2}\rceil+1}}}{2-\frac{1}{p^{\lceil{\nu_{p}/2}\rceil}}-\frac{1}{p^{\lceil{\nu_{p}/2}\rceil+1}}}
& \text{if}\  \nu_{p} \equiv 0 \pmod 2,\\
\frac{2p - \frac{1}{p^{\lceil{\nu_{p}/2}\rceil}} + \varepsilon_p(n)\frac{1}{p^{\lceil{\nu_{p}/2}\rceil}}}{2 - \frac{1}{p^{\lceil{\nu_{p}/2}\rceil}} + \varepsilon_p(n)\frac{1}{p^{\lceil{\nu_{p}/2}\rceil}}} & \text{if}\ \nu_{p} \equiv 1 \pmod 2,
	\end{cases}\\
&=		\begin{cases}
		\frac{2p^{\lceil{\nu_{p}/2}\rceil+2}-p-1}{2p^{\lceil{\nu_{p}/2}\rceil+1}-p-1}
		& \text{if }  \nu_{p} \equiv 0 \pmod 2,\\[10pt]
	\frac{2p^{\lceil{\nu_{p}/2}\rceil+1}-1+\varepsilon_p(np^2)}{2p^{\lceil{\nu_{p}/2}\rceil}-1+\varepsilon_p(n)}
		& \text{if}\ \nu_{p} \equiv 1 \pmod 2.
		\end{cases}
\end{align*}
\end{extradetails}
On the other hand, note that $\frac{4d_Ln}{f^2}\in\N$ and 
\[
\mu_p:=\ord_p\left(\frac{4d_Ln}{f^2}\right)=\ord_p(n)+1=\nu_p+1.
\]
Hence $\lfloor \mu_p/2 \rfloor = \lceil \nu_p/2 \rceil$. 
By applying Corollary \ref{cor-Hurwitz-reduce-q} for $N=\frac{4d_Ln}{f^2}$ and $\frac{4d_Ln}{f^2p^2}$ when $\nu_p\ge1$, and noting that $H\pfrac{4d_Ln}{f^2p^2}=0$ when $\nu_p=0$, we have
\begin{align}
\label{eqn:H+H/p2}
 H\pfrac{4d_Ln}{f^2}&+pH\pfrac{4d_Ln}{f^2p^2}\\
\nonumber
&=H\left(\frac{4d_Ln}{f^2} p^{-2\lceil \nu_{p}/2\rceil}\right)
\times
\begin{cases}
	\frac{2p^{\lceil{\nu_{p}/2\rceil}+1}-p-1}{p-1} & \text{if } \nu_p \equiv 0 \pmod {2},\\[5pt] 
	2p^{\lceil\nu_p/2\rceil} & \text{if } \nu_p \equiv 1 \pmod {2} \text{ and } \varepsilon_p(n)=1,\\[5pt]
	2(p+1)\frac{p^{\lceil\nu_p/2\rceil}-1}{p-1} & \text{if } \nu_p \equiv 1 \pmod {2} \text{ and } \varepsilon_p(n)=-1.\\
\end{cases}
\end{align}
Plugging in \eqref{eqn:H+H/p2} with $n\to np^2$ in the numerator and $n$ in the denominator, we obtain
\begin{equation}\label{prop2:eq3}
\frac{H\pfrac{4d_Lnp^{2}}{f^2}+pH\pfrac{4d_Ln}{f^2}}{H\pfrac{4d_Ln}{f^2}+pH\pfrac{4d_Ln}{f^2p^{2}}}
=\begin{cases}
	\frac{2p^{\lceil{\nu_{p}/2\rceil}+2}-p-1}{2p^{\lceil{\nu_{p}/2\rceil}+1}-p-1} & \text{if } \nu_p \equiv 0 \pmod {2},\\[5pt]
	p & \text{if } \nu_p \equiv 1 \pmod {2} \text{ and } \varepsilon_p(n)=1,\\[5pt]
        \frac{p^{\lceil\nu_{p}/2\rceil+1}-1}{p^{\lceil\nu_{p}/2\rceil}-1} & \text{if } \nu_p \equiv 1 \pmod {2} \text{ and } \varepsilon_p(n)=-1.	
\end{cases}
\end{equation}
Noting that $\varepsilon_p(np^2)=\varepsilon_p(n)=\pm1$, one may verify that the right-hand sides of \eqref{prop2:eq2} and \eqref{prop2:eq3} are equal.
This proves the proposition.
\end{proof}
We have a similar result for the prime $2$. 
\begin{proposition}\label{prop3}
	Let $L$ be a ternary lattice that is stable at $2$. For any $n\in\N$ that is locally represented by $L$, we have
	$$
	\frac{r(4n,gen(L))}{r(n,gen(L))}=
	\begin{cases}
		\frac{H\left(\frac{16d_Ln}{f^{2}}\right)+S_2^\ast(L)2H\left(\frac{4d_Ln}{f^{2}}\right)}{H\left(\frac{4d_Ln}{f^{2}}\right)+S_2^\ast(L)2H\left(\frac{d_Ln}{f^{2}}\right)} & \text{if } L \text{ is odd},\\[15pt]
		\frac{H\pfrac{4d_Ln}{f^2}}{H\pfrac{d_Ln}{f^2}}  & \text{if } L \text{ is even}
	\end{cases}
	$$
	for any $f\in \mathbb{N}$ such that $(f,2) =1$ and $H\pfrac{4d_Ln}{f^2}\neq0$ (resp. $H\pfrac{d_Ln}{f^2}\neq0$) if $L$ is odd (resp. if $L$ is even).
\end{proposition}
\begin{proof}
	This is a consequence of Lemma \ref{localdensities-stable2} and Corollary \ref{cor-Hurwitz-reduce-q} by following the same argument as in the proof of Propositions \ref{prop1} and \ref{prop2}.
	Since the proofs are quite similar to each other, we only provide the proof for the case when $L$ is even.
	
	Assume that $L$ is even. Note that $\ord_2(d_L)=1$ and $n$ is an even integer since it is locally represented by $L$.
	Let us write $n=\beta\cdot 2^a$ with $\beta\equiv 1\pmod{2}$ and  $a\ge1$.
	By the Minkowski--Siegel formula in Theorem \ref{Minkowski-Siegel} and by Lemma \ref{localdensities-stable2}, we have 
	\begin{equation}\label{prop3:eq1}
		\frac{r(4n,gen(L))}{r(n,gen(L))}=2\cdot\frac{\a_2(4n,L)}{\a_2(n,L)}=
		\begin{cases}
			\frac{2^{(a+2)/2}-1}{2^{a/2}-1} & \text{if } a \text{ is even with }a\ge2,\\[5pt]
			\frac{2^{(a+3)/2}-1}{2^{(a+1)/2}-1} & \text{if } a \text{ is odd, } \beta\equiv d_L/2\pmod{4},\\[5pt]
			\frac{3\cdot 2^{(a+1)/2}-1}{3\cdot2^{(a-1)/2}-1} & \text{if } a \text{ is odd, } \beta\equiv 3d_L/2\pmod{8},\\[5pt]
			2 & \text{if } a \text{ is odd, } \beta\equiv 7d_L/2\pmod{8}.\\			
		\end{cases}
	\end{equation}
	On the other hand, we shall apply Corollary \ref{cor-Hurwitz-reduce-q} with $N=\frac{d_Ln}{f^2}\in\N$ and $q=2$. Note that 
	\begin{equation}\label{prop3:cond1}
	\mu_2:=\ord_2\pfrac{d_Ln}{f^2}-2\delta_{q=2}\delta_{d\equiv0\pmod{4}}=
	\begin{cases}
		a+1 & \text{if } \beta \equiv 3d_L/2 \pmod{4},\\
		a-1 & \text{otherwise},
	\end{cases}
	\end{equation}
	where $d$ is the discriminant of $\Q\left(\sqrt{-d_Ln/f^2}\right)$. Noting that $\legendre{d}{2}=\legendre{-N\cdot 2^{-2\lfloor \mu_2/2\rfloor} }{2}$,	Corollary \ref{cor-Hurwitz-reduce-q} implies
	\begin{equation}\label{eqn:HdLn/f2}
	H\pfrac{d_Ln}{f^2}=H\left(\frac{d_Ln}{f^2}\cdot 2^{-2\lfloor{\mu_{2}/2}\rfloor}\right)\cdot \left(2^{\lfloor{\mu_{2}/2}\rfloor+1}-1-\legendre{d}{2}\left(2^{\lfloor{\mu_{q}/2}\rfloor}-1\right)\right).
	\end{equation}
Plugging in \eqref{eqn:HdLn/f2}  (with $n\mapsto 4n$ in the numerator below and $n$ in the denominator) yields
	\begin{equation}\label{prop3:eq2}
	\frac{H\pfrac{4d_Ln}{f^2}}{H\pfrac{d_Ln}{f^2}}=
	\frac{2^{\lfloor{\mu_{2}/2}\rfloor+2}-1-\legendre{d}{2}\left(2^{\lfloor{\mu_{q}/2}\rfloor+1}-1\right)}
	     {2^{\lfloor{\mu_{2}/2}\rfloor+1}-1-\legendre{d}{2}\left(2^{\lfloor{\mu_{q}/2}\rfloor}-1\right)}.
	\end{equation}
	By using \eqref{prop3:cond1} and by noting that 
	\begin{equation*}
		\begin{cases}
			d\equiv 5 \pmod{8} & \text{if $a$ is odd and }\beta\equiv 3d_L/2\pmod{8},\\ 
			d\equiv 1 \pmod{8} & \text{if $a$ is odd and }\beta\equiv 7d_L/2\pmod{8},\\ 
			d\equiv 0 \pmod{4} &\text{otherwise},
		\end{cases}
	\end{equation*}
	one may verify through a case-by-case computation that the right-hand sides of \eqref{prop3:eq1} and \eqref{prop3:eq2} coincide.
	This proves the proposition.
	\begin{extradetails}
		Assume that $L$ is odd. Since $L$ is stable at $2$, we know that $\ord_2(d_L)=0$. 
		Let $d<0$ be the fundamental discriminant for $Q(\sqrt{-4d_Ln})$ and let $F\in \frac{1}{2}\mathbb{N}$ with $ord_{2}(F)\geq -1$ be defined by $\frac{-4d_Ln}{f^{2}} = d F^{2}$. We furthermore write $F_{2} := 2^{-\ord_2(F)}F$ and $n = \beta \times 2^{a}$.
		
		\begin{enumerate}
			\item When $L_2$ is odd and anisotropic, we have $S_2^*(L)=-1$, and Corollary \ref{cor-Hurwitz-reduce-q}  implies that 
			\[
			\frac{H\left(\frac{16d_Ln}{f^{2}}\right)+ S_2^*(L)2H\left(\frac{4d_Ln}{f^{2}}\right)}{H\left(\frac{4d_Ln}{f^{2}}\right)+S_{2}^*(L)2H\left(\frac{d_Ln}{f^{2}}\right)}=\frac{H\left(\frac{16d_Ln}{f^{2}}\right)-2H\left(\frac{4d_Ln}{f^{2}}\right)}{H\left(\frac{4d_Ln}{f^{2}}\right)-2H\left(\frac{d_Ln}{f^{2}}\right)} = \frac{H(|d|\times F_{2}^{2})\left(1-\legendre{d}{2}\right)}{H(|d|\times F_{2}^{2})\left(1-\legendre{d}{2}\right)} = 1.
			\]
By Lemma \ref{localdensities-stable2}, we have
\begin{align*}
			\a_2(4n,L)&=\begin{cases}
				3\cdot 2^{-\frac{a+3}{2}} & \text{if } a \text{ is odd},\\
				3\cdot 2^{-\frac{a+4}{2}} & \text{if } a \text{ is even and } \beta\equiv d_L \pmod{4},\\
				2^{-\frac{a+2}{2}} & \text{if } a \text{ is even and } \beta\equiv 3d_L \pmod{8},\\
				0 & \text{if } a \text{ is even and } \beta\equiv 7d_L \pmod{8}.\\
			\end{cases}\\
			\a_2(n,L)&=\begin{cases}
				3\cdot 2^{-\frac{a+1}{2}} & \text{if } a \text{ is odd},\\
				3\cdot 2^{-\frac{a+2}{2}} & \text{if } a \text{ is even and } \beta\equiv d_L \pmod{4},\\
				2^{-\frac{a}{2}} & \text{if } a \text{ is even and } \beta\equiv 3d_L \pmod{8},\\
				0 & \text{if } a \text{ is even and } \beta\equiv 7d_L \pmod{8}.\\
			\end{cases}
			\end{align*}
			Hence 
			\[
			2\cdot \frac{\alpha_2(4n,L)}{\alpha_2(n,L)}=1=			\frac{H\left(\frac{16dLn}{f^{2}}\right)-2H\left(\frac{4dLn}{f^{2}}\right)}{H\left(\frac{4dLn}{f^{2}}\right)-2H\left(\frac{dLn}{f^{2}}\right)}
			\].
			
			\item When $L_2$ is odd and isotropic, we have $S_2^*(L)=1$, so Corollary \ref{cor-Hurwitz-reduce-q} yields 
			\begin{multline*}
			\frac{H\left(\frac{16d_Ln}{f^{2}}\right)+S_2^*(L)2H\left(\frac{4d_Ln}{f^{2}}\right)}{H\left(\frac{4d_Ln}{f^{2}}\right)+S_2^*(L)2H\left(\frac{d_Ln}{f^{2}}\right)}= \frac{H\left(\frac{16d_Ln}{f^{2}}\right)+2H\left(\frac{4d_Ln}{f^{2}}\right)}{H\left(\frac{4d_Ln}{f^{2}}\right)+2H\left(\frac{d_Ln}{f^{2}}\right)}\\
 = \frac{2\times2^{ord_{2}(F)+2}-3-\legendre{d}{2}(2\times2^{ord_{2}(F)+1}-3)}{2\times2^{ord_{2}(F)+1}-3-\legendre{d}{2}(2\times2^{ord_{2}(F)}-3)}\\
			=\begin{cases}
				\frac{2\cdot2^{\lfloor{a/2\rfloor}+2}-3}{2\cdot2^{\lfloor{a/2\rfloor}+1}-3} & \text{if $a$ is odd or $a$ is even with $\beta \equiv d_L \pmod{4}$, $(\implies d \equiv 0\pmod{4})$,}\\
				\frac{2^{\lfloor{a/2\rfloor}+3}-2}{2^{\rfloor{a/2\rfloor}+2}-2} & \text{if $a$ even with $\beta \equiv 3d_L \pmod{8}$, $\left(\implies \legendre{d}{2} = \legendre{-3}{2} = -1\right)$},\\
				\frac{2^{a/2+3}}{2^{a/2+2}} = 2 & \text{if $a$ even with $\beta \equiv 7d_L \pmod{8}$, $\left(\implies \legendre{d}{2} = \legendre{-7}{2} = 1\right)$}.\\
			\end{cases}
			\end{multline*}
		\end{enumerate}
To compare, by Lemma \ref{localdensities-stable2} we have
\begin{align*}
\a_2(4n,L)&=\begin{cases}
			2-3\cdot 2^{-(\frac{a+3}{2})} & \text{if } a \text{ is odd},\\
			2-3\cdot 2^{-(\frac{a+4}{2})} & \text{if } a \text{ is even and } \beta\equiv d_L \pmod{4},\\
			2-2^{-(\frac{a+2}{2})} & \text{if } a \text{ is even and } \beta\equiv 3d_L \pmod{8},\\
			2 & \text{if } a \text{ is even and } \beta\equiv 7d_L \pmod{8}.\\
		\end{cases}\\
\a_2(n,L)&=\begin{cases}
			2-3\cdot 2^{-(\frac{a+1}{2})} & \text{if } a \text{ is odd},\\
			2-3\cdot 2^{-(\frac{a+2}{2})} & \text{if } a \text{ is even and } \beta\equiv d_L \pmod{4},\\
			2-2^{-(\frac{a}{2})} & \text{if } a \text{ is even and } \beta\equiv 3d_L \pmod{8},\\
			2 & \text{if } a \text{ is even and } \beta\equiv 7d_L \pmod{8}.\\
		\end{cases}
\end{align*}
		
		\begin{enumerate}
			\item When $a$ is odd, we simplify (note that $\left\lfloor\frac{a}{2}\right\rfloor=\frac{a-1}{2}$)
\[
2\times \frac{\a_2(4n,L)}{\a_{2}(n,L)} = 2\times\left(\frac{2-3\cdot 2^{-(\frac{a+3}{2})}}{2-3\cdot 2^{-(\frac{a+1}{2})}}\right)=
			2\times\left(\frac{2^{\frac{a+1}{2}}}{2^{\frac{a+3}{2}}}\right) \left(\frac{2\cdot2^{\frac{a+3}{2}}-3}{2\cdot2^{\frac{a+1}{2}}-3}\right) = \frac{2\cdot2^{\lfloor{a/2\rfloor}+2}-3}{2\cdot2^{\lfloor{a/2\rfloor}+1}-3}.
\]
			
			\item When $a$ is even and $\beta\equiv d_L\pmod{4}$, we may simplify
\[2\times \frac{\a_2(4n,L)}{\a_{2}(n,L)} = 2\times\left(\frac{2-3\cdot 2^{-(\frac{a+4}{2})}}{2-3\cdot 2^{-(\frac{a+2}{2})}}\right)=
			2\times\left(\frac{2^{\frac{a+2}{2}}}{2^{\frac{a+4}{2}}}\right) \left(\frac{2\cdot2^{\frac{a+4}{2}}-3}{2\cdot2^{\frac{a+2}{2}}-3}\right) = \frac{2\cdot2^{\lfloor{a/2\rfloor}+2}-3}{2\cdot2^{\lfloor{a/2\rfloor}+1}-3}.
\]
			
			\item When $a$ is even and $\beta\equiv 3d_L\pmod{8}$, we have
\[
2\times \frac{\a_2(4n,L)}{\a_{2}(n,L)}= 2\times\left(\frac{2^{\frac{a}{2}}}{2^{\frac{a+2}{2}}}\right) \left(\frac{2\cdot2^{\frac{a+2}{2}}-1}{2\cdot2^{\frac{a}{2}}-1}\right)= \left(\frac{2\cdot2^{\frac{a}{2}+1}-1}{2\cdot2^{\frac{a}{2}}-1}\right) = \frac{2^{\lfloor{a/2\rfloor}+3}-2}{2^{\lfloor{a/2\rfloor}+2}-2}.
\]
			
			\item When $a$ is even and $\beta\equiv 7 d_L\pmod{8}$, we have
\[			
2\times \frac{\a_2(4n,L)}{\a_{2}(n,L)}= 2\times\frac{2}{2} = 2.
\]
		\end{enumerate}
These all match the calculation above, verifying the claim in this case.
\end{extradetails}
\end{proof}

In a somewhat orthogonal direction to Jones's \cite[Theorem 86]{Jones}, we next relate $r(n,\gen(L))$ to class numbers whenever $\ord_p(n)\leq 1$ for $p\nmid 2d_L$.
\begin{lemma}\label{ternary-generic-theta}
	Let $L$ be a ternary lattice. For any $n\in\N$ with $\ord_p(n)\le 1$ for all $p\nmid 2d_L$,
	\begin{equation}
		\begin{aligned}
			r(n,\gen(L))&=\frac{12}{\pi}\frac{\sqrt{n}}{\sqrt{d_L}} L(1,\chi_{-4d_Ln}) \prod_{p\mid 2d_L} \a_p(n,L)\left( 1-\frac{1}{p^2}\right)^{-1} \\
			&=\frac{12\sqrt{n}}{\sqrt{d_L}\sqrt{|d|}} H(|d|) \prod_{p\mid 2d_L} \a_p(n,L)\left(1-\legendre{d}{p}\frac{1}{p}\right)\left( 1-\frac{1}{p^2}\right)^{-1},
		\end{aligned}
	\end{equation}
	where $d$ is the discriminant of $\Q\left(\sqrt{-4d_Ln}\right)$, $\chi_a=\legendre{a}{\cdot}$, and $L(s,\chi_a)$ is a Dirichlet $L$-function.
\end{lemma}
\begin{proof}
	By the Minkowski--Siegel formula in Theorem \ref{Minkowski-Siegel}, we have
	\begin{equation}\label{ternary-generic-theta:eq1}
		r(n,\gen(L))=2\pi \sqrt{\frac{n}{d_L}} \prod_p \a_p(n,L)=2\pi \sqrt{\frac{n}{d_L}} \prod_{p\mid 2d_L} \a_p(n,L) \prod_{p\nmid2d_L} \a_p(n,L).		
	\end{equation}
	Note that for any prime $p\nmid 2d_L$, by \cite[Theorem 3.1]{Yang}, we have
	$$
	\a_p(n,L)=\begin{cases}
		1+\legendre{-nd_L}{p}\frac{1}{p} & \text{if } p\nmid n,\\
		1-\frac{1}{p^2} & \text{if } p\mid n.
	\end{cases}
	$$
\begin{extradetails}
We use \eqref{eqn:alphapnLord0} with $\nu_p\leq 1$ (note that the fact that $L$ was stable at $p$ was not used in this calculation). By \eqref{eqn:alphapnLord0}, we have 
\[
\alpha_{p}(n,L)=1+p^{-1} +\begin{cases}
 -p^{\frac{\nu_p}{2}-1}+ \left(\frac{-n'd_L}{p}\right) p^{-\frac{\nu_p}{2}-1}&\text{if }\nu_{p}\text{ is even},\\
-p^{-\frac{\nu_p+1}{2}}-p^{-\frac{\nu_p+3}{2}}&\text{if }\nu_p\text{ is odd}.
\end{cases}
\]
If $\nu_p=0$, then we are in the first case and 
\[
\alpha_{p}(n,L)=1+p^{-1}-p^{-1} + \left(\frac{-n'd_L}{p}\right) p^{-1}=1+\left(\frac{-n'd_L}{p}\right) p^{-1},
\]
as claimed. If $\nu_p=1$, then we are in the second case and 
\[
\alpha_{p}(n,L)=1+p^{-1}-p^{-1} - p^{-2}=1-p^{-2},
\]
as claimed.
\end{extradetails}
	Thus, using the difference of squares identity $(1+\chi(p) p^{-1})(1-\chi(p)p^{-1})=1-p^2$ for $\chi(p)=\pm 1$ and noting that $\legendre{-4d_Ln}{p}=0$ for $p\mid 2d_Ln$, we have
	$$
	\begin{aligned}
		\prod_{p\nmid2d_L} \a_p(n,L) &= \prod_{\substack{p\nmid2d_L \\ p\nmid n}} \left(1+\legendre{-nd_L}{p}\frac{1}{p} \right) \cdot  \prod_{\substack{p\nmid2d_L\\ p\mid n}} \left(1-\frac{1}{p^2} \right)\\	
		&= \prod_{p} \left(1-\legendre{-4nd_L}{p}\frac{1}{p} \right)^{-1} \cdot  \prod_{p\nmid2d_L} \left(1-\frac{1}{p^2} \right)\\	
		&=L(1,\chi_{-4d_Ln}) \prod_{p\nmid2d_L} \left(1-\frac{1}{p^2}\right)  =L(1,\chi_{-4d_Ln}) \cdot \zeta(2)^{-1} \prod_{p\mid2d_L} \left(1-\frac{1}{p^2}\right)^{-1},
	\end{aligned}
	$$
	where $\zeta(2)=\frac{\pi^2}{6}$ is the value of Riemann zeta function at $2$. Plugging the above equation into \eqref{ternary-generic-theta:eq1}, we obtain the first equality of the lemma.
	To obtain the second equality of the lemma, we note our assumption that $\ord_p(n)\leq 1$ for $p\nmid 2d_L$ implies that for $-4d_Ln=dF^2$ every prime dividing $F$ must also divide $2d_L$. Hence we may rewrite the $L$-value and use Dirichlet's class number formula to obtain
	$$
	L(1,\chi_{-4d_Ln})=L(1,\chi_{d}) \cdot \prod_{p\mid 2d_L} \left(1-\legendre{d}{p} \frac{1}{p}\right)=\frac{2\pi  h(d)}{w_d\sqrt{|d|}} \cdot \prod_{p\mid 2d_L} \left(1-\legendre{d}{p}\frac{1}{p}\right).
	$$
\begin{extradetails}
We have 
\[
	L(1,\chi_{-4d_Ln})=\prod_{p} \left(1-\chi_{-4d_Ln}(p)p^{-1}\right)^{-1}=\prod_{p\nmid 2d_Ln} \left(1-\chi_{dF^2}(p)p^{-1}\right)^{-1}=\prod_{p\nmid 2d_L} \left(1-\chi_{d}(p)p^{-1}\right)^{-1},
\]
where in the last step we used the fact that for $p\nmid 2d_L$ we have $p\nmid F$, so $\chi_{dF^2}(p)=\chi_d(p)$. We now rewrite 
\[
\prod_{p\nmid 2d_L} \left(1-\chi_{d}(p)p^{-1}\right)^{-1}=\prod_{p} \left(1-\chi_{d}(p)p^{-1}\right)^{-1}\prod_{p\mid 2d_L} \left(1-\chi_{d}(p)p^{-1}\right)=L\left(1,\chi_d\right)\prod_{p\mid 2d_L} \left(1-\chi_{d}(p)p^{-1}\right).
\]
This is the first identity. Since $d$ is a fundamental discriminant, we may now use the class number formula (cf. (5) on \url{https://mathworld.wolfram.com/ClassNumber.html}) to obtain 
\[
L(1,\chi_d)=\frac{2\pi h(d)}{w_d\sqrt{|d|}},
\]
giving the second identity.
\end{extradetails}
Since $d$ is fundamental, we have $H(|d|)=\frac{h(d)}{w_d/2}$, completing the proof.
\end{proof}

For a ternary lattice $L$, let $\mathcal{P}$ be a set of primes defined by
\begin{equation}\label{defn-mathcalP}
\mathcal{P}=\mathcal{P}_L:=\{p :\text{prime} \mid  L \text{ is stable at } 2 \text{ if } p=2, \text{ and } \ord_p(d_L)=1 \text{ otherwise}\}.	
\end{equation}
Note that for any $p\in\mathcal{P}$ we have $p\mid 2d_L$ . Also, let us define
\begin{equation}\label{defn-mathfrakP}
\mathfrak{P}=\mathfrak{P}_L:=2^{-\delta_{2\in \mathcal{P}}\cdot \ord_2(d_L)}\prod_{p\in\mathcal{P}}p.	
\end{equation}
Note that $\mathfrak{P}$ is a squarefree positive integer, and $\mathfrak{P}=2^{1-2\ord_2(d_L)} d_L$ if $L$ is a stable lattice, and hence our definition of $\mathfrak{P}$ agrees with the definition in the introduction for stable lattices. We are now ready to prove Theorem \ref{mainthm}.

\begin{proof}[Proof of Theorem \ref{mainthm}]
By Duke and Schulze-Pillot's formula \eqref{eqn:DukeSchulzePillot}, it suffices to show that 
\begin{equation}\label{formula-stable}
r(n,\gen(L))= 12\prod_{p\mid \mathfrak{P}} \left(p+S_p^\ast(L)\right)^{-1} \times \sum_{f\mid \mathfrak{P}} \left[\prod_{p\mid f} S_p^\ast(L) \right] \cdot f \cdot H\pfrac{ed_Ln}{f^2},
\end{equation}
where $e=4$ if $L$ is odd, and $e=1$ if $L$ is even.

We first prove the theorem when $L$ is odd, in which case $\mathfrak{P}=2d_L$ and $e=4$.
We prove that \eqref{formula-stable} holds in two steps. We first show that if \eqref{formula-stable} holds for a positive integer $n$, then it also holds for $nq^2$ for any prime $q$, reducing the claim to showing that \eqref{formula-stable} holds for any squarefree positive integer $n$.
To show the claim, let $n\in \N$ and let $q$ be a prime. First of all, recall that $r(n,\gen(L))\neq0$ if and only if $n$ is locally represented by $L$. Hence, since $Q(q\bm{x})=q^2Q(\bm{x})$, if $r(n,\gen(L))\neq 0$, then $r(nq^2,\gen(L))\neq 0$. Conversely, if $n$ is not represented by $L$, then there exists a prime $p$ such that $n$ is not represented by $L_p$ over $\Z_p$. By Corollary \ref{cor-local-rep-num}, $L_p$ should be anisotropic (hence $p\mid 2d_L$), and $n$ is not represented by $L_p$ if and only if $nq^2$ is not represented by $L_p$. Thus 
\begin{center}
	$r(n,\gen(L))\neq0$ if and only if $r(nq^2,\gen(L))\neq0$.
\end{center}

Now we assume that $r(n,\gen(L))\neq 0$. If $q \nmid 2d_L$, then $(q,\mathfrak{P})=1$. Hence by Proposition \ref{prop1}, for any $f\mid \mathfrak{P}$, we have 
$$
\frac{r(nq^2,\gen(L))}{r(n,\gen(L))} = \frac{H\pfrac{4d_Lnq^2}{f^2}}{H\pfrac{4d_Ln}{f^2}}.
$$
Rearranging and plugging in \eqref{formula-stable} for $n$ yields \eqref{formula-stable} for $nq^2$.

On the other hand, if $q\mid 2d_L$, then since $\mathfrak{P}$ is squarefree and $q\mid \mathfrak{P}$, the sum in the right hand side of \eqref{formula-stable} may be written as
\begin{equation}\label{mainthm:eqq}
\sum_{f\mid (\mathfrak{P}/q)} \left[\prod_{p\mid f} S_p^\ast(L) \right] \cdot f \cdot \left( H\pfrac{4d_Ln}{f^2} +S_q^\ast(L)qH\pfrac{4d_Ln}{f^2q^2} \right).
\end{equation}
Thus, noting that $(f,q)=1$ for any $f\mid (\mathfrak{P}/q)$, the claim follows by Propositions \ref{prop2} and \ref{prop3} in a similar manner to the case when $q\nmid 2d_L$.

Now assume that $r(n,\gen(L))=0$. We will show that the right hand side of \eqref{formula-stable} is equal to zero in this case.
Note that there is a prime $q_0$ such that $n$ is not represented by $L_{q_0}$.
By Lemma \ref{JordanDecomp-stable} and Corollary \ref{cor-local-rep-num}, $L_{q_0}$ is anisotropic, $S_{q_0}^\ast(L)=-1$, $n\in (-d_L)\Z_{q_0}^2$, and $q_0\mid\mathfrak{P}$.
By Corollary \ref{cor-Hurwitz-reduce-q} if $H\pfrac{4d_Ln}{f^2}\neq0$ (i.e., $\frac{4d_Ln}{f^2}\in\N)$, then we conclude that $H\pfrac{4d_Ln}{f^2} +S_{q_0}^\ast(L)q_0H\pfrac{4d_Ln}{f^2q_0^2}$ is equal to
$$
H\left(\frac{4d_Ln}{f^2}q_0^{-2a}\right)\cdot 
\left[\frac{q_0^{a+1}-1-\legendre{-4d_Lnq_0^{-2a}}{q_0}\left(q_0^{a}-1\right)}{q_0-1}
-q_0\frac{q_0^{a}-1-\legendre{-4d_Lnq_0^{-2a}}{q_0}\left(q^{a-1}-1\right)}{q_0-1}\right]
$$
where $a=\lfloor \mu_{q_0}/2\rfloor$, $\mu_{q_0}:=\ord_{q_0}(\frac{4d_Ln}{f^2})-2\delta_{q_0=2}\delta_{d\equiv0\pmod{4}}$ with $d$ the discriminant of $Q\left(\sqrt{\frac{4d_Ln}{f^2}}\right)$.Since $n\in (-d_L)\Z_{q_0}^2$, we have $\legendre{-4d_Lnq_0^{-2a}}{q_0}=1$, hence the above is equal to zero.
\begin{extradetails}
\bf
Since $n\in (-d_L)\Z_{q_0}^2$, we see that $\ord_{q_0}(4d_Ln)-2\delta_{q_0=2}\delta_{d\equiv 0\pmod{4}} =2a$ is even. Thus for $q_0$ odd we have  $q_0\nmid \frac{4d_Ln}{f^2}$, and hence 
\[
\legendre{-4d_Lnq_0^{-2a}}{q_0}=1.
\]

 If $q_0=2$, then to show that 
\[
\legendre{-4d_Lnq_0^{-2a}}{q_0}=1
\]
we need to verify that $d\not\equiv 0\pmod{4}$. However, since $n\in (-d_L)\Z_{2}^2$, we see that $-4d_Ln\in \Z_2^2$, so the odd part is congruent to $1$ modulo $4$ and hence the corresponding fundamental discriminant is odd.

Plugging in $\legendre{-4d_Lnq_0^{-2a}}{q_0}=1$, the factor on the inside becomes
\[
\left[\frac{q_0^{a+1}-q_0^{a}}{q_0-1}-\frac{q_0^{a+1}-q^{a}}{q_0-1}\right]=0.
\]
\rm
\end{extradetails}
Thus, from the viewpoint of \eqref{mainthm:eqq} with $q=q_0$, the right hand side of \eqref{formula-stable} is zero as claimed.

We have now reduced the proof of \eqref{formula-stable} to the case that $n$ is a squarefree positive integer which is represented by the genus. Then by Lemma \ref{ternary-generic-theta}, we have
\begin{equation}\label{mainthm:eq1}
r(n,\gen(L))=\frac{12\sqrt{n}}{\sqrt{d_L}\sqrt{|d|}} H(|d|) \prod_{p\mid 2d_L} c_p(n,L),	
\end{equation}
where $d$ is the discriminant of $\Q(\sqrt{-4d_Ln})$ and $c_p(n,L):=\a_p(n,L)\left(1-\legendre{d}{p}\frac{1}{p}\right)\left( 1-\frac{1}{p^2}\right)^{-1}$.

Let $F$ be the positive integer defined by $-4d_Ln=dF^2$ and let $F_p:=p^{\ord_p{F}}$ for each prime $p$.
Note that since we are assuming $n$ to be squarefree and $L$ is stable, $F$ is squarefree. Indeed, for any prime $p\nmid 2d_L$, $F_p=1$ and $p\mid d$ if and only if $p\mid n$, and
\begin{center}
$F_2=1$ and $|d|\equiv 0\pmod{4}$  if $n\not\equiv 3d_L \pmod{4}$, and $F_2=2$ and $|d|\equiv 3\pmod{4}$ otherwise. 
\end{center}
Furthermore, for an odd prime $p\mid d_L$, we have 
\begin{center}
$F_p=1$ and $p\mid d$ if $p\nmid n$, and $F_p=p$ and $p\nmid d$ if $p\mid n$.
\end{center}

By Lemma \ref{localdensities-stable2} and Corollary \ref{cor-Hurwitz-reduce-q}, one may verify that 
$$
H(|d|) c_2(n,L)= \frac{2}{2+S_2^\ast(L)} \times 
\begin{cases}
	H(|d|) & \text{if } n\not \equiv 3d_L \pmod{4},\\
	\frac{1}{2}\left(H(4|d|)+S_2^\ast(L)\cdot 2H(|d|)\right) & \text{if } n \equiv 3d_L \pmod{4}.
\end{cases}
$$
\begin{extradetails}
\bf
If $n\not\equiv 3d_L\pmod{4}$, then $d\equiv 0\pmod{4}$, so $\legendre{d}{2}=0$. For $L$ odd and anisotropic (i.e., $S_{2}^*(L)=-1$), Lemma \ref{localdensities-stable2} gives 
\begin{align*}
c_2(n,L)&=\alpha_2(n,L)\left(1-\legendre{d}{2}\frac{1}{2}\right)\frac{4}{3}=\frac{4}{3}\alpha_2(n,L)\\
&=
\begin{cases}
2^{2-\frac{a+1}{2}}=2^{1}=2=\frac{2}{2-1}&\text{if $2\mid n$ ($a=1$),}\\
2^{2-\frac{a+2}{2}}=2^{1}=2=\frac{2}{2-1}&\text{if $n\equiv d_L\pmod{4}$ ($a=0$)},
\end{cases}
\end{align*}
For $L$ odd and isotropic (i.e., $S_{2}^*(L)=1$), Lemma \ref{localdensities-stable2} gives 
\begin{align*}
c_2(n,L)&=\frac{4}{3}\alpha_2(n,L)\\
&=
\begin{cases}
\frac{4}{3}\left(2-3\cdot 2^{-1}\right)=\frac{2}{3}=\frac{2}{2+1}&\text{if $2\mid n$ ($a=1$),}\\
\frac{4}{3}\left(2-3\cdot 2^{-1}\right)=\frac{2}{3} = \frac{2}{2+1}&\text{if $n\equiv d_L\pmod{4}$ ($a=0$)}.
\end{cases}
\end{align*}
These cases then follow directly.

For $n\equiv 3d_L\pmod{4}$, then $|d|\equiv 3\pmod{4}$ and 
\[
\legendre{d}{2}=\begin{cases} -1&\text{if }n\equiv 3d_L\pmod{8},\\
1&\text{if }n\equiv 7d_L\pmod{8}.
\end{cases}
\]
For $L$ odd and anisotropic (i.e., $S_{2}^*(L)=-1$ and $n\equiv 3d_L\pmod{8}$), Lemma \ref{localdensities-stable2} gives 
\[
c_2(n,L)=\alpha_2(n,L)\left(1-\legendre{d}{2}\frac{1}{2}\right)\frac{4}{3}=\left(\frac{3}{2}\right)\frac{4}{3}=2.
\]
We thus need to show in this case that 
\[
2H(|d|)=H(4|d|)-2H(|d|)\Leftrightarrow H(4|d|)=4H(|d|).
\]
Corollary \ref{cor-Hurwitz-reduce-q} yields (note that $\mu_2=2$ because $F_2=2$) 
\[
H(4|d|)=H(|d|)\frac{2^2-1 - \legendre{d}{2}(2-1)}{2-1}=4H(|d|).
\]
So this case follows as well.

For $L$ odd and isotropic (i.e., $S_{2}^*(L)=1$), Lemma \ref{localdensities-stable2} gives 
\begin{align*}
c_2(n,L)&=\frac{4}{3}\left(1-\legendre{d}{2}\frac{1}{2}\right)\alpha_2(n,L)\\
&=
\begin{cases}
\frac{4}{3}\cdot \frac{3}{2}\cdot (2-1)=2&\text{if $n\equiv 3d_L\pmod{8}$},\\
\frac{4}{3}\cdot\frac{1}{2}(2)=\frac{4}{3}&\text{if $n\equiv 7d_L\pmod{8}$}.
\end{cases}
\end{align*}
Hence we want to show that 
\begin{multline*}
\begin{cases}
2H(|d|)=\frac{1}{3}\left(H(4|d|)+2H(|d|)\right)&\text{if }n\equiv 3d_L\pmod{8},\\
\frac{4}{3}H(|d|)=\frac{1}{3}\left(H(4|d|)+2H(|d|)\right)&\text{if }n\equiv 7d_L\pmod{8},
\end{cases}\\
\Leftrightarrow 
\begin{cases}
4H(|d|)=H(4|d|)&\text{if }n\equiv 3d_L\pmod{8},\\
2H(|d|)=H(4|d|)&\text{if }n\equiv 7d_L\pmod{8}.
\end{cases}
\end{multline*}
Now by Corollary \ref{cor-Hurwitz-reduce-q} (with $\mu_2=2$) 
\begin{equation}\label{eqn:H4nHn}
H(4|d|)=H(|d|)\frac{2^2-1 - \legendre{d}{2}(2-1)}{2-1}=
\begin{cases}
4H(|d|)&\text{if }n\equiv 3d_L\pmod{8},\\
2H(|d|)&\text{if }n\equiv 7d_L\pmod{8}.
\end{cases}
\end{equation}
So this case follows as well.
\rm
\end{extradetails}
Moreover, noting that $H\pfrac{|d|}{4}=0$ if $n\not\equiv 3d_L\pmod{4}$, this may be written as
\begin{equation}\label{mainthm:eq2}
H(|d|) c_2(n,L)= \frac{2}{2+S_2^\ast(L)} \times \frac{1}{F_2}\left(H(|d|F_2^2)+S_2^\ast(L)\cdot 2H\pfrac{|d|F_2^2}{4}\right).
\end{equation}
If we plug \eqref{mainthm:eq2} into \eqref{mainthm:eq1}, then we have
\begin{equation}\label{mainthm:eq3}
	r(n,\gen(L))=\frac{12\sqrt{n}}{\sqrt{d_L}\sqrt{|d|F_2^2}} \cdot  \frac{2}{2+S_2^\ast(L)} \cdot \left( H(|d|F_2^2) +S_2^\ast(L)\cdot 2H\pfrac{|d|F_2^2}{4}\right) \prod_{p\mid d_L} c_p(n,L),
\end{equation}
This results in a disappearance of the term $c_2(n,L)$ in the formula for $r(n,\gen(L))$ and the $2$-powers cancel in
\[
\frac{2\sqrt{n}}{\sqrt{d_L} \sqrt{|d|F_2^2}}=\frac{\sqrt{4d_L n}}{d_L \sqrt{|d|F_2^2}}
\]
because $\ord_2(4d_Ln)=\ord_2(|d|F_2^2)$ and $d_L$ is odd, but we obtain a sum of two Hurwitz class numbers instead.

We next combine $c_p(n,L)$ times a class number in a similar way for an odd prime $p\mid d_L$ to rewrite \eqref{mainthm:eq3} in the desired shape. For any positive integer $f\mid F$ with $(f,p)=1$, one may verify by Lemma \ref{localdensities-stable1} and Corollary \ref{cor-Hurwitz-reduce-q} that
$$
H(|d|f^2) c_p(n,L)= \frac{p}{p+S_p^\ast(L)} \times 
\begin{cases}
	H(|d|f^2) & \text{if } p\nmid n,\\
	\frac{1}{p}\left(H(|d|f^2p^2)+S_p^\ast(L)\cdot pH(|d|f^2)\right) & \text{if } p\mid n,
\end{cases}
$$
\begin{extradetails}
\bf
Since $n$ is squarefree, we have $\nu_p=\ord_p(n)\in\{0,1\}$, and since $p\mid d_L$ we have $\ord_p(d_L)=1$. Note that $p\mid d$ if and only if $\nu_p=0$.

By Lemma \ref{localdensities-stable1}, we have 
\begin{align*}
c_p(n,L)&=\alpha_p(n,L)\left(1-\legendre{d}{p}\frac{1}{p}\right)\left(1-\frac{1}{p^2}\right)^{-1}\\
&=\begin{cases}
\left(1-\frac{1}{p}\right)\left(1-\frac{1}{p^2}\right)^{-1}=\frac{p}{p+1}&\text{if }\nu_p=0\text{ and $L_p$ isotropic},\\
\left(1+\frac{1}{p}\right)\left(1-\frac{1}{p^2}\right)^{-1}=\frac{p}{p-1}&\text{if }\nu_p=0\text{ and $L_p$ anisotropic},\\
\left(1+\left(1-\frac{1}{p}+\legendre{d}{p}\frac{1}{p}\right)\right)\left(1-\legendre{d}{p}\frac{1}{p}\right)\left(1-\frac{1}{p^2}\right)^{-1}&\text{if }\nu_p=1\text{ and $L_p$ isotropic},\\
\left(1-\left(1-\frac{1}{p}+\legendre{d}{p}\frac{1}{p}\right)\right)\left(1-\legendre{d}{p}\frac{1}{p}\right)\left(1-\frac{1}{p^2}\right)^{-1}&\text{if }\nu_p=1\text{ and $L_p$ anisotropic}.
\end{cases}\\
&=\begin{cases}
\frac{p}{p+S_{p}^*(L)}&\text{if }\nu_p=0\text{ and $L_p$ isotropic},\\
\frac{p}{p+S_{p}^*(L)}&\text{if }\nu_p=0\text{ and $L_p$ anisotropic},\\
\left(2-\frac{1}{p}+\legendre{d}{p}\frac{1}{p}\right)\left(1+\legendre{d}{p}\frac{1}{p}\right)^{-1}&\text{if }\nu_p=1\text{ and $L_p$ isotropic},\\
\frac{1}{p} \left(1-\legendre{d}{p}\right)\left(1+\legendre{d}{p}\frac{1}{p}\right)^{-1}&\text{if }\nu_p=1\text{ and $L_p$ anisotropic}.
\end{cases}
\end{align*}
Since $n$ is represented by the genus, in the last case we must have $\legendre{d}{p}=-1$, so we obtain
\[
c_{p}(n,L)=
\begin{cases}
\frac{p}{p+S_{p}^*(L)}&\text{if }\nu_p=0\text{ and $L_p$ isotropic},\\
\frac{p}{p+S_{p}^*(L)}&\text{if }\nu_p=0\text{ and $L_p$ anisotropic},\\
\left(2-\frac{1}{p}+\legendre{d}{p}\frac{1}{p}\right)\left(1+\legendre{d}{p}\frac{1}{p}\right)^{-1}&\text{if }\nu_p=1\text{ and $L_p$ isotropic},\\
\frac{2}{p}\left(1-\frac{1}{p}\right)^{-1}&\text{if }\nu_p=1\text{ and $L_p$ anisotropic}.
\end{cases}
\]
For $p\nmid n$, we are done. For $p\mid n$ we need to verify
\begin{multline*}
\begin{cases}
H\left(|d|f^2\right)\left(2-\frac{1}{p}+\legendre{d}{p}\frac{1}{p}\right)\left(1+\legendre{d}{p}\frac{1}{p}\right)^{-1}=\frac{1}{p+1}\left(H\left(|d|f^2p^2\right) + pH\left(|d|f^2\right)\right)&\text{if $L_p$ isotropic},\\
H\left(|d|f^2\right)\frac{2}{p}\left(1-\frac{1}{p}\right)^{-1}=\frac{1}{p-1}\left(H\left(|d|f^2p^2\right) -pH\left(|d|f^2\right)\right)&\text{if $L_p$ anisotropic}.
\end{cases}\\
\Leftrightarrow\begin{cases}
H\left(|d|f^2\right)\left(2-\frac{1}{p}+\legendre{d}{p}\frac{1}{p}\right)\left(1+\legendre{d}{p}\frac{1}{p}\right)^{-1}=\frac{1}{p+1}\left(H\left(|d|f^2p^2\right) + pH\left(|d|f^2\right)\right)&\text{if $L_p$ isotropic},\\
(p+2)H\left(|d|f^2\right)=H\left(|d|f^2p^2\right)&\text{if $L_p$ anisotropic}.
\end{cases}
\end{multline*}
In the last case, Corollary \ref{cor-Hurwitz-reduce-q} gives (noting that $\legendre{d}{p}=-1$ and $\mu_p=2$ in this case)
\[
H\left(|d|f^2p^2\right)=H\left(|d|f^2\right)\frac{p^2-1 +p-1}{p-1}=H\left(|d|f^2\right)(p+1+1),
\]
verifying the claim. 

It remains to show that for $L_p$ isotropic we have 
\begin{equation}\label{eqn:Hdf2toshow}
H\left(|d|f^2\right)\left(2-\frac{1}{p}+\legendre{d}{p}\frac{1}{p}\right)\left(1+\legendre{d}{p}\frac{1}{p}\right)^{-1}=\frac{1}{p+1}\left(H\left(|d|f^2p^2\right) + pH\left(|d|f^2\right)\right).
\end{equation}
By Corollary \ref{cor-Hurwitz-reduce-q}, we have 
\[
H\left(|d|f^2p^2\right)=H\left(|d|f^2\right)\frac{p^2-1 -\legendre{d}{p}(p-1)}{p-1}=H\left(|d|f^2\right)\left(p+1-\legendre{d}{p}\right),
\]
Plugging this into \eqref{eqn:Hdf2toshow}, we want to verify that 
\begin{multline*}
H\left(|d|f^2\right)\left(2-\frac{1}{p}+\legendre{d}{p}\frac{1}{p}\right)\left(1+\legendre{d}{p}\frac{1}{p}\right)^{-1}=\frac{1}{p+1}\left(\left(p+1-\legendre{d}{p}\right)H\left(|d|f^2\right) + pH\left(|d|f^2\right)\right)\\
\Leftrightarrow \left(2-\frac{1}{p}+\legendre{d}{p}\frac{1}{p}\right)\left(1+\legendre{d}{p}\frac{1}{p}\right)^{-1}=\frac{1}{p+1}\left(p+1-\legendre{d}{p} + p\right)\\
\Leftrightarrow \left(2-\frac{1}{p}+\legendre{d}{p}\frac{1}{p}\right)(p+1)=\left(1+\legendre{d}{p}\frac{1}{p}\right)\left(2p+1-\legendre{d}{p}\right)\\
\Leftrightarrow \left(2p-1+\legendre{d}{p}\right)(p+1)=\left(p+\legendre{d}{p}\right)\left(2p+1-\legendre{d}{p}\right)\\
\Leftrightarrow 
2p^2+\left(\legendre{d}{p}+1\right) p +\legendre{d}{p}-1 = 2p^2+\left(\legendre{d}{p}+1\right) p + \legendre{d}{p}-\legendre{d}{p}^2, 
\end{multline*}
which holds because $\legendre{d}{p}^2=1$.\rm
\end{extradetails}
which may equivalently be written as
\begin{equation}\label{mainthm:eq4}
H(|d|f^2) c_p(n,L)= \frac{p}{p+S_p^\ast(L)} \times \frac{1}{F_p}\left(H(|d|f^2F_p^2)+S_p^\ast(L)\cdot pH\pfrac{|d|f^2F_p^2}{p^2}\right).
\end{equation}
We iteratively use \eqref{mainthm:eq4} to remove each prime factor in \eqref{mainthm:eq3} one-by-one. Note that \eqref{mainthm:eq4} is still valid for any rational number $f$ of the form $\frac{f_1}{f_2}\not\in\Z$ for some positive integers $f_1$ and $f_2$ such that $(f_1,f_2)=(f_1,p)=(f_2,p)=1$, since in that case both sides of \eqref{mainthm:eq4} are equal to zero. 
If there is an odd prime $p$ dividing $d_L$, then we may apply \eqref{mainthm:eq4} for $f=F_2$ and $\frac{F_2}{2}$, respectively, and then plug those into \eqref{mainthm:eq2} in order to remove the factor $c_p(n,L)$ from \eqref{mainthm:eq3}. By induction, iteratively plugging \eqref{mainthm:eq4} into \eqref{mainthm:eq3} yields
$$
r(n,\gen(L))=\frac{12\sqrt{n}}{\sqrt{d_L}\sqrt{|d|\prod_{p\mid 2d_L}F_p^2}} \cdot  \prod_{p\mid 2d_L} \frac{p}{p+S_p^\ast(L)} \times \sum_{f\mid 2d_L} \left[\prod_{p\mid f} S_p^\ast(L) \right] \cdot f \cdot H\pfrac{4d_Ln}{f^2}
$$
Noting that $|d|\prod_{p\mid 2d_L}F_p^2=4d_Ln$ and $\prod_{p\mid 2d_L}p=2d_L=\mathfrak{P}$, we obtain \eqref{formula-stable}.

The proof of the theorem when $L$ is even, in which case $\mathfrak{P}=d_L/2$ and $e=1$, is quite similar to the above, so we just outline the proof. Firstly, we note that if $\ord_2(n)=0$, then both sides of \eqref{formula-stable} are equal to zero; the left-hand side is zero by Theorem \ref{Minkowski-Siegel} and since $\a_2(n,L)=0$, and the right-hand side is zero since $d_Ln\equiv 2\pmod{4}$, hence all terms $H\pfrac{d_Ln}{f^2}$ are zero.
Using a similar argument to the proof for the case when $L$ is odd, it suffice to show that the formula \eqref{formula-stable} holds for any positive integer $n$ such that $\delta_{p,2}\le \ord_p(n) \le 1+\delta_{p,2}$ for any prime $p$.
\begin{extradetails}
\bf
The claim that $r(n,\gen(L))\neq 0$ if and only if $r(nq^2,\gen(L))\neq 0$ is locally equivalent to what was checked for $L$ odd for $q\neq 2$. For $L$ even, Lemma \ref{localdensities-stable2} implies that $r(n,\gen(L))=0$ if and only if $n$ is odd. 

Now assume that $r(n,\gen(L))\neq 0$. If $q\neq 2$, the assumption that $L$ is even is does not affect the local calculation, so \eqref{formula-stable} for $n$ implies \eqref{formula-stable} for $nq^2$. For $q=2$ and $n$ with $2\mid n$, we want to show that \eqref{formula-stable} for $n$ implies \eqref{formula-stable} for $4n$. In this case, Proposition \ref{prop3} implies that 
\[
\frac{r(4n,\gen(L)}{r(n,\gen(L))} = \frac{H\left(\frac{4d_Ln}{f^2}\right)}{H\left(\frac{d_Ln}{f^2}\right)}
\]
for any odd $f$ for which $H\left(\frac{d_Ln}{f^2}\right)\neq 0$. Plugging in \eqref{formula-stable} (recalling that $e=1$) then yields 
\begin{align*}
r(4n,\gen(L))&=r(n,\gen(L))\frac{H\left(\frac{4d_Ln}{f^2}\right)}{H\left(\frac{d_Ln}{f^2}\right)}\\
&=12\prod_{p\mid \mathfrak{P}} \left(p+S_p^\ast(L)\right)^{-1} \times \sum_{f\mid \mathfrak{P}} \left[\prod_{p\mid f} S_p^\ast(L) \right] \cdot f \cdot H\pfrac{d_Ln}{f^2}\frac{H\left(\frac{4d_Ln}{f^2}\right)}{H\left(\frac{d_Ln}{f^2}\right)}\\
&=12\prod_{p\mid \mathfrak{P}} \left(p+S_p^\ast(L)\right)^{-1} \times \sum_{f\mid \mathfrak{P}} \left[\prod_{p\mid f} S_p^\ast(L) \right] \cdot f \cdot H\pfrac{4d_Ln}{f^2}.
\end{align*}
This is \eqref{formula-stable} for $4n$.
\rm
\end{extradetails}
For those $n$, let $d$ be the discriminant of $Q(\sqrt{-d_Ln})$ and let $F\in\N$ be defined uniquely by $-d_Ln=dF^2$. Using Lemma \ref{localdensities-stable2}, one obtains
\begin{equation}\label{mainthm:eq5}
H(|d|)c_2(n,L)=2\times \frac{1}{F_2}H(|d|F_2^2),
\end{equation}
and one may complete the proof in a similar manner to the proof of the case when $L$ is odd.
\begin{extradetails}
\bf
Again, for $p\neq 2$, the argument is precisely as before because there is no difference locally between odd or even lattices.

We have either $\ord_2(n)=1$ or $\ord_2(n)=2$ by assumption. By Lemma \ref{localdensities-stable2}, we have 
\[
\alpha_2(n,L)=\begin{cases} 
\frac{3}{2}&\text{if }\ord_2(n)=2,\\ 
\frac{3}{2}&\text{if }\ord_2(n)=1\text{ and }\beta\equiv \frac{d_L}{2}\pmod{4},\\
2&\text{if }\ord_2(n)=1\text{ and }\beta\equiv 3\frac{d_L}{2}\pmod{8},\\ 
3&\text{if }\ord_2(n)=1\text{ and }\beta\equiv 7\frac{d_L}{2}\pmod{8}.
\end{cases}
\]
Hence, noting that
\[
\legendre{d}{2}=\begin{cases} 0&\text{if }\ord_2(n)=2,\\
0&\text{if }\ord_2(n)=1\text{ and }\beta\equiv \frac{d_L}{2}\pmod{4},\\
-1&\text{if }\ord_2(n)=1\text{ and }\beta\equiv 3\frac{d_L}{2}\pmod{8},\\ 
1&\text{if }\ord_2(n)=1\text{ and }\beta\equiv 7\frac{d_L}{2}\pmod{8},
\end{cases}
\]
we conclude that 
\begin{align*}
c_2(n,L)&=\alpha_2(n,L)\left(1-\legendre{d}{2}\frac{1}{2}\right)\frac{4}{3}\\
&=\begin{cases} 
2&\text{if }\ord_2(n)=2,\\ 
2&\text{if }\ord_2(n)=1\text{ and }\beta\equiv \frac{d_L}{2}\pmod{4},\\
4&\text{if }\ord_2(n)=1\text{ and }\beta\equiv 3\frac{d_L}{2}\pmod{8},\\ 
2&\text{if }\ord_2(n)=1\text{ and }\beta\equiv 7\frac{d_L}{2}\pmod{8}.
\end{cases}
\end{align*}
We hence need to check that
\[
\begin{cases}
2H(|d|)= 2\times \frac{1}{1} H(|d|) &\text{if }\ord_2(n)=2,\\ 
2H(|d|)=2\times \frac{1}{1} H(|d|) &\text{if }\ord_2(n)=1\text{ and }\beta\equiv \frac{d_L}{2}\pmod{4},\\
4H(|d|)=2\times \frac{1}{2} H(4|d|) &\text{if }\ord_2(n)=1\text{ and }\beta\equiv 3\frac{d_L}{2}\pmod{8},\\ 
2H(|d|)=2\times \frac{1}{2}H(4|d|)&\text{if }\ord_2(n)=1\text{ and }\beta\equiv 7\frac{d_L}{2}\pmod{8}.
\end{cases}
\]
The first two are tautologies and the last two follow by \eqref{eqn:H4nHn}.

Hence we conclude that 
$$
H(|d|)c_2(n,L)=2\times \frac{1}{F_2}H\left(|d|F_2^2\right).
$$
Plugging this into \eqref{mainthm:eq1}, then we obtain 
\[
r(n,\gen(L))=\frac{12\sqrt{n}}{\sqrt{d_L}\sqrt{|d|}}  \frac{2}{F_2} H\left(|d|F_2^2\right) \prod_{p\mid \frac{d_L}{2}} c_p(n,L),	
\]
Once again, the power of $2$ cancels in 
\[
\frac{2\sqrt{n}}{\sqrt{d_L}\sqrt{|d|F_2^2}}=\frac{\sqrt{d_Ln}}{\frac{d_L}{2}\sqrt{|d|F_2^2}}.
\]
We then directly simplify using the formulas for the $L$ odd case.
 \rm
\end{extradetails}
\end{proof}

\section{Formula construction for reductions to stable lattices}\label{section-formula-reduction}
Let $L$ be a primitive ternary lattice. Throughout this section, we always write $L_p=M_p\perp U_p$, where $M_p$ is the unimodular Jordan component and $\mathfrak{s}(U_p) \subseteq p\mathfrak{s}(M_p)$. Recall that if $\gcd(n,2d_L)=1$, then Jones \cite[Theorem 86]{Jones} obtained a formula for $r(n,\gen(L))$ in terms of Hurwitz class numbers. In the next theorem, we extend this by obtaining a formula for $r(n,\gen(L))$ whenever $n$ is relatively prime to the divisors of $2d_L$ for which the lattice $L$ is not stable. More precisely, we write the number of representations by the genus in terms of Hurwitz class numbers up to a computation of finitely-many local factors $c_{p}(n,L)$ for primes $p$ for which the lattice is not stable $p$-adically, and these local factors can also be easily computed in this case (see Remark \ref{rmk-n,P-coprime}).
\begin{theorem}\label{formula-n,P-coprime}
	Let $L$ be a primitive ternary lattice, and let $\mathcal{P}=\mathcal{P}_L$ and $\mathfrak{P}=\mathfrak{P}_L$ be defined as in \eqref{defn-mathcalP} and \eqref{defn-mathfrakP}, respectively, and let $\mathcal{P}'$ be the set of prime divisors of $2d_L$ which do not belong to $\mathcal{P}$. 
	Let $N\in\{1,2\}$ be such that $\n(L)=N\Z$, and let $n\in\N$ be such that $n\in N\Z_p^\times$ for any $p\in \mathcal{P}'$. Writing $4d_L=s\cdot t^2$ for some squarefree $s\in\N$ and $t\in\N$, we have
	\[
	r(n,\gen(L))=\frac{\epsilon}{st} \left[\prod_{p\mid \mathfrak{P}} \frac{p}{p+S_p^\ast(L)}\right] \prod_{p\in\mathcal{P}'} c_p(n,L) \times {\sum_{f\mid \mathfrak{P}}} \left[ \prod_{p\mid f}S_p^\ast(L)\right]\cdot f \cdot H\pfrac{esn}{f^2},
	\]
Here $S_p^\ast(L)=(-1)^{\delta_{p,2}}S_p(L)$ (cf. \eqref{HasseSymbol-IsoAniso}),
$c_p(n,L)=\a_p(n,L)\left(1-\legendre{d}{p}\frac{1}{p}\right) \left(1-\frac{1}{p^2}\right)^{-1}$ with  $d$ the discriminant of $\Q(\sqrt{-4d_Ln})=\Q(\sqrt{-sn})$, 
and $\epsilon$ and $e$ are constants defined by
\[
(\epsilon,e)=(\epsilon_L,e_L) =\begin{cases}
	\left(24(s,n,2),\frac{1}{(s,n,2)^2}\right)& \text{if  $2\not\in\mathcal{P}$, $\frac{sn}{(s,n)^2}\equiv 3 \pmod{4}$,}\\
	\left(12(s,n,2),\frac{4}{(s,n,2)^2}\right)& \text{if $2\not\in\mathcal{P}$, $\frac{sn}{(s,n)^2}\not\equiv 3 \pmod{4}$, or $2\in\mathcal{P}$, $\ord_2(d_L)=0$,}\\
	(48,1)& \text{if $2\in \mathcal{P}$, $\ord_2(d_L)=1$.}	
\end{cases}
\]
\end{theorem}

\begin{remark}\label{rmk-n,P-coprime}
For any given $L$ and for any $n\in\N$ under the assumption in Theorem \ref{formula-n,P-coprime}, note that $\epsilon_L$ and $e_L$ are indeed constants since $(s,n,2)^2=4$ if $\n(L)=2\Z$ and $\ord_2(d_L)$ is odd, and $1$ otherwise.
Moreover, the computation of $c_p(n,L)$ is somewhat straightforward, so Theorem \ref{formula-n,P-coprime} yields a formula resembling Theorem \ref{thm:sumclassnumapproximate} for these $n$. Precisely, for any such $n$ and an odd prime $p\in \mathcal{P}'$, we have by \cite[Theorem 3.1]{Yang} that
\[
\a_p(n,L)=\begin{cases}
	1-\legendre{-d_{M_p}}{p}\frac{1}{p} & \text{if rank }M_p=2,\\
	1+\legendre{nd_{M_p}}{p}=0\text{ or }2 & \text{if rank }M_p=1.
\end{cases}
\]
Moreover, note that for any such $n$ and a prime $p$, $\legendre{d}{p}=\legendre{-4^bsn}{p}$ for some $b\in\{0,\pm1\}$. Thus, the value $\prod_{p\in\mathcal{P}'} c_p(n,L)$ only depends on the residue class of $n$ modulo $2^k\prod_{p\in\mathcal{P}'} p$ for some integer $0\le k \le 4$.
\end{remark}
\begin{proof}[Proof of Theorem \ref{formula-n,P-coprime}]
	We will show that the formula holds for any positive integer $n$ such that $\delta_{p,2}\ord_2(N)\le \ord_p(n)\le 1 +\delta_{p,2}\ord_2(N)$ for any prime $p\in\mathcal{P}$ or $p\nmid 2d_L$ (i.e. $p\not\in \mathcal{P}'$). 
	Then by Propositions \ref{prop1}, \ref{prop2}, and \ref{prop3}, the formula will also hold in general as is described in the proof of Theorem \ref{mainthm}.
	
	Assume that $\delta_{p,2}\ord_2(N) \le \ord_p(n)\le \delta_{p\notin\mathcal{P}'} +\delta_{p,2}\ord_2(N)$.
	By Lemma \ref{ternary-generic-theta}, we have
	\begin{equation}\label{eqn:rgenLclassnum}
	r(n,\gen(L))=\frac{12\sqrt{n}}{\sqrt{d_L}\sqrt{|d|}} H(|d|) \prod_{p\mid 2d_L} c_p(n,L).
	\end{equation}
	First, consider the case when $2\in\mathcal{P}$.
	For each prime $p\in \mathcal{P}$, we may repeat the same process of removing the term $c_p(n,L)$ from \eqref{eqn:rgenLclassnum} by using \eqref{mainthm:eq2} or \eqref{mainthm:eq5} and then iteratively applying \eqref{mainthm:eq4}, as is done in the proof of Theorem \ref{mainthm}. This yields
	$$
	r(n,\gen(L))=\frac{12\sqrt{n}}{\sqrt{d_L}\sqrt{esn}} \left[\prod_{p\in\mathcal{P}} \frac{p}{p+S_p^\ast(L)}\right] 
	\prod_{\substack{p\mid 2d_L \\ p\not\in\mathcal{P}}} c_p(n,L) \cdot \sum_{f\mid \mathfrak{P}} \left[ \prod_{p\mid f}S_p^\ast(L)\right]\cdot f \cdot H\pfrac{esn}{f^2},
	$$
	where $e=4$ if $\ord_2(d_L)=0$ and $e=1$ if $\ord_2(d_L)=1$.
\begin{extradetails}
First suppose that $L$ is odd. We write 
\[
-4sn=d F^2=d\prod_p F_p^2
\]
and first claim that 
\begin{multline}\label{eqn:notstableclaim-Lodd}
r(n,\gen(L))=\frac{12\sqrt{n}}{\sqrt{d_L}\sqrt{|d|\prod_{p\in\mathcal{P}}F_p^2}}\prod_{p\in\mathcal{P}}\frac{p}{p+S_p^*(L)} \prod_{\substack{p\mid d_L\\ p\notin \mathcal{P}}} c_{p}(n,L)\\
\times \sum_{f\mid \mathfrak{P}}\prod_{p\mid f}\left[S_p^\ast(L)\right]\cdot f H\pfrac{|d| \prod_{p\in \mathcal{P}} F_p^2}{f^2}.
\end{multline}

We begin by using \eqref{mainthm:eq2} to rewrite 
\begin{multline*}
r(n,\gen(L))=\frac{12\sqrt{n}}{\sqrt{d_L}\sqrt{|d|F_2^2}}\frac{2}{2+S_2^\ast(L)} \left(H(|d|F_2^2)+S_2^\ast(L)\cdot 2H\pfrac{|d|F_2^2}{4}\right) \prod_{p\mid d_L} c_p(n,L)\\
=\frac{12\sqrt{n}}{\sqrt{d_L}\sqrt{|d|\prod_{p\in\{2\}} F_p^2}}\prod_{p\in\{2\}} \frac{p}{p+S_p^\ast(L)} \prod_{\substack{p\mid d_L\\ p\notin\{2\}}} c_p(n,L) \sum_{f\mid 2} \prod_{p\mid f} S_{p}^*(L)\cdot f\cdot H\left(\frac{|d| \prod_{p\in\{2\}} F_p^2}{f^2}\right).
\end{multline*}
This is \eqref{eqn:notstableclaim-Lodd} if $\mathcal{P}=\{2\}$. We proceed to show \eqref{eqn:notstableclaim-Lodd} by induction on subsets of $\mathcal{P}$. Suppose that $2\in \mathcal{S}\subsetneq \mathcal{P}$ and 
\[
r(n,\gen(L))=\frac{12\sqrt{n}}{\sqrt{d_L}\sqrt{|d|\prod_{p\in\mathcal{S}} F_p^2}}\prod_{p\in\mathcal{S}} \frac{p}{p+S_p^\ast(L)} \prod_{\substack{p\mid d_L\\ p\notin\mathcal{S}}} c_p(n,L) \sum_{f\mid \prod_{p\in \mathcal{S}}p} \prod_{p\mid f} S_{p}^*(L) H\left(\frac{|d| \prod_{p\in \mathcal{S}} F_p^2}{f^2}\right).
\]
Now let $q\in \mathcal{P}\setminus\mathcal{S}$ be given. Setting $\mathcal{S}_q:=\mathcal{S}\cup\{q\}$, \eqref{mainthm:eq4} implies that  
\begin{multline*}
r(n,\gen(L))=\frac{12\sqrt{n}}{\sqrt{d_L}\sqrt{|d|\prod_{p\in\mathcal{S}_q} F_p^2}}\prod_{p\in\mathcal{S}_q} \frac{p}{p+S_p^\ast(L)} \prod_{\substack{p\mid d_L\\ p\notin\mathcal{S}_q}} c_p(n,L)\\
\times \sum_{f\mid \prod_{p\in \mathcal{S}_q}p} \prod_{p\mid f} S_{p}^*(L) H\left(\frac{|d| \prod_{p\in \mathcal{S}_q} F_p^2}{f^2}\right).
\end{multline*}
By induction on $\#\mathcal{S}$, we conclude \eqref{eqn:notstableclaim-Lodd}.

We now claim that 
\begin{equation}\label{eqn:dFp2=4sn}
|d|\prod_{p\in\mathcal{P}}  F_p^2 =4sn,
\end{equation}
which implies that 
\[
\frac{H\left(\frac{|d|\prod_{p\in\mathcal{P}} F_p^2}{f^2}\right)}{\sqrt{|d|\prod_{p\in\mathcal{P}}F_p^2}} = \frac{H\left(\frac{4sn}{f^2}\right)}{\sqrt{4sn}}.
\]
Plugging this formula into \eqref{eqn:notstableclaim-Lodd} then yields the claimed identity.

To see \eqref{eqn:dFp2=4sn}, first recall that $-4sn=dF^2$. Since $s$ is squarefree, for an odd prime $p$ we conclude that $p\mid F$ if and only if either $p^2\mid n$ or $p\mid (s,n)$. By assumption, for every odd prime $p$ we have  $0\leq \ord_{p}(n)\leq 1$ and $\ord_p(n)=0$ if $p\in\mathcal{P}'$. Thus $p^2\mid n$ never occurs, and $p\mid (s,n)$ implies that $p\in\mathcal{P}$. Writing $F=\prod_p F_p$, we conclude that for odd primes $p$ we have 
\[
F_p=\begin{cases} p&\text{if }p\mid (s,n)\text{ and }p\in\mathcal{P},\\ 
1&\text{otherwise}.
\end{cases}
\]
Therefore 
\[
(s,n)^2=\prod_{p\neq 2} F_p^2 = \prod_{p\in\mathcal{P}\setminus\{2\}} F_p^2.
\]
For odd primes $p$, also have $p\mid d$ if and only if either ($p\mid s$ and $p\nmid n$) or ($pmid n$ and $p\nmid s$). The product of all such primes is precisely 
\[
\frac{sn}{(s,n)}.
\]
Thus 
\[
4sn=|d|F^2= |d|F_2^2 \prod_{p\neq 2}F_p^2 = |d|\prod_{p\in\mathcal{P}} F_p^2,
\]
as claimed.

Now suppose that $L$ is even. If we define $F$ to be the positive integer such that $-sn=dF^2$, where $d$ is the discriminant of $\Q(\sqrt{-4d_Ln})=\Q(\sqrt{-snt^2})=\Q(\sqrt{-sn})$, then we have
\[
\begin{aligned}
	r(n,\gen(L))&=\frac{12\sqrt{n}}{\sqrt{d_L}\sqrt{|d|\prod_{p\in\mathcal{P}}F_p^2}}\cdot2\prod_{p\in\mathcal{P}\setminus\{2\}}\frac{p}{p+S_p^*(L)} \prod_{\substack{p\mid d_L\\ p\notin \mathcal{P}}} c_{p}(n,L) \sum_{f\mid \mathfrak{P}}\prod_{p\mid f}\left[S_p^\ast(L)\right]\cdot f H\pfrac{|d| \prod_{p\in \mathcal{P}} F_p^2}{f^2}\\
	&=\frac{24\sqrt{n}}{\sqrt{d_L}\sqrt{|d|\prod_{p\in\mathcal{P}}F_p^2}}\prod_{p\mid \mathfrak{P}}\frac{p}{p+S_p^*(L)} \prod_{\substack{p\mid d_L\\ p\notin \mathcal{P}}} c_{p}(n,L) \sum_{f\mid \mathfrak{P}}\prod_{p\mid f}\left[S_p^\ast(L)\right]\cdot f H\pfrac{|d| \prod_{p\in \mathcal{P}} F_p^2}{f^2}
\end{aligned}
\]
Under the assumption that $n\in N\Z_p^\times$ for any $p\in\mathcal{P}'$, we have 
\[
|d|\prod_{p\in\mathcal{P}} F_p^2=sn,
\]
hence we obtain the formula since $\frac{24\sqrt{n}}{\sqrt{d_L}\sqrt{|d|\prod_{p\in\mathcal{P}}F_p^2}}=\frac{24\sqrt{n}}{\sqrt{st^2/4}\sqrt{sn}}=\frac{48}{st}$.
\end{extradetails}
	Plugging $d_L=st^2/4$ into the above formula, we have the theorem.
	
	Now we consider the case when $2\notin \mathcal{P}$. Note that in this case, $L$ is not stable at $2$ so that we may not apply Proposition \ref{prop3}. However, we may still repeat the same process of removing the term $c_p(n,L)$ for each prime $p\in \mathcal{P}$ from \eqref{eqn:rgenLclassnum} as is done in the proof of Theorem \ref{mainthm} to obtain
	\[
	r(n,\gen(L))=\frac{12\sqrt{n}}{\sqrt{d_L}\sqrt{|d|\frac{(s,n)^2}{(s,n,2)^2}}} \left[\prod_{p\in\mathcal{P}} \frac{p}{p+S_p^\ast(L)}\right] 
	\prod_{\substack{p\mid 2d_L \\ p\not\in\mathcal{P}}} c_p(n,L) \cdot \sum_{f\mid \mathfrak{P}} \left[ \prod_{p\mid f}S_p^\ast(L)\right]\cdot f \cdot H\pfrac{|d|\frac{(s,n)^2}{(s,n,2)^2}}{f^2}.
	\]
	Plugging $4d_L=st^2$, and $|d|(s,n)^2=sn$ if $\frac{sn}{(s,n)^2}\equiv 3 \pmod{4}$, $|d|(s,n)^2=4sn$ if $\frac{sn}{(s,n)^2}\not\equiv 3 \pmod{4}$ into the above formula, we have the theorem.
\begin{extradetails}
Define $F$ to be the positive integer such that $-4sn=dF^2$, where $d$ is the discriminant of $\Q(\sqrt{-4d_Ln})=\Q(\sqrt{-snt^2})=\Q(\sqrt{-4sn})$.
	In this case, we also obtain
	\[
	r(n,\gen(L))=\frac{12\sqrt{n}}{\sqrt{d_L}\sqrt{|d|\prod_{p\in\mathcal{P}}F_p^2}}\prod_{p\in\mathcal{P}}\frac{p}{p+S_p^*(L)} \prod_{\substack{p\mid d_L\\ p\notin \mathcal{P}}} c_{p}(n,L) \sum_{f\mid \mathfrak{P}}\prod_{p\mid f}\left[S_p^\ast(L)\right]\cdot f H\pfrac{|d| \prod_{p\in \mathcal{P}} F_p^2}{f^2}
	\]
	Under the assumption that $n\in N\Z_p^\times$ for any $p\in\mathcal{P}'$, note that for any odd prime $p\mid d_L$, (note that $n$ is indeed squarefree under the assumption in this case)
	\begin{center}
		$p\nmid d$ and $F_p=p$ $\iff$ $p\mid (s,n)$ $\iff$ $p\mid s$, $p\mid n$, $p\in \mathcal{P}$ $\iff$ $p\mid n$, $p\in \mathcal{P}$,
	\end{center}
	and
	\begin{center}
		$p\mid d$ and $F_p=1$ $\iff$ $p\mid s$ and $p\nmid n$, or $p\nmid s$, $p\mid n$ $\iff$ $p\mid s$ and $p\nmid n$ \\(since $p\mid n$ implies $p\in\mathcal{P}$ hence $p\mid s$).
	\end{center}
	Therefore, it is clear for any odd prime $p\mid d_L$ that
	\[
	\ord_p(d)=\ord_p\left(\frac{sn}{(s,n)^2}\right)=\ord_p\left(\frac{4sn}{(s,n)^2}\right),
	\]
	and
	\[
	|d|\prod_{p\in\mathcal{P}} F_p^2= |d|\frac{(s,n)^2}{(s,n,2)^2}
	\]
	For $p=2$, since $2\notin\mathcal{P}$ (i.e. $2\in \mathcal{P}'$), we have:\\
	If $L$ is even, then $n\in 2\Z_2^\times$. 
	\begin{enumerate}
		\item If $2\mid s$, then $\begin{cases}
			\ord_2(d)=2=\ord_2(\frac{4sn}{(s,n)^2}) & \text{if } \frac{sn}{(s,n)^2}\equiv sn/4 \not\equiv 3 \pmod{4},\\
			\ord_2(d)=0=\ord_2(\frac{sn}{(s,n)^2}) & \text{if } \frac{sn}{(s,n)^2}\equiv sn/4 \equiv 3 \pmod{4}.
		\end{cases}$
		\item If $2\nmid s$, then $\ord_2(d)=3=\ord_2(\frac{4sn}{(s,n)^2})$ (in which case $\frac{sn}{(s,n)^2}\equiv 2\pmod{4}$).
	\end{enumerate}
	If $L$ is odd, then $n\in \Z_2^\times$. 
	\begin{enumerate}
		\item If $2\mid s$, then $\ord_2(d)=3=\ord_2(\frac{4sn}{(s,n)^2})$ (in which case $\frac{sn}{(s,n)^2}\equiv 2\pmod{4}$).
		
		\item If $2\nmid s$, then $\begin{cases}
			\ord_2(d)=2=\ord_2(\frac{4sn}{(s,n)^2}) & \text{if } \frac{sn}{(s,n)^2}\equiv sn \not\equiv 3 \pmod{4},\\
			\ord_2(d)=0=\ord_2(\frac{sn}{(s,n)^2}) & \text{if } \frac{sn}{(s,n)^2}\equiv sn \equiv 3 \pmod{4}.
		\end{cases}$
	\end{enumerate}
	To sum up, we have
	\[
	\ord_2(d) = \begin{cases}
		\ord_2\left(\frac{4sn}{(s,n)^2}\right) & \text{if } \frac{sn}{(s,n)^2} \not\equiv 3 \pmod{4},\\
		\ord_2\left(\frac{sn}{(s,n)^2}\right) & \text{if } \frac{sn}{(s,n)^2}\equiv 3 \pmod{4}.
	\end{cases}
	\]
	Therefore, since $p\mid d \implies p\mid 2d_L$ we have
	\[
	|d|\prod_{p\in\mathcal{P}} F_p^2= |d|\frac{(s,n)^2}{(s,n,2)^2} = \frac{1}{(s,n,2)^2}\begin{cases}
		4sn & \text{if } \frac{sn}{(s,n)^2}\not\equiv 3\pmod{4},\\
		sn & \text{if } \frac{sn}{(s,n)^2}\equiv 3\pmod{4}.\\
	\end{cases}
	\]
	
\end{extradetails}

\end{proof}
Since every lattice can be reduced to a stable lattice after applying finitely-many Watson transformations by Corollary \ref{cor:finitereduction} and we have a formula for $r(n,\gen(L))$ when $L$ is stable by Theorem \ref{mainthm}, a formula for $r(n,\gen(L))$ for an arbitrary lattice $L$ can be recursively obtained from formulas for $r(n,\gen(\lambda_m(L)))$ once one obtains a relation between $r(n,\gen(L))$ and $r(n,\gen(\lambda_m(L)))$ (note that in practice we do not always have to apply the full reduction, as we may halt the reduction process whenever the reduced form satisfies the conditions assumed in Theorem \ref{formula-n,P-coprime}). The next theorem provides such a relation.
\begin{theorem}\label{reduction-formulas}
	Let $L$ be a ternary lattice and $n\in\N$. Suppose that $L$ is not stable at a prime $p$, and write $L_p=M_p\perp U_p$, where $M_p$ is the unimodular Jordan component and $\mathfrak{s}(U_p)\subseteq p\Z_p$.
	
	\begin{enumerate}[label={\rm(\arabic*)}, leftmargin=*]
		\item\label{reduction-formulas:1}  If $M_p$ is anisotropic and $\n(L_p)=\Z_p$, then
		$$
		r(pn,\gen(L)) = r(pn,\gen(\Lambda_p(L)))=r(n^\ast,\gen(\lambda_p(L))),
		$$
		where $n^\ast=n$ if $\mathfrak{s}(U_p)=p\Z_p$ or $p=2$ with $\text{rank}(M_2)=2$, and $n^\ast = n/p$ otherwise.
		\item\label{reduction-formulas:2} If $M_2$ is anisotropic and $\n(L_2)=2\Z_2$, equivalently $M_2\cong \left(\begin{smallmatrix}	2&1\\1&2 \end{smallmatrix}\right)$ with $\ord_2(d_L)\ge2$, then
		$$
		r(4n,\gen(L)) = r(n,\gen(\lambda_4(L))).
		$$
		
		\item\label{reduction-formulas:3} Otherwise, $L_p=\left(\begin{smallmatrix}	0&1\\1&0 \end{smallmatrix}\right)\perp \la p^m \varepsilon\ra$ for some $m\ge2$, $\varepsilon\in\Z_p^\times$, and we have
		$$
		r(epn,\gen(L)) = 2r(en,\gen(K))-r(en/p,\gen(\lambda_{ep}(L))),
		$$
		where $e=2$ if $p=2$, and $1$ otherwise, and $K$ is a lattice such that $K_p\cong \left(\begin{smallmatrix}	0&1\\1&0 \end{smallmatrix}\right)\perp \la p^{m-1} \varepsilon\ra$ and $K_q\cong L_q$ for any prime $q\neq p$.
	\end{enumerate}
In particular, $\ord_p(d_{\lambda_p(L)})<\ord_p(L)$ and $\ord_p(d_K)<\ord_p(L)$ in any case.
\end{theorem}
\begin{proof}
	\ref{reduction-formulas:1} Assume that $M_p$ is anisotropic and $\n(L_p)=\Z_p$. Then it is well known that 
	$$
	r(pn,T)=r(pn,\Lambda_p(T)).
	$$
	for any $T\in \gen(L)$. Summing up over $[T]\in\gen(L)/_\sim$, by \eqref{partition-genus}, \eqref{sum-orbit}, and Lemma \ref{weight-ratio},
	$$
	\begin{aligned}
		\sum_{[T]\in\gen(L)/_\sim} \frac{r(pn,T)}{o(T)} &= \sum_{[T]\in\gen(L)} \frac{r(pn,\Lambda_p(T))}{o(T)}\\
		&\overset{\eqref{partition-genus}}{=}\sum_{[N]\in\gen(\Lambda_p(L))/_\sim}\sum_{[T]\in\Gamma_p^L(N)/_\sim} \frac{o(N)}{o(T)} \frac{r(pn,N)}{o(N)}\\
		&\overset{\eqref{sum-orbit}}{=}\sum_{[N]\in\gen(\Lambda_p(L))/_\sim}\left| \Gamma_p^L(N) \right| \frac{r(pn,N)}{o(N)}\\
		&=\frac{w(L)}{w(\Lambda_p(L))} \sum_{[N]\in\gen(\Lambda_p(L))/_\sim} \frac{r(pn,N)}{o(N)}.\\
	\end{aligned}
	$$
	This yields the first identity of part \ref{reduction-formulas:1} of the theorem. The second identity follows by  Lemmas \ref{even-transform} and \ref{odd-transform}.
\begin{extradetails}
We scale the lattice by either a factor of $p$ or $p^2$, depending on the highest power that we can scale to get a primitive lattice. For an odd prime $p$, Lemma \ref{even-transform} implies that
\[
\Lambda_{p}(L)_p=pM_p\perp U_p.
\]
If $\mathfrak{s}(U_p)=p\Z_p$, then scaling by a factor of $p$ gives a primitive lattice. If $\mathfrak{s}(U_p)\subseteq p^2\Z_p$, then we can scale by a factor of $p^2$.

For $p=2$, we use Lemma \ref{odd-transform}. If $M_2$ has rank $2$, then Lemma \ref{odd-transform} (a) implies that we scale by a factor of $2$. If $M_2$ has rank $1$ and $\mathfrak{s}(U_2)=2\Z_2$, then Lemma \ref{odd-transform} (c) implies that we scale by a factor of $2$. Finally, if $M_2$ has rank $1$ and $\mathfrak{s}(U_2)\subseteq 4\Z_2$, then Lemma \ref{odd-transform} (b) implies that we scale by a factor of $4$.
\end{extradetails}
		
	\ref{reduction-formulas:2} In this case, $r(4n,T)=r(4n,\Lambda_4(T))$ for any $T\in \gen(L)$, and $\lambda_4(L)=\Lambda_4(L)^{\frac{1}{4}}$. The rest of the proof may be done in a similar way to \ref{reduction-formulas:1}.

	\ref{reduction-formulas:3} This is indeed a consequence of the results in \cite{JuLeeOh}, which introduce a generalization of Watson transformation. 
	We first provide the proof for the case when $p$ is odd.	
	Let $S\in \gen(L)$. There are exactly two sublattices of $S$, say $\Gamma_{p,1}(S)$ and $\Gamma_{p,2}(S)$, of index $p$ and of norm ideal $p\Z_p$.
	Note that $\Gamma_{p,1}(S) \cap \Gamma_{p,2}(S)=\Lambda_{p}(S)$, and they satisfy
	$$
	\begin{aligned}
		r(pn,S)&=r(pn,\Gamma_{p,1}(S)) + r(pn,\Gamma_{p,2}(S)) - r(pn,\Lambda_{p}(S))\\
		&=r(n,\Gamma_{p,1}(S)^\frac{1}{p}) + r(n,\Gamma_{p,2}(S)^\frac{1}{p}) - r(pn,\Lambda_{p}(S)).		
	\end{aligned}
	$$
	Moreover, $\Gamma_{p,1}(S)$ and $\Gamma_{p,2}(S)$ are the only sublattices of $S$ which belong to $\gen(K^p)$.
	On the other hand, for any lattice $T\in\gen(K)$, the number of sublattices in $S$ that are isometric to $T^p$ is $\frac{r(T^p,S)}{o(T)}$.
	Hence, from the above equality and properties, we have
	\begin{equation}\label{eqn:reduction-1}
		r(pn,S)=\sum_{[T]\in\gen(K)/_\sim}\frac{r(T^p,S)}{o(T)} r(n,T) - r(pn,\Lambda_{p}(S)).		
	\end{equation}
	We refer the readers to the proof of Proposition 4.1 of \cite{JuLeeOh} for more details.
	
	Multiplying by $\frac{1}{o(S)}$ on both sides of \eqref{eqn:reduction-1} and summing up over $[S]\in\gen(L)/_\sim$, we have
\begin{equation}\label{eqn:reduction-2}
	\sum_{[S]\in\gen(L)/_\sim} \frac{r(pn,S)}{o(S)}=\sum_{[S]\in\gen(L)/_\sim}\sum_{[T]\in\gen(K)/_\sim}\frac{r(T^p,S)}{o(S)} \frac{r(n,T)}{o(T)} - \sum_{[S]\in\gen(L)/_\sim} \frac{r(pn,\Lambda_p(S))}{o(S)}.		
\end{equation}
Following the proof of \ref{reduction-formulas:1}, the final sum on the right-hand side of \eqref{eqn:reduction-2} is equal to 
\begin{equation}\label{eqn:secondsummand}
	w(L)\cdot r(pn,\gen(\Lambda_p(L)))=w(L)\cdot r(n/p,\gen(\lambda_p(L))),
\end{equation}
where in the last step we the used the fact that  $\mathfrak{s}(\Lambda_p(L))=p^2\Z_p$.  To evaluate the first sum in \eqref{eqn:reduction-2}, we interchange the sums and then use the fact from \cite[Lemma 2.7]{JuLeeOh} that 
\[
	r(T^p,S)=\tilde{r}(S^p,T):=\{\sigma:S^p\rightarrow T \mid T/\sigma(S^p)\simeq \Z/p\Z\oplus\Z/p\Z\}.
\]
Since $m\ge 2$, \cite[Lemma 3.1]{JuLeeOh} implies that 
\begin{equation}\label{eqn:firstsummandweight}
	\sum_{[S]\in\gen(L)/_\sim}\frac{r(T^p,S)}{o(S)} = \sum_{[S]\in\gen(L)/_\sim}\frac{\tilde{r}(S^p,T)}{o(S)}= 2p.
\end{equation}
	Hence, the first double sum in \eqref{eqn:reduction-2} is equal to 
	$$
	2p\cdot \sum_{[T]\in\gen(K)/_\sim} \frac{r(n,T)}{o(T)} = 2p \cdot w(K) \cdot r(n,\gen(K)).
	$$
	Therefore, we have
	$$
	r(pn,\gen(L))=2p\cdot \frac{w(K)}{w(L)} \cdot r(n,\gen(K)) - r(n/p,\gen(\lambda_p(L))).
	$$
	We claim that $2p\cdot \frac{w(K)}{w(L)}=2$, from which the theorem follows.  By Theorem \ref{Minkowski-Siegel}, we have 
	\[	
	\frac{w(K)}{w(L)}=\frac{(d_K)^2}{(d_L)^2}\cdot \frac{\a_p(L,L)}{\a_p(K,K)}=\frac{1}{p^2}\cdot\frac{\beta_p(L,L)}{\beta_p(K,K)}
	\]
	and by using the formula in \cite[Theorem 5.6.3]{KiBook}, we may compute the local densities 
	\[
	\beta_p(K,K)=4p^{m-1}\left(1-\frac{1}{p}\right) \quad \text{and} \quad \beta_p(L,L)=4p^{m}\left(1-\frac{1}{p}\right).
	\]
	Thus, $\frac{w(K)}{w(L)}=\frac{1}{p}$, which proves the claim, hence the proof of \ref{reduction-formulas:3} when $p$ is odd.
\begin{extradetails}
	Note that in our case (3), letting $-d_L=p^m\varepsilon$ ($m\ge2$), we have
	\[
	L_p\cong \begin{pmatrix}
		0 & 1\\
		1 & 0
	\end{pmatrix} \perp \langle p^m\varepsilon\rangle, \
	K_p\cong \begin{pmatrix}
				0 & 1\\
		1 & 0
	\end{pmatrix} \perp \langle p^{m-1}\varepsilon\rangle,
	\]
	so \cite[Theorem 5.6.3]{KiBook} implies the following.
	Let $N_i = \perp_{j\in\Z} M_j$ be a regular quadratic lattice over $\Z_p$, where $M_j$ is $p^j$-modular or $\{0\}$, and $M_j\neq \{0\}$ occurs only for finitely many $j$, and write $M_j=N_j^{p^j}$ for the scaling of a unimodular lattice $N_j$ by $p^j$. Put
	\[
	n_j=\text{rank}N_j = \text{rank}M_j, \quad w:=\sum_j jn_j\{\sum_{k>j}n_k+(n_j+1)/2\}, \quad P(m):=\prod_{i=1}^m (1-p^{-2i}) (=1 \text{ if } m=0).
	\]
	With, for a regular quadratic space $U$ or $U=\{0\}$ over $\Z/p\Z$,
	\[
	\chi(U):=\begin{cases}
		0 & \text{if $\dim U$ is odd},\\
		1 & \text{if $U$ is hyperbolic space},\\
		-1& \text{otherwise},
	\end{cases}
	\]
	and for unimodular lattice $M$ over $\Z_p$ with $\n(M)=2\mathfrak{s}(M)$, we define a regular quadratic space $\bar{M}:=(M/pM,\frac{1}{2}Q)$ over $\Z/p\Z$, and put 
	\[
	\chi(M):=\chi(\bar{M})
	\]
	If $p\neq 2$, then 
	\[
	\beta_p(N,N)=2^sp^wPE,
	\] 
	where
	\begin{gather*}
		s:=\# \{j|M_j\neq \{0\}\},\ P:=\prod_jP(\lfloor n_j/2\rfloor), \ E:=\prod_{j:M_j\neq 0}(1+\chi(N_j)p^{-n_j/2})^{-1}
	\end{gather*}
	In our case, for $L_p\cong \begin{pmatrix}
		0 & 1\\
		1 & 0
	\end{pmatrix} \perp \langle p^m\varepsilon\rangle$ with $m\ge1$,
	\begin{gather*}
	n_0=2, \ n_m=1,\ n_j=0 \text{ for any }j\neq0,m ,\ w=m \cdot (\sum_{k>m} n_k+1)=m, \ s=2,\\
	P=P(0)P(1)=1\cdot \left(1-\frac{1}{p^2}\right), \ E=\left(1+\frac{1}{p}\right)^{-1}.
	\end{gather*}
Hence
\[
\beta_p(L,L) = 2^sp^wPE = 4p^m\left(1-\frac{1}{p^2}\right)\left(1+\frac{1}{p}\right)^{-1} = 4p^m\left(1-\frac{1}{p}\right).
\]
Since $K_p\cong \begin{pmatrix}
	0 & 1\\
	1 & 0
\end{pmatrix} \perp \langle p^{m-1}\varepsilon\rangle$, and we are assuming $m\ge2$, we may apply the above formula to get
\[
\beta_p(K,K) =  4p^{m-1}\left(1-\frac{1}{p}\right).
\]	
\end{extradetails}

The proof for $p=2$ is the same, although one needs to be careful about the term ``primitive" lattice. 
	In \cite{JuLeeOh}, it means a lattice $L$ with $\n(L)=\Z$, whereas we mean a lattice $L$ with $\mathfrak{s}(L)=\Z$, which cause some differences in the definitions of the Watson transformations.
	Precisely, using the terminology from this paper, two sublattices $\Gamma_{2,1}(S)$ and $\Gamma_{2,2}(S)$ are the only sublattices of $S$ of index $2$ and of norm $4\Z_2$, and their intersection is $\Lambda_4(S)$. 
	If we replace $n$ by $2n$, $\Lambda_p(S)$ by $\Lambda_4(S)$, and $\lambda_p(S)$ by $\lambda_4(S)$, respectively, then all the equalities until  \eqref{eqn:reduction-2} and \eqref{eqn:secondsummand} holds.
	On the other hand, by Lemma \cite[Lemma 3.1]{JuLeeOh}, \eqref{eqn:firstsummandweight} should be replaced by
	\[
		c:=\sum_{[S]\in\gen(L)/_\sim}\frac{r(T^2,S)}{o(S)} = \sum_{[S]\in\gen(L)/_\sim}\frac{\tilde{r}(S^2,T)}{o(S)}= \begin{cases}
			3 & \text{if } m=2,\\
			4 & \text{if } m\ge 3,
		\end{cases}
	\]
	so that we may obtain
	\begin{equation}\label{eqn:reductionp=2}
	r(4n,\gen(L))=c\cdot \frac{w(K)}{w(L)} \cdot r(2n,\gen(K)) - r(n,\gen(\lambda_4(L))).
	\end{equation}
	On the other hand,  one may compute by \cite[Theorem 5.6.3]{KiBook} that
	\[
	\beta_2(L,L)=2^{m+1} \quad \text{and} \quad 
	\beta_2(K,K)=\begin{cases} 3 &\text{if } m=2,\\ 2^m & \text{if }m\ge 3. \end{cases}
	\]
\begin{extradetails}
	If $p=2$, then 
	\[
	\beta_2(N,N)=2^{w-q}PE,
	\] 
	where $q,P,E$ are defined as follows: 
	With 
	\[
	q_j:=\begin{cases}
		0 & \text{if $N_j$ is even},\\
		n_j & \text{if $N_j$ is odd and $N_{j+1}$ is even},\\
		n_j+1 & \text{if $N_j$ and $N_{j+1}$ are odd},\\
	\end{cases}
	\]
	we put 
	\[
	q:=\sum_j q_j.
	\]
	We write for a unimodular lattice $M$, $M=M(e)\perp M(o)$, where $M(e)$ is even and $M(o)$ is either odd or $\{0\}$ with $\text{rank}\, M(o)\le 2$. Then we put
	\[
	P:=\prod_j P\left(\frac{1}{2}\text{rank}N_j(e)\right).
	\]
	Defining $E_j$ by 
	\[
	\begin{cases}
		\frac{1}{2}(1+\chi(N_j(e)) 2^{-\text{rank}N_j(e)/2}) & \text{if both $N_{j-1}$ and $N_{j+1}$ are even} \\
		&	\text{and unless $N_j(o)\cong \langle\varepsilon_1,\varepsilon_2\rangle$ with $\varepsilon_1\equiv \varepsilon_2 \pmod{4}$},\\
		\frac{1}{2} & \text{otherwise},
	\end{cases}
	\]
	we put 
	\[
	E:=\prod_j E_j^{-1}
	\]
	Note that $q_j$ depends on $N_j$ and $N_{j+1}$, and that
	\[
	q_j=0\quad \text{if } N_j=0.
	\]
	Moreover, $E_j$ depends on $N_{j-1}$, $N_j$ and $N_{j+1}$ and
	\[
	E_j=1 \quad \text{if } N_{j-1}=N_j=N_{j+1}=0,
	\]
	but
	\[
	E_j=\begin{cases}
		1 & \text{if either } (N_{j-1}=N_j=0 \text{ and } N_{j+1} \text{ is even}) \text{ or}\\
		&\hfill (N_{j-1} \text{ is even and }N_j=N_{j+1}=\{0\}),\\
		\frac{1}{2}& \text{if either } (N_{j-1}=N_j=0 \text{ and } N_{j+1} \text{ is odd}) \text{ or}\\
		&\hfill (N_{j-1} \text{ is odd and }N_j=N_{j+1}=\{0\}).
	\end{cases}
	\]
	
	In our case, for $L_2\cong \begin{pmatrix}
		0 & 1\\
		1 & 0
	\end{pmatrix} \perp \langle 2^m\varepsilon\rangle$ with $m\ge 1$,
	\begin{gather*}
		n_0=2, \ n_m=1, \ n_j=0 \text{ for any } j\neq 0,m \ \implies \ w=m\cdot 1\cdot \{0+(1+1)/2\} = m \\
		q_0=0, \ q_m=1, \ q_j=0 \text{ for any } j\neq 0,m \ \implies \ q=1,\\
		P=P\left(\frac{1}{2}\text{rank}N_0(e)\right)=P(1)=1-\frac{1}{2^2}=\frac{3}{4}.
	\end{gather*}
	\begin{gather*}
		E_{\leq -2}=1,\ E_{-1}=\frac{1}{2}(1+1\cdot 2^0)=1, \  E_0=\begin{cases}
			\frac{1}{2} & \text{if } m=1,\\
			\frac{3}{4} & \text{if } m\ge2,
		\end{cases}, \ 
		E_1=\begin{cases}
			\frac{1}{2}(1+1\cdot 2^0) = 1 & \text{if }m=1,\\
			\frac{1}{2} & \text{if } m=2,\\
			\frac{1}{2}(1+1\cdot 2^0) = 1 & \text{if }m\geq3,
		\end{cases}\\
		E_{m-1} = \frac{1}{2}, \ E_m = \frac{1}{2}(1+1\cdot 2^0)=1, \ E_{m+1}=\frac{1}{2}, \ E_{\geq m+2} = 1.	
	\end{gather*}
	Hence, if $m=1$, then
	\[
	E_{\leq-1}=1, \ E_0=\frac{1}{2},\ E_1=1,\ E_2=\frac{1}{2},\ E_{\geq 3}=1 \ \implies E=\pfrac{1}{4}^{-1}.
	\]
	If $m=2$, then
	\[
	E_{\leq-1}=1, \ E_0=\frac{3}{4},\ E_1=\frac{1}{2},\ E_2=1,\ E_3=\frac{1}{2},\ E_{\geq 4}=1 \ \implies E=\pfrac{3}{16}^{-1}.
	\]
	If $m\geq 3$, then
	\[
	E_{\leq-1}=1, \ E_0=\frac{3}{4},\ E_{[1,m-2]}=1,\ E_{m-1}=\frac{1}{2},\ E_m=1,\ E_{m+1}=\frac{1}{2},\ E_{\geq m+2}=1 \ \implies E=\pfrac{3}{16}^{-1}.
	\]
	Thus, for $m\ge 2$ we have
	\[
	\beta_2(L,L)=2^{w-q}PE= 2^{m-1} \cdot \frac{3}{4} \cdot \frac{16}{3} = 2^{m+1}
	\]
	On the other hand, since $K_2\cong \begin{pmatrix}
		0 & 1\\
		1 & 0
	\end{pmatrix} \perp \langle 2^{m-1}\varepsilon\rangle$ with $m\ge2$, we may apply the above variables by putting $m\mapsto m-1$ to compute 
	\[
	\beta_2(K,K)=2^{w-q}PE=2^{(m-1)-1}\cdot \frac{3}{4}\cdot \begin{cases}
		4 & \text{if } m=2,\\
		\frac{16}{3} & \text{if } m\ge3,\\
	\end{cases}
	=\begin{cases}
		3 & \text{if } m=2,\\
		2^m & \text{if } m\ge 3.
	\end{cases}
	\]
\end{extradetails}
By Theorem \ref{Minkowski-Siegel}, we have 
	\[
	\frac{w(K)}{w(L)}=\frac{(d_K)^2}{(d_L)^2}\cdot \frac{\a_2(L,L)}{\a_2(K,K)}=\frac{1}{2^2}\cdot\frac{\beta_2(L,L)}{\beta_2(K,K)}=\begin{cases}
		\frac{2}{3} & \text{if } m=2,\\
		\frac{1}{2} & \text{if } m\ge 3.
	\end{cases}
\]
Thus, we have $c\cdot\frac{w(K)}{w(L)}=2$. Plugging this into \eqref{eqn:reductionp=2} completes the proof of the theorem.\qedhere
\end{proof}

We are finally ready to prove Theorem \ref{thm:sumclassnumapproximate}. 
\begin{proof}[Proof of Theorem \ref{thm:sumclassnumapproximate}]
Again using \eqref{eqn:DukeSchulzePillot}, it remains to show that 
\begin{equation}\label{eqn:reductiongenusformula}
r(n,\gen(L))=c_L(n)\sum_{f\mid 2d_L} a_{L,f}(n)\cdot H\pfrac{4s_Ln}{f^2},
\end{equation}
where $c_L(n)$ and $a_{L,f}(n)$ are constants depending only on the residue class of $n$ modulo some positive integer $M$.
Moreover, we may assume that $L$ is primitive.
For any ternary lattice $L$, define 
$$
\mathfrak{D}_L=2^{\delta_{2\notin\mathcal{P}} \cdot \ord_2(d_L)} \cdot \prod_{p\mid d_L, \ p\,\text{odd}} p^{\ord_p(d_L)-1},
$$
and let $k_L$ denote the number of prime divisors (counting with multiplicity) of $\mathfrak{D}_L$.
Note that $k_L\ge 0$, $k_L=0$ if and only if $L$ is stable, and $p\mid \mathfrak{D}_L$ if and only if $p\mid 2d_L$ and $p \notin \mathcal{P}$.

We prove the theorem by induction on $k_L$. When $k_L=0$, that is, when $L$ is a stable lattice, we directly obtain \eqref{eqn:reductiongenusformula}  from Theorem \ref{mainthm} (see also \eqref{formula-stable}).
Precisely, if $L$ is odd, then noting that $\mathfrak{P}=2d_L$, $e_L=4$, and $s_L=d_L$, we may take 
$$
c_L(n)=12\prod_{p\mid \mathfrak{P}} \left(p+S_p^\ast(L)\right)^{-1} \quad \text{and} \quad a_{L,f}(n)=f\cdot \prod_{p\mid f} S_p^\ast(L) ,
$$
which are clearly independent of $n$.
If $L$ is even, we put $c_L(n)$ as above and take 
$$
a_{L,f}(n)=\frac{f}{2} \cdot \prod_{p\mid f} S_p^\ast(L) \text{ if $\ord_2(f)=1$ and $a_{L,f}=0$ otherwise.}
$$
Noting that $\mathfrak{P}=d_L/2$, $e_L=1$, $s_L=d_L$, and $S_2^\ast(L)=1$, one may easily observe that the right-hand side of \eqref{eqn:reductiongenusformula} with these constants is equal to that of \eqref{formula-stable}.

Now, let $k\in\N$ be given and assume that \eqref{eqn:reductiongenusformula} is valid for any ternary lattice $L$ with $k_L<k$.
Let $L$ be a ternary lattice with $k_L=k$ and let $\n(L)=N\Z$ for $N\in\{1,2\}$. We split the argument depending on the divisibility of $n$ by primes dividing $\mathfrak{D}_{L}$. 

First suppose that $n\in\N$ satisfies $n\in N\Z_p^\times$ for any $p\mid \mathfrak{D}_L$. Then $n\in\N\Z_p^{\times}$ for any  $p\mid 2d_L$ with $p\notin \mathcal{P}$, and we may apply Theorem \ref{formula-n,P-coprime}.  Noting that $\frac{2}{\sqrt{e}}\mathfrak{P} \mid 2d_L$, where $e$ is the constant defined in Theorem \ref{formula-n,P-coprime}, Theorem \ref{formula-n,P-coprime} and Remark \ref{rmk-n,P-coprime} implies that there are $m_1\in\N$ and constants $c^{(1)}_L(n)$ and $a^{(1)}_{L,f}(n)$ for any $f\mid 2d_L$ depending only on $n$ modulo $m_1$ satisfying \eqref{eqn:reductiongenusformula} with $c_L(n)=c^{(1)}_L(n)$ and $a_{L,f}(n)=a^{(1)}_{L,f}(n)$ for any $n$ such that $n\in N\Z_p^\times$.

Now, we fix a prime $p\mid \mathfrak{D}_L$ and consider the case when $n\in pN\Z_p$. We claim that there are $m_p\in\N$ and constants $c^{(p)}_L(n)$ and $a^{(p)}_{L,f}(n)$ for any $f\mid 2d_L$ dependent only on the residue class of $n$ modulo $m_p$ satisfying \eqref{eqn:reductiongenusformula} with $c_L(n)=c^{(p)}_L(n)$ and $a_{L,f}(n)=a^{(p)}_{L,f}(n)$ for any $n$ such that $n\in pN\Z_p$.
The proof of the claim will be done case-by-case, relying on reduction formulas in Theorem \ref{reduction-formulas}.
Since the proofs are similar for each cases, we only provide the proof for the case when $p$ is an odd prime.

Assume that $p$ is an odd prime. Since $p\mid \mathfrak{D}_L$, $L$ is not stable at $p$. Hence by Theorem \ref{reduction-formulas}
\begin{equation}\label{eqn:reductongenusformula-odd-prime}
r(n,\gen(L))=
\begin{cases}
	r(n/p,\gen(L')) & \text{if $M_p$ is anisotropic}, \mathfrak{s}(U_p)=p\Z_p,\\
	r(n/p^2,\gen(L')) & \text{if $M_p$ is anisotropic}, \mathfrak{s}(U_p)\subseteq p^2\Z_p,\\
	r(n/p,\gen(K))-r(n/p^2, \gen(L')) & \text{if $M_p$ is isotropic},
\end{cases}
\end{equation}
where $L'$ is the lattice $\lambda_p(L)$ obtained by  Watson transformation, $K$ is a lattice defined in Theorem \ref{reduction-formulas} \ref{reduction-formulas:3}, and $L_p=M_p\perp U_p$ is a Jordan decomposition of $L_p$ with the leading Jordan component $M_p$. Note that $d_{L'}=d_L/p$ in the first case, $d_{L'}=d_L/p^2$ or $d_L/p^4$ in the second case, and $d_{L'}=d_L/p^2$ and $d_K=d_L/p$ in the third case.
Moreover, $s_{L'}=s_L/p$ or $ps_L$ in the first case, $s_{L'}=s_L$ in the other cases, and $s_{K}=s_L/p$ or $ps_L$. In particular, $k_{L'}, k_K<k_L=k$ for every term on the right-hand side. By the induction hypothesis, we obtain a formula of the form \eqref{eqn:reductiongenusformula} for each term on the right-hand side of \eqref{eqn:reductongenusformula-odd-prime}.

As an example, for the first case with $s_{L'}=s_L/p$, there are $m_p'\in \N$ and constants $c_{L'}(n')$ and $a_{L',f'}(n')$ for any $f'\mid 2d_{L'}=2d_L/p$ depending only on the residue class $n'$ modulo $m_p'$ such that
$$
\begin{aligned}
r(n/p,\gen(L'))&=c_{L'}(n/p)\sum_{f'\mid 2d_{L'}} a_{L',f'}(n/p)\cdot H\pfrac{4s_{L'}(n/p)}{f'^2}\\
&=c_{L'}(n/p)\sum_{(pf')\mid 2d_L} a_{L',f'}(n/p)\cdot H\pfrac{4s_{L}n}{(pf')^2}.	
\end{aligned}
$$
Put $c^{(p)}_L(n)=c_{L'}(n/p)$, and for any $f\mid 2d_L$ put $a^{(p)}_{L,f}(n)=a_{L',f/p}(n/p)$ if $p\mid f$ and $a^{(p)}_{L,f}(n)=0$ otherwise.
Then these constants $c^{(p)}_L(n)$ and $a^{(p)}_{L,f}(n)$ are dependent only on the residue class of $n$ modulo $m_p=p\cdot m_p'$, and one may observe that \eqref{eqn:reductiongenusformula}  holds for any $n$ such that $n\in pN\Z_p$ with $c_L(n)=c^{(p)}_L(n)$ and $a_{L,f}(n)=a^{(p)}_{L,f}(n)$. 
Other cases may also be deduced in the same manner, and we leave the details for the remaining cases to the reader.

Finally, let $\mathcal{S}=\{1\}\cup \{ \text{prime divisors of }\mathfrak{D}_L\}=\{p_0=1,p_1,\ldots,p_h\}$, and for any $p\in\mathcal{S}$, let $m_p\in\N$, $c^{(p)}_L(n)$ and $a^{(p)}_{L,f}(n)$ for any $f\mid 2d_L$ be constants defined above.
Let $M$ be the least common multiple of $m_1$ and $m_p$ for all $p\mid \mathfrak{D}_L$.
Clearly, constants $c^{(p)}_L(n)$ and $a^{(p)}_{L,f}(n)$ are dependent only on the residue class of $n$ modulo $M$.
Define constants $c_L(n)$ and $a_{L,f}(n)$ recursively as follows:
\begin{enumerate}
	\item for any $n$ with $n\in N\Z_p^\times$ for all $p\mid \mathfrak{D}_L$, define $c_L(n)=c^{(1)}_L(n)$, $a_{L,f}(n)=a^{(1)}_{L,f}(n)$, 
	\item for $1\le i \le h$, and for any $n$ with $n\in p_iN\Z_{p_i}$ and $n\notin p_jN\Z_{p_j}$ for all $1\le j \le i-1$, define $c_L(n)=c^{(p_i)}_L(n)$, $a_{L,f}(n)=a^{(p_i)}_{L,f}(n)$.
\end{enumerate}
Then by the construction of the constants, $c_L(n)$ and $a_{L,f}(n)$ satisfies \eqref{eqn:reductiongenusformula}. This completes the proof of the theorem. 
\end{proof}
\begin{remark}
	In practice, one would better find the constants $c_L(n)$ and $a_{L,f}(n)$ depending on the condition $n\in \left(\prod_{q\mid \mathfrak{D}_L}q^{e_q}\right)N\Z_p^\times$ $(N\Z=\n(L))$ for all primes $p\mid \mathfrak{D}_L$ for some powers $e_q\in\N$. One may use the reduction formulas in Theorem \ref{reduction-formulas} recursively until one can apply Theorem \ref{mainthm} or Theorem \ref{formula-n,P-coprime}.
	This should always work by the ``in particular"-part of Theorem \ref{reduction-formulas}.
	The existence of $M\in\N$ also follows from this since the formula for $r(n,\gen(L))$ will eventually be reduced to certain linear combination of formulas for stable lattices for any $n\in \left(\prod_{q\mid \mathfrak{D}_L}q^{e_q}\right)\N$ for sufficiently large $e_q\in \N$.
\end{remark}
\section{Examples}\label{sec:examples}
In this section, we provide several examples of how to use Theorems \ref{mainthm}, \ref{formula-n,P-coprime}, and \ref{reduction-formulas} to obtain a formula for $r(n,\gen(L))$.
\subsection{Example with stable lattices}

\begin{example}
We first give an example for a stable lattice. Let 
$$
L=\la 1,3,5 \ra.
$$
Clearly, $L$ is a stable lattice, hence we may apply Theorem \ref{mainthm}. Note that $d_L=15$,
$$
\mathcal{P}_L=\{2,3,5\} \quad \text{and} \quad \mathfrak{P}_L=2^{-0}\cdot 2\cdot 3\cdot 5 = 30.
$$
Also, $e=4$ since $L$ is odd, and for each $p\mid \mathfrak{P}$, we have
$$
S_2^\ast(L)=1,\quad S_3^\ast(L)=1,\quad \text{and} \quad S_5^\ast(L)=-1.
$$
Therefore, by Theorem \ref{mainthm}, we have
$$
\begin{array}{rl}
	r(n,\gen(L))&=12\cdot \frac{1}{2+1}\cdot\frac{1}{3+1}\cdot\frac{1}{5-1} \times \sum_{f\mid 30} \left[ \prod_{p\mid f}S_p^\ast(L)\right]\cdot f\cdot H\pfrac{60n}{f^2}\\ [5pt]
	&=\frac{1}{4} \big(H(60n)+2H(15n)+3H\pfrac{20n}{3}+6H\pfrac{5n}{3} \\ [5pt]
	&\hfill  -5H\pfrac{12n}{5}-10H\pfrac{3n}{5}-15H\pfrac{4n}{5}-30H\pfrac{n}{15} \big).
\end{array}
$$	
\end{example}

\begin{example}
	We provide another example for a stable lattice, which is even and non-diagonal. Let 
	$$
	L=\begin{pmatrix}
		2&1&0\\
		1&2&1\\
		0&1&4
	\end{pmatrix}.
	$$
	Note that $d_L=10$ and $L$ is even, and hence it is a stable lattice. Hence $e=1$,
	$$
	\mathcal{P}_L=\{2,5\} \quad \text{and} \quad \mathfrak{P}_L=2^{-1}\cdot 2 \cdot 5 = 5.
	$$
	Moreover, $L_5 \cong \la 1,2, 10\ra$ over $\Z_5$, hence $S_5^\ast(L)=-1$.
	Therefore, by Theorem \ref{mainthm}, we have
	\begin{center}
	$r(n,\gen(L))=12\cdot\frac{1}{5-1} \times \sum_{f\mid 5} \left[ \prod_{p\mid f}S_p^\ast(L)\right]\cdot f\cdot H\pfrac{10n}{f^2}=3 \left( H(10n)-5H\pfrac{2n}{5} \right)$.		
	\end{center}
\end{example}
\subsection{An example of formulas for arbitrary lattices}
We split Theorem \ref{formula-n,P-coprime} into further cases for ease of use in our examples. Namely, Theorem \ref{formula-n,P-coprime} may be written as 
	\begin{equation*}
		\begin{aligned}
			r(&n,\gen(L))=\prod_{\substack{p\mid 2d_L \\ p\not\in \mathcal{P}}}\a_p(n,L)\left(1-\legendre{d}{p}\frac{1}{p}\right) \left(1-\frac{1}{p^2}\right)^{-1} \\
			&\times
			\begin{cases}
				\frac{24(s,n,2)}{st} \left[\prod_{p\in\mathcal{P}} \frac{p}{p+S_p^\ast(L)}\right] \cdot {\displaystyle\sum_{f\mid \mathfrak{P}}} \left[\prod_{p\mid f}S_p^\ast(L)\right]\cdot f \cdot H\pfrac{sn}{((s,n,2)f)^2}  & \text{if $2\not\in\mathcal{P}$, $\frac{sn}{(s,n)^2}\equiv 3 \pmod{4}$},\\
				
				\frac{12(s,n,2)}{st} \left[\prod_{p\in\mathcal{P}} \frac{p}{p+S_p^\ast(L)}\right] \cdot {\displaystyle\sum_{f\mid \mathfrak{P}}} \left[ \prod_{p\mid f}S_p^\ast(L)\right] \cdot f \cdot H\pfrac{4sn}{((s,n,2)f)^2}  & \text{if $2\not\in\mathcal{P}$, $\frac{sn}{(s,n)^2}\not\equiv 3 \pmod{4}$},\\
				
				\frac{12}{st} \left[\prod_{p\in\mathcal{P}} \frac{p}{p+S_p^\ast(L)}\right] \cdot {\displaystyle\sum_{f\mid \mathfrak{P}}} \left[ \prod_{p\mid f}S_p^\ast(L)\right]\cdot f \cdot H\pfrac{4sn}{f^2}  & \text{if $2\in\mathcal{P}$, $\ord_2(d_L)=0$},\\
				
				\frac{48}{st} \left[\prod_{p\in\mathcal{P}\setminus \{2\}} \frac{p}{p+S_p^\ast(L)}\right] \cdot {\displaystyle\sum_{f\mid \mathfrak{P}/2}} \left[ \prod_{p\mid f}S_p^\ast(L)\right]\cdot f \cdot H\pfrac{sn}{f^2}  & \text{if $2\in \mathcal{P}$, $\ord_2(d_L)=1$.}
			\end{cases}
		\end{aligned}
	\end{equation*}

\begin{example}\label{ex:arbitrary}
 Let $n$ be a positive integer and let
$$
L=\la 1,1,75 \ra.
$$
Note that $4d_L=300$, $L$ is odd with $\ord_2(d_L)=0$, and $L$ is stable at prime numbers $2$ and $3$, hence
$$
s=3, \quad t=10, \quad \mathcal{P}_L=\{2,3\}, \quad \mathfrak{P}_L=2^{-0}\cdot 2\cdot 3 = 6, \quad \text{and} \quad (\epsilon,e)=(12,4).
$$
Moreover, $S_2^\ast(L)=1$ and $S_3^\ast(L)=-1$.

Assume that $(n,5)=1$. Noting that $\legendre{d}{5}=\legendre{-sn}{5}=-\legendre{n}{5}$ and by Remark \ref{rmk-n,P-coprime}, we have
$$
c_5(n,L)=\left(1-\frac{1}{5}\right)\cdot \left(1+\legendre{n}{5}\frac{1}{5}\right)\cdot \left(1-\frac{1}{5^2}\right)^{-1}
=\begin{cases}
	1 & \text{if } n\equiv \pm 1 \pmod{5},\\
	\frac{2}{3} & \text{if } n\equiv \pm 2 \pmod{5}.
\end{cases}
$$
Thus, by Theorem \ref{formula-n,P-coprime}, we have
\begin{equation}\label{example-1,1,75-eq0}
\begin{aligned}
	r(n,\gen(L))&=\frac{12}{3\cdot 10}\cdot \frac{2}{2+1}\cdot \frac{3}{3-1}\cdot c_5(n,L)\times {\sum_{f\mid 6}} \left[ \prod_{p\mid f}S_p^\ast(L)\right]\cdot f \cdot H\pfrac{12n}{f^2}\\
	&=c_L(n)\left( H(12n)+2H(3n)-3H\pfrac{4n}{3}-6H\pfrac{n}{3}\right),
\end{aligned}
\end{equation}
where $c_L(n)=\frac{2}{5}$ if $n\equiv \pm 1 \pmod{5}$, and $c_L(n)=\frac{4}{15}$ if $n\equiv \pm 2 \pmod{5}$.

Now assume that $n\equiv 0 \pmod{5}$ and write $n=5n_0$ for some $n_0\in\N$. Note that 
\begin{center}
	$L_5\cong \la 1,-1, 3\cdot 5^2\ra \cong \left(\begin{smallmatrix} 0&1\\1&0 \end{smallmatrix}\right)\perp \la 3\cdot 5^2 \ra$ over $\Z_5$.
\end{center}
By Theorem \ref{reduction-formulas} \ref{reduction-formulas:3}, 
\begin{equation}\label{example-1,1,75-eq1}
r(n,\gen(L))=r(5n_0,\gen(L))=2r(n_0,\gen(K))-r(n_0/5,\gen(\lambda_5(L))).
\end{equation}
Here, $\lambda_5(L)=\la 1,1,3 \ra$ and $K$ is a lattice with $d_K=15$ whose genus is characterized by
\begin{center}
$K_5\cong \left(\begin{smallmatrix} 0&1\\1&0 \end{smallmatrix}\right)\perp \la 3\cdot 5 \ra$ over $\Z_5$ and $K_p\cong L_p$ for any $p\neq 5$.
\end{center}
Indeed, $K$ may explicitly be chosen as $\la 1,1,15\ra$, which is a lattice obtained by scaling one of two sublattices of $L$ of index $5$ and of norm ideal $5\Z$ by $\frac{1}{5}$.
Note that both $K$ and $\lambda_5(L)$ are stable lattices. Hence by Theorem \ref{mainthm}, we may have
$$\begin{array}{rr}
\multicolumn{2}{l}{r(n_0,\gen(K))=\frac{1}{3} \big( H(60n_0)+2H(15n_0)-3H\pfrac{20n_0}{3}-6H\pfrac{5n_0}{3}}\\ [5pt]
&+ 5H\pfrac{12n_0}{5}+10H\pfrac{3n_0}{5}-15H\pfrac{4n_0}{15}-30H\pfrac{n_0}{15}\big)\\ [5pt]
\multicolumn{2}{c}{=\frac{1}{3} \left( H(12n)+2H(3n)-3H\pfrac{4n}{3}-6H\pfrac{n}{3} + 5H\pfrac{12n}{25}+10H\pfrac{3n}{25}-15H\pfrac{4n}{75}-30H\pfrac{n}{75}\right),}
\end{array}
$$
and
$$
\begin{array}{rl}
r(n_0/5,\gen(\lambda_5(L)))&=2 \left( H(12n_0/5)+2H(3n_0)-3H\pfrac{4n_0/5}{3}-6H\pfrac{n_0/5}{3} \right)\\ [5pt]
&=2 \left( H\pfrac{12n}{25}+2H\pfrac{3n}{25}-3H\pfrac{4n}{75}-6H\pfrac{n}{75}\right).
\end{array}
$$
Plugging the two formulas above into \eqref{example-1,1,75-eq1}, we obtain a formula for $n\equiv 0\pmod{5}$.
Combining this together with \eqref{example-1,1,75-eq0}, we obtain overall that 
$$
\begin{array}{lllr}
\multicolumn{3}{l}{r(n,\gen(L))=c_L(n)\cdot\big( H(12n)+2H(3n)-3H\pfrac{4n}{3}-6H\pfrac{n}{3}}&\hspace{4cm}\\ [5pt]
&&\multicolumn{2}{r}{+ 2H\pfrac{12n}{25}+4H\pfrac{3n}{25}-6H\pfrac{4n}{75}-12H\pfrac{n}{75}\big),}
\end{array}
$$
where $c_L(n)=\frac{2}{5}$, $\frac{4}{15}$, or $\frac{2}{3}$ if $n\equiv \pm 1 \pmod{5}$, $n\equiv \pm 2 \pmod{5}$, or $n\equiv 0 \pmod{5}$, respectively.
\end{example}

\end{document}